\documentclass[reqno]{amsart}
\usepackage{amsmath}
\usepackage{amsfonts}
\usepackage{amssymb}
\usepackage{enumerate}
\usepackage{amstext}
\usepackage{amsbsy}
\usepackage{amsopn}
\usepackage{bbm,amsthm}
\usepackage{amscd}
\usepackage[pdftex]{color}
\usepackage{amsxtra}
\usepackage{upref}
\usepackage{graphicx}
\usepackage{epstopdf}
\DeclareFontFamily{OML}{rsfs}{\skewchar\font'177}
\DeclareFontShape{OML}{rsfs}{m}{n}{ <5> <6> rsfs5 <7> <8> <9> rsfs7
  <10> <10.95> <12> <14.4> <17.28> <20.74> <24.88> rsfs10 }{}
\DeclareMathAlphabet{\mathfs}{OML}{rsfs}{m}{n}

\newcommand{\x}{\ensuremath{\underline{x}}}
\newtheorem{thm}{Theorem}[section]
\newtheorem{lem}[thm]{Lemma}
\newtheorem{prop}[thm]{Proposition}

\newtheorem{cor}[thm]{Corollary}

\newtheorem*{theorem*}{Theorem}

\newtheorem*{example*}{Example}

\numberwithin{equation}{section}

\newcommand{\del}{\partial}

\renewcommand{\epsilon}{\varepsilon}
\def\wt{\widetilde}

\def\text#1{\textrm{#1}}

\def\emptyset{\varnothing}

\def\e{\epsilon}

\def\vf{\varphi}
\def\a{\alpha}

\def\d{\delta}

\def\g{\gamma}

\def\L{\Lambda}

\def\s{\sigma}

\def\x{\times}
\def \R{\mathbb R}
\def \N{{\mathbb N}}

\def \Z{\mathbb Z}

\def\ov{\overline}
\def\un{\underline}

\def\Q{\mathbb Q}
\def\wh{\widehat}
\def\({\biggl(}
\def\){\biggr)}

\def\<{\mathbf\langle}
\def\>{\mathbf\rangle}

\DeclareMathOperator\const{const}
\DeclareMathOperator\diam{diam}
\DeclareMathOperator\dist{dist}
\DeclareMathOperator\dom{dom}

\DeclareMathOperator\Lip{Lip}
\DeclareMathOperator\Span{span}
\DeclareMathOperator\NUH{NUH}

\DeclareMathOperator\Hol{H\ddot{o}l}
\DeclareMathOperator\id{Id}

\DeclareMathOperator\Exp{Exp}

\def\epto{\xrightarrow{\epsilon}}

\title[Symbolic dynamics for  flows]{Symbolic dynamics for three dimensional flows with positive topological entropy}
\author{Yuri Lima and Omri M. Sarig }\thanks{O.S. was partially supported by  ERC award  ERC-2009-StG n$^\circ$ 239885.}
\date{October 29, 2016}
\address{Laboratoire de Math\'ematiques d'Orsay, Universit\'e Paris-Sud\\ CNRS, Universit\'e
Paris-Saclay, $91405$ Orsay, France}
\email{yurilima@gmail.com}
\address{Faculty of Mathematics and Computer Science\\ The Weizmann Institute of Science\\ POB 26, Rehovot, Israel}
\email{omsarig@gmail.com}

\begin{document}
\maketitle
\begin{abstract}
We construct symbolic dynamics on sets of full measure (with respect to an ergodic measure of positive entropy)
for $C^{1+\epsilon}$ flows  on closed smooth three dimensional manifolds.
One consequence is that the geodesic flow on the unit tangent bundle of a closed $C^\infty$
surface has at least $\const\times(e^{hT}/T)$ simple closed orbits of period less than $T$,
whenever the topological entropy $h$ is positive -- and without further assumptions  on the curvature.
\end{abstract}

\section{Introduction}

The  aim of this paper  is to develop symbolic dynamics for smooth flows  with topological entropy $h>0$
on three dimensional closed (compact and boundaryless) Riemannian manifolds.

Earlier works treated geodesic flows on hyperbolic surfaces
\cite{Series-Symbolic-Acta,Series-ICM,Katok-Ugarcovici-Symbolic},
geodesic flows on surfaces with variable negative curvature \cite{Ratner-MP-three-dimensions},
and uniformly hyperbolic  flows in any dimension \cite{Ratner-MP-n-dimensions,Bowen-Symbolic-Flows}.
This work only assumes that $h>0$ and that the flow has {\em positive speed} (i.e. the vector that generates
the flow has no zeroes).
This  generality allows us to cover several cases of interest that could not be treated before, for example:
\begin{enumerate}[(1)]
\item {\em Geodesic flows with positive entropy in positive curvature\/:} There are  many Riemannian metrics with positive curvature somewhere (even everywhere) whose geodesic flow has positive topological  entropy \cite{Donnay,Burns-Gerber,Knieper-Weiss-Generic-Positive-entropy,Contreras-Paternain}.
\item {\em Reeb flows with positive entropy\/:} These arise from Hamiltonian flows on surfaces of constant energy, see \cite{Hutchings}.  Examples  with positive topological entropy are given in \cite{Macarini-2011}. (This application was suggested to us  by G. Forni.)
\item {\em Abstract non-uniformly hyperbolic flows in three dimensions}, see
\cite{Barreira-Pesin-Non-Uniform-Hyperbolicity-Book,Pesin-Izvestia-1976}.
\end{enumerate}


The statement of our main result is somewhat technical,  therefore we begin with a down-to-earth corollary.
Let $\vf$ be a flow. A {\em simple closed orbit} of {\em length} $\ell$ is  a parameterized curve $\gamma(t)=\vf^t(p)$,
$0\leq t\leq\ell$ s.t. $\gamma(0)=\gamma(\ell)$ and $\gamma(0)\neq \gamma(t)$ when $0<t<\ell$.
The {\em trace} of $\gamma$ is defined to be the set $\{\gamma(t):0\leq t\leq\ell\}$.  Let $[\gamma]$ denote the
equivalence class of the relation  $\gamma_1\sim\gamma_2\Leftrightarrow\gamma_1,\gamma_2$
have equal lengths and traces. Let
$
\pi(T):=\#\{[\gamma]:\ell(\gamma)\leq T,\gamma\textrm{ is simple}\}.
$

\begin{thm}\label{ThmClosedOrbits}
Suppose $\vf$ is a $C^\infty$ flow with positive speed  on a $C^\infty$ closed three dimensional manifold.
If $\vf$ has positive topological entropy $h$, then there is a positive constant $C$ s.t.
$\pi(T)\geq C{e^{hT}}/{T}$ for all $T$ large enough.
\end{thm}

The theorem strengthens  Katok's   bound $\liminf_{T\to\infty}\frac{1}{T}\log\pi(T)\geq h$,
see \cite{KatokIHES,Katok-Closed-Geodesics}. It  extends to flows of lesser regularity,
under the additional assumption that they possess a measure of maximal entropy
(Theorem \ref{ThmClosedOrbits(General)}).  The lower bound $C{e^{ht}}/{T}$ is
sharp in many special cases
\cite{Huber-Closed-Geodesics,Margulis-Closed-Orbits,Parry-Pollicott-PNT,Knieper-Closed-Orbits},
but not in the general setup of this paper. For more on this, see \S\ref{SectionClosedOrbits}.

We obtain Theorem \ref{ThmClosedOrbits} by constructing a symbolic model that is a finite-to-one extension
of $\vf$. The orbits of this model are easier to understand than those of the original flow.
This technique, called  {\em ``symbolic dynamics"},
can be traced back to the work of Hadamard, Morse, Artin, and Hedlund.

We proceed to describe the symbolic models used in this work.
Let $\mathfs G$ be a directed graph with a countable set of vertices $V$.
We write $v\to w$ if there is an edge from $v$ to $w$, and we assume throughout that
for every $v$  there are $u,w$ s.t. $u\to v, v\to w$.

\medskip
\noindent
{\sc Topological Markov shifts:}  The {\em topological Markov shift} associated to $\mathfs G$
is the discrete-time topological dynamical system $\sigma:\Sigma\to\Sigma$ where
$$
\Sigma=\Sigma(\mathfs G):=\{\textrm{paths on }\mathfs G\}=\{\{v_i\}_{i\in\Z}:v_i\to v_{i+1}\textrm{ for all }i\in\Z\},
$$
equipped with  the metric $d(\un{v},\un{w}):=\exp[-\min\{|n|:v_n\neq w_n\}]$,
and $\sigma:\Sigma\to\Sigma$ is  the {\em left shift map}, $\sigma:\{v_i\}_{i\in\Z}\mapsto \{v_{i+1}\}_{i\in\Z}$.

\medskip
\noindent
{\sc Birkhoff cocycle:} Suppose $r:\Sigma\to\R$ is a function. The {\em Birkhoff sums} of $r$
are $r_n:=r+r\circ\sigma+\cdots+r\circ\sigma^{n-1}$ $(n\geq 1)$. There is a unique way to extend
the definition to $n\leq 0$ in such a way that the {\em cocycle identity} $r_{m+n}=r_n+r_m\circ\sigma^n$
holds for all $m,n\in\Z$: $r_0:=0$ and $r_{n}:=-r_{|n|}\circ\sigma^{-|n|}$ $(n<0)$.\label{TopMarkovFlowDefiPage}

\medskip
\noindent
{\sc Topological Markov flow:}\label{TopMarkovFlowDefiPage} Suppose $r:\Sigma\to\R^+$
is H\"older continuous and bounded away from zero and infinity.
The {\em topological Markov flow} with {\em roof function} $r$ and {\em base map}
$\sigma:\Sigma\to\Sigma$ is the flow $\sigma_r:\Sigma_r\to\Sigma_r$ where
$$\Sigma_r:=\{(\un{v},t):\un{v}\in\Sigma,0\leq t<r(\un{v})\},\
\sigma_r^\tau(\un{v},t)=(\sigma^n(\un{v}),t+\tau-r_n(\un{v}))
$$
for the unique $n\in\Z$ s.t. $0\leq t+\tau-r_n(\un{v})<r(\sigma^{n}(\un{v}))$.

\medskip
Informally,  $\sigma_r$ increases the $t$ coordinate at unit speed subject to the identifications
$(\un{v}, r(\un{v}))\sim (\sigma(\un{v}),0)$. The cocycle identity guarantees that
$\sigma_r^{\tau_1+\tau_2}=\sigma_r^{\tau_1}\circ\sigma_r^{\tau_2}$.
There is a natural metric $d_r(\cdot,\cdot)$ on $\Sigma_r$, called the {\em Bowen-Walters metric},  s.t. $\sigma_r$
is a continuous flow \cite{Bowen-Walters-Metric}. Moreover, $\exists C>0, 0<\kappa<1$ s.t. $d_r(\sigma_r^\tau(\omega_1),\sigma_r^\tau(\omega_2))\leq C d_r(\omega_1,\omega_2)^\kappa$ for all $|\tau|<1$ and every $\omega_1,\omega_2\in\Sigma_r$ (Lemma \ref{Lemma-BW}).


\medskip
\noindent
{\sc Regular parts:\/}
The {\em regular part} of $\Sigma$ is the set
$$
\Sigma^\#:=\left\{\un v\in\Sigma:\exists v,w\in V\text{ s.t. }\begin{array}{l}v_n=v\text{ for infinitely many }n>0\\
v_n=w\text{ for infinitely many }n<0
\end{array}\right\},
$$
and the {\em regular part} of $\Sigma_r$ is $\Sigma_r^\#:=\{(\un{v},t)\in\Sigma_r: \un{v}\in\Sigma^\#\}$.

\medskip
By the Poincar\'e recurrence theorem, $\Sigma_r^\#$ has full measure with respect to
any $\sigma_r$--invariant probability measure, and it  contains all the closed orbits of $\sigma_r$.
We now state our main result. Let $M$ be a three dimensional closed $C^\infty$ Riemannian manifold,
let $X$ be a $C^{1+\beta}$ ($0<\beta<1$) vector field on $M$ s.t. $X_p\neq 0$ for all $p$,
let $\vf:M\to M$ be the flow  determined by $X$, and let $\mu$ be a $\vf$--invariant Borel probability measure.

\begin{thm}\label{ThmSymbolicDynamicsErgodic} If $\mu$ is ergodic and its Kolmogorov-Sina{\u\i}
entropy is positive, then there is a topological Markov flow $\sigma_r:\Sigma_r\to\Sigma_r$
and a map  $\pi_r:\Sigma_r\to M$ s.t.:
\begin{enumerate}[$(1)$]
\item $r:\Sigma\to\R^+$ is H\"older continuous and bounded away from zero and infinity.
\item $\pi_r$ is H\"older continuous with respect to the Bowen-Walters metric (see \S\ref{section-coding-flow}).
\item $\pi_r\circ\sigma^t_r=\vf^t\circ\pi_r$ for all $t\in\R$.
\item $\pi_r[\Sigma^\#_r]$ has full measure with respect to $\mu$.
\item If $p=\pi_r(\un{x},t)$ where $x_i=v$ for infinitely many $i<0$ and $x_i=w$ for infinitely many $i>0$,
then $\#\{(\un{y},s)\in\Sigma^\#_r:\pi_r(\un{y},s)=p\}\leq N(v,w)<\infty$.
\item $\exists N=N(\mu)<\infty$ s.t.  $\mu$--a.e. $p\in M$ has exactly $N$ pre-images in $\Sigma_r^\#$.
\end{enumerate}
\end{thm}

Some of the applications we have in mind require a version of this result for non-ergodic measures.
To state it, we need to recall some facts from smooth ergodic theory \cite{Barreira-Pesin-Non-Uniform-Hyperbolicity-Book}.
Let $T_p M$ be the tangent space at $p$ and let $(d\vf^t)_p:T_p M\to T_{\vf^t(p)} M$
be the differential of $\vf^t$ at $p$. Suppose $\mu$ is a $\vf$--invariant Borel probability measure
on $M$ (not necessarily ergodic). By the Oseledets Theorem,  for $\mu$--a.e. $p\in M$,
for every  $0\neq\vec{v}\in T_p M$, the  limit
$
\chi(p,\vec{v}):=\lim_{t\to\infty}\frac{1}{t}\log\|(d\vf^t)_p \vec{v}\|_{\vf^t(p)}
$ exists.
The values of $\chi(p,\cdot)$ are called the {\em Lyapunov exponents} at $p$.
If $\dim(M)=3$, then  there are at most three distinct such values.
At least one of them, $\chi(p,X_p)$, equals zero.

\medskip
\noindent
{\sc Hyperbolic measures:} Suppose $\chi_0>0$.
A {\em $\chi_0$--hyperbolic} measure is an invariant measure $\mu$ s.t. $\mu$--a.e. $p\in M$  has one  Lyapunov exponent in $(-\infty,-\chi_0)$, one Lyapunov exponent in $(\chi_0,\infty)$ and one Lyapunov exponent equal to zero.

\medskip
In dimension three, every {\em ergodic} invariant measure with positive metric entropy is $\chi_0$--hyperbolic
for any $0<\chi_0<h_\mu(\vf)$, by the Ruelle inequality \cite{Ruelle-Entropy-Inequality}.
But some hyperbolic measures, e.g. those carried by hyperbolic closed orbits,  have zero entropy.

\begin{thm}\label{ThmSymbolicDynamics} Suppose $\mu$ is a $\chi_0$--hyperbolic
invariant probability measure for some  $\chi_0>0$. Then there is a topological Markov
flow $\sigma_r:\Sigma_r\to\Sigma_r$ and  a map $\pi_r:\Sigma_r\to M$ satisfying {\em (1)--(5)}
in Theorem \ref{ThmSymbolicDynamicsErgodic}. If $\mu$ is ergodic, then {\em (6)} holds as well.
\end{thm}

Results in this spirit were first proved by Ratner and Bowen for Anosov flows and Axiom A flows
in any dimension \cite{Ratner-MP-three-dimensions,Ratner-MP-n-dimensions,Bowen-Symbolic-Flows}, using the technique of {\em Markov partitions} introduced by Adler \& Weiss and Sina{\u \i} for discrete-time dynamical systems \cite{Adler-Weiss-PNAS,Adler-Weiss-Similarity-Toral-Automorphisms,Sinai-Construction-of-MP,Sinai-MP-U-diffeomorphisms}.\footnote{For geodesic flows on hyperbolic surfaces, alternative geometric and number theoretic methods are possible,
see \cite{Series-Symbolic-Acta,Series-Green-Book,Katok-Ugarcovici-Symbolic} and references therein.
These methods are more restrictive than those of Ratner and Bowen, but they make the coding procedure more transparent.}

In 1975 Bowen gave a new construction of Markov partitions for Axiom A diffeomorphisms,
using shadowing techniques \cite{Bowen-LNM,Bowen-Regional-Conference}.
The second author extended these techniques to general $C^{1+\beta}$ surface diffeomorphisms with
positive topological entropy \cite{Sarig-JAMS}. Our strategy is to apply these methods to a suitable
Poincar\'e section  for the flow. The main difficulty is that \cite{Sarig-JAMS} deals with diffeomorphisms,
while Poincar\'e sections are discontinuous.

In part 1 of the paper, we construct a  Poincar\'e section $\Lambda$ with the following property:
If $f:\Lambda\to\Lambda$ is the Poincar\'e return map and $\mathfrak S\subset\Lambda$
is the set of discontinuities of $f$, then
$\liminf_{|n|\to\infty}\frac{1}{n}\log\dist_\Lambda(f^n(p),\mathfrak S)=0$ a.e. in $\Lambda$.
This places us in the context of ``non-uniformly hyperbolic maps with singularities" studied in \cite{Katok-Strelcyn}.

In part 2 we explain why the  methods of \cite{Sarig-JAMS} apply to $f:\Lambda\to\Lambda$
despite its discontinuities. The result is a countable Markov partition for $f:\Lambda\to\Lambda$,
which leads to a coding of  $f$ as a topological Markov shift, and a coding of  $\vf:M\to M$
as a topological Markov flow.

In part 3, we provide two applications:  Theorem \ref{ThmClosedOrbits} on the growth of the
number of closed orbits, and a  result saying that the set of measures of maximal entropy is
finite or countable.  The proof of Theorem \ref{ThmClosedOrbits} uses a mixing/constant
suspension dichotomy for topological Markov flows, in the spirit of  \cite{Plante-Anosov-Flows-Dichotomy}.


\subsection*{Standing assumptions} Let $M$ be a three dimensional closed $C^\infty$ Riemannian manifold,
with tangent bundle $TM=\bigcup_{p\in M} T_p M$, Riemannian metric  $\<\cdot,\cdot\>_p$, norm $\|\cdot\|_p$, and exponential map $\exp_p$ (this is different from the  $\Exp_p$ in \S \ref{SectionPesinCharts}).

Given $Y\subset M$,  $\dist_Y(y_1,y_2):=\inf\{\textrm{lengths of rectifiable curves in $Y$ from $y_1$}$
$\textrm{to $y_2$}\}$, where $\inf\emptyset:=\infty$. Given two metric spaces $(A,d_A),(B,d_B)$
and a map $F:A\to B$, $\mathrm{\Hol}_\alpha(F):=\sup_{x\neq y}\frac{d_B(F(x),F(y))}{d_A(x,y)^\alpha}$
for $0<\alpha\leq 1$, and $\Lip(F):=\Hol_1(F)$.

We let $X:M\to TM$ be a $C^{1+\beta}$ vector field on $M$ $(0<\beta<1)$, and $\vf:M\to M$ be
the flow generated by $X$. This means that $\vf$ is a one-parameter family of maps $\vf^t:M\to M$
s.t. $\vf^{t+s}=\vf^t\circ\vf^s$ for all $t,s\in\R$, and s.t.  $X_p(f)=\left.\frac{d}{dt}\right|_{t=0}f[\vf^t(p)]$
for all $f\in C^\infty(M)$. In this case $(t,p)\mapsto \vf^t(p)$  is a
$C^{1+\beta}$ map $[-1,1]\x M\to M$ \cite[page 112]{Ebin-Marsden}.
We assume throughout that $X_p\neq 0$ for all $p$.

\part{The Poincar\'e section}

\section{Poincar\'e sections}

\subsection*{Basic definitions}
Suppose $\vf:M\to M$ is a flow.

\medskip
\noindent
{\sc Poincar\'e section:\/} $\L\subset M$ Borel set s.t. for every $p\in M$, $\{t>0: \vf^t(p)\in\Lambda\}$
is a sequence tending to $+\infty$, and $\{t<0:\vf^t(p)\in\Lambda\}$ is a sequence tending to $-\infty$.

\medskip
\noindent
{\sc Roof function:\/} $R_\Lambda:\Lambda\to (0,\infty)$,
$
R_\L(p):=\min\{t>0:\vf^t(p)\in\L\}.
$

\medskip
\noindent
{\sc Poincar\'e map:\/}  $f_{\Lambda}:\Lambda\to\Lambda$,
$
f_\Lambda(p):=\vf^{R_{\Lambda}(p)}(p).
$

\medskip
\noindent
{\sc Induced measure:\/} Every $\vf$--invariant probability measure $\mu$ on $M$ induces an $f_\Lambda$--invariant  measure $\mu_\L$ on $\L$ defined by the equality
$$
\int_M g d\mu=\frac{1}{\int_\Lambda R_\Lambda d\mu_\Lambda}\int_{\Lambda} \left(\int_0^{R_\Lambda(p)} g[\vf^t(p)]dt\right)d\mu_\Lambda(p),\ \text{for all }g\in L^1(\mu).
$$

\medskip
\noindent
{\sc Uniform Poincar\'e section:\/} The Poincar\'e section $\Lambda$ is called {\em uniform}
if its roof function is bounded away from zero and infinity.
If $\Lambda$ is uniform, then $\mu_\Lambda$ is finite and it can be normalized.
 With this normalization, for every Borel subset $E\subset \L$ and  $0<\e<\inf R_\Lambda$ it holds
$
\mu_\L(E)=\mu[\bigcup_{0<t<\epsilon}\vf^t(E)]\big/\mu[\bigcup_{0<t<\epsilon}\vf^t(\L)].
$

\medskip
All the Poincar\'e sections considered in this paper will be uniform, and each of them
will be the disjoint union of finitely many embedded smooth two dimensional discs.
Let $\partial\Lambda$ denote the union of the boundaries of these discs.
The set $\partial \Lambda$  will introduce discontinuities to the Poincar\'e map of $\Lambda$.

\medskip
\noindent
{\sc Singular set:\/} The {\em singular set} of a Poincar\'e section  $\L$ is
$$
\mathfrak S(\Lambda):=\left\{p\in\L:\begin{array}{l} p\textrm{ does {\em not} have a relative neighborhood $V\subset\L\setminus\partial\Lambda$ s.t. }\\
 \textrm{$V$ is diffeomorphic to an open disc, and $f_{\Lambda}:V\to f_{\Lambda}(V)$ }\\
 \textrm{and $f^{-1}_{\Lambda}:V\to f^{-1}_{\Lambda}(V)$ are diffeomorphisms}
 \end{array}\right\}.
$$

\medskip
\noindent
{\sc Regular set:\/} $\Lambda':=\Lambda\setminus\mathfrak S(\Lambda)$.

\subsection*{Basic constructions}
Let $\vf$ be a flow satisfying our standing assumptions.

\medskip
\noindent
 {\sc Canonical transverse disc:\/}
$
S_r(p):=\{\exp_p(\vec{v}):\vec{v}\in T_p M, \vec{v}\perp X_p, \|\vec{v}\|_p\leq r\}.
$

\medskip
\noindent
{\sc Canonical flow box:\/}
$
{\rm FB}_r(p):=\{\vf^t(q):q\in S_r(p),|t|\leq r\}.
$

\medskip
The following  lemmas are standard, see the appendix for proofs.

\begin{lem}\label{LemmaTD}
There is a constant $\mathfrak r_s>0$ which only depends on $M$ and $\vf$ s.t. for every
$p\in M$ and $0<r<\mathfrak r_s$, $S:=S_r(p)$ is a $C^\infty$ embedded closed disc,
$|\measuredangle(X_q,T_q S)|\geq \frac{1}{2}$ radians for all $q\in S$, and
$\dist_M(\cdot,\cdot)\leq \dist_{S}(\cdot,\cdot)\leq 2\dist_M(\cdot,\cdot)$.
\end{lem}

\begin{lem}\label{LemmaFB_1}
There are constants $\mathfrak r_f,\mathfrak d\in (0,1)$ which only depend on $M$
and $\vf$ s.t. for every $p\in M$, ${\rm FB}_{\mathfrak r_f}(p)$ contains an open ball
with center $p$ and radius $\mathfrak d$, and $(q,t)\mapsto\vf^t(q)$ is a diffeomorphism
from $S_{\mathfrak r_f}(p)\x[-\mathfrak r_f,\mathfrak r_f]$ onto ${\rm FB}_{\mathfrak r_f}(p)$.
\end{lem}

\begin{lem}\label{LemmaFB_2}
There are constants $\mathfrak L, \mathfrak H>1$ which only depend on $M$ and $\vf$ s.t.
$\mathfrak t_p:{\rm FB}_{\mathfrak r_f}(p)\to [-\mathfrak r_f,\mathfrak r_f]$ and
$\mathfrak q_p:{\rm FB}_{\mathfrak r_f}(p)\to S_{\mathfrak r_f}(p)$ defined by
$z=\vf^{\mathfrak t_p(z)}[\mathfrak q_p(z)]$ are well-defined maps with
$\mathrm{Lip}(\mathfrak t_p), \mathrm{Lip}(\mathfrak q_p)\leq \mathfrak L$ and
$\|\mathfrak t_p\|_{C^{1+\beta}}, \|\mathfrak q_p\|_{C^{1+\beta}}\leq \mathfrak H$.
\end{lem}

\noindent
We call $\mathfrak t_p,\mathfrak q_p$ the {\em flow box coordinates}.
Set  $\mathfrak r:=10^{-1}\min\{1,\mathfrak r_s,\mathfrak r_f,\mathfrak d\}/(1+\max\|X_p\|)$.

\medskip
\noindent
{\sc Standard Poincar\'e section:\/}
A  Poincar\'e section $\Lambda$ is {\em standard} if it has the form
$$
\Lambda=\L(p_1,\ldots,p_N;r):=\biguplus_{i=1}^N S_r(p_i)
$$
where $r< \mathfrak r$, $\sup R_\Lambda<\mathfrak r$, and $S_r(p_i)$ are pairwise disjoint.
The points $p_1,\ldots,p_N$ are called the {\em centers} of $\Lambda$,
and $r$ is called the {\em radius} of $\Lambda$. (Here and throughout, $\biguplus$ means the union of  pairwise disjoint sets.)  

\medskip

Standard Poincar\'e sections are special cases of the ``proper families" Bowen used in
\cite[\S2]{Bowen-Symbolic-Flows} to build Markov partitions for Axiom A flows. Their existence is discussed below (Lemma \ref{LemmaStandardSection}).
For the moment, let us assume Standard Poincar\'e sections exist, and discuss some of  their properties.

Fix  a standard Poincar\'e section $\Lambda=\Lambda(p_1,\ldots,p_N;r)$ and write
$f=f_\Lambda$, $R=R_\Lambda$, $\mathfrak S:=\mathfrak S(\Lambda)$,
and $\Lambda':=\Lambda\setminus\mathfrak S$.

\begin{lem}
Every standard Poincar\'e section is a uniform Poincar\'e section.
\end{lem}

\begin{proof}
We have $\sup R<\infty$ by the definition of standard sections,
so it remains to see that $\inf R>0$.
Let $x\in S_r(p_i)$, $f(x)\in S_r(p_j)$. If $i=j$ then $R(x)>\mathfrak r_f$, otherwise
there would exist $0<t\leq \mathfrak r_f$ s.t. $\vf^0(f(x))=f(x)=\vf^t(x)$, which contradicts
the last part of Lemma \ref{LemmaFB_1}. If $i\neq j$ then $\{\vf^t(x)\}_{0\leq t\leq R(x)}$
is a curve from $S_r(p_i)$ to $S_r(p_j)$, thus $R(x)\geq \dist_M(S_r(p_i),S_r(p_j))/\max\|X_p\|$.
Hence $\inf R>0$.
\end{proof}

\begin{lem}\label{Lemma_Smooth_Section}
$R$, $f$ and $f^{-1}$ are differentiable on $\Lambda'$, and  $\exists\mathfrak C>0$
only depending on $M$ and $\vf$ s.t. $\sup_{x\in\Lambda'}\|dR_x\|<\mathfrak C$,
$\sup_{x\in\Lambda'}\|df_x\|<\mathfrak C$, $\sup_{x\in\Lambda'}\|(df_x)^{-1}\|<\mathfrak C$,
$\|f\restriction_{U}\|_{C^{1+\beta}}<\mathfrak C$ and $\|f^{-1}\restriction_{U}\|_{C^{1+\beta}}<\mathfrak C$
for all open and connected $U\subset\Lambda'$.
\end{lem}
\begin{proof}
Suppose $x\in\Lambda'$, then $\exists i,j,k$ s.t.
$f^{-1}(x)\in S_r(p_i)$, $x\in S_r(p_j)$, and $f(x)\in S_r(p_k)$.
Since $f$ is continuous at $x$ and the canonical discs composing $\Lambda$ are closed
and disjoint, $x$ has an open neighborhood  $V$ in $S_r(p_j)$ s.t. for all $y\in V$
it holds $f(y)\in S_r(p_k)$ and $f^{-1}(y)\in S_r(p_i)$.
Since $\sup R<\mathfrak r< 10^{-1}\mathfrak d/\max\|X_p\|$, if $y\in V$ then
$
\dist_M(y,p_k)\leq \dist_M(y,f(y))+\dist_M(f(y),p_k)\leq \max\|X_p\|\sup R+\mathfrak r<\mathfrak d.
$ Similarly, $
\dist_M(y,p_i)<\mathfrak d$. Thus
 $V\subset B_{\mathfrak d}(p_i)\cap B_{\mathfrak d}(p_k)\subset {\rm FB}_{\mathfrak r_f}(p_i)\cap {\rm FB}_{\mathfrak r_f}(p_k)$, whence  $R\restriction_V=-\mathfrak t_{p_k}$,
$f\restriction_V=\mathfrak q_{p_k}$,  $f^{-1}\restriction_V=\mathfrak q_{p_i}$.
Now use  Lemma \ref{LemmaFB_2}.
\end{proof}

Let $\mu$ be a $\vf$--invariant probability measure, and let $\mu_\Lambda$ be
the induced measure on $\Lambda$. If $\mu_\Lambda(\mathfrak S)=0$, then
$\mu_\Lambda[\bigcup_{n\in\Z}f^n(\mathfrak S)]=0$, and the derivative cocycle
$df^n_x:T_x\Lambda\to T_{f^n(x)}\Lambda$ is well-defined $\mu_\Lambda$--a.e.
By Lemma \ref{Lemma_Smooth_Section}, $\log\|df_x\|, \log\|df_x^{-1}\|$ are integrable
(even bounded), so the Oseledets Multiplicative Ergodic Theorem applies, and $f$ has
well-defined Lyapunov exponents $\mu_\Lambda$--a.e. Fix $\chi>0$.

\begin{lem}\label{Lemma_Lyap_Exp}
Suppose $\mu_\Lambda(\mathfrak S)=0$.
If $\mu$ is $\chi$--hyperbolic then $f$ has one Lyapunov exponent in $(-|\ln\mathfrak C|,-\chi\inf R)$ and
another in $(\chi\inf R,|\ln\mathfrak C|)$ for $\mu_{\Lambda}$--a.e. $x\in\Lambda$.
\end{lem}

\begin{proof}
Let $\Omega_\chi$ denote the set of points where the flow has one zero Lyapunov exponent,
one Lyapunov exponent in $(-\infty,-\chi)$ and another in $(\chi,\infty)$. By assumption $\mu[\Omega_\chi^c]=0$, thus
$\Lambda_\chi:=\{x\in\Lambda\setminus\bigcup_{n\in\Z} f^{-n}(\mathfrak S):\exists t>0\textrm{ s.t. }\vf^t(x)\in\Omega_\chi\}$
has full measure with respect to $\mu_\Lambda$.

Let $\Lambda_\chi^\ast:=\{x\in\Lambda_\chi:
\chi(x,\vec{v}):=\lim\limits_{n\to\pm\infty}\frac{1}{n}\log\|df_x^n\vec{v}\|
\text{ exists for all }0\neq \vec{v}\in T_x\Lambda\}.$ By the Oseledets theorem,
$\Lambda_\chi^\ast$ has full $\mu_\Lambda$--measure. By Lemma \ref{Lemma_Smooth_Section},
$|\chi(x,\vec{v})|\leq |\ln\mathfrak C|$. The Lyapunov exponents of $\vf$ are constant along flow lines,
therefore for every $x\in\Lambda^\ast_\chi$ there are vectors $\un{e}^s_x, \un{e}^u_x\in T_x M$ s.t.
$
\lim\limits_{t\to\infty}\frac{1}{t}\log\|d\vf^t_x \un{e}^s_x\|_{\vf^t(x)}<-\chi$ and
 $\lim\limits_{t\to \infty}\frac{1}{t}\log\|d\vf^t_x \un{e}^u_x\|_{\vf^t(x)}>\chi.
$
Let $\vec{n}(x):=\frac{X_x}{\|X_x\|}$. Since $\lim\limits_{t\to\infty}\frac{1}{t}\log\|d\vf^t_x \vec{n}(x)\|_{\vf^t(x)}=0$,
$\{\un{e}^s_x, \un{e}^u_x,\vec{n}(x)\}$ span $T_x M$. Note that $\un{e}^s_x,\un{e}^u_x$
are {\em not} necessarily in $T_x\Lambda$.

Pick two independent vectors $\vec{v}_1,\vec{v}_2\in T_x\Lambda$ and write
$\vec{v}_i=\alpha_i \un{e}^s_x+\beta_i \un{e}^u_x+\gamma_i\vec{n}(x)$, $i=1,2$.
The vectors ${\alpha_1\choose\beta_1},{\alpha_2\choose\beta_2}$ must be linearly independent,
otherwise some non-trivial linear combination of $\vec{v}_1,\vec{v}_2$ equals $\vec{n}(x)$,
which is impossible since $\Span\{\vec{v}_1,\vec{v}_2\}=T_x \Lambda$
and $\Lambda$ is tranverse to the flow. It follows that $T_x\Lambda$ contains two vectors of the form
$$
\vec{v}^s_x=\un{e}^s_x+\gamma_s\vec{n}(x)\, , \  \vec{v}^u_x=\un{e}^u_x+\gamma_u\vec{n}(x).
$$
These vectors are the projections of $\un{e}^s_x, \un{e}^u_x$ to $T_x\Lambda$ along $\vec{n}(x)$.
We will estimate their Lyapunov exponents.

Write $\Lambda=\Lambda(p_1,\ldots,p_N;r)$. As in the proof of Lemma \ref{Lemma_Smooth_Section},
for every $x\in\Lambda\setminus\mathfrak S$, if $f(x)\in S_r(p_i)$, then $x$ has a neighborhood
$V$ in $\Lambda$ s.t.  $V\subset {\rm FB}_{\mathfrak r_f}(p_i)$, $R\restriction_V=-\mathfrak t_{p_i}$,
and $f\restriction_V=\mathfrak q_{p_i}$. More generally, suppose $f^n(x)\in S_r(q_n)$ for $q_n\in\{p_1,\ldots,p_N\}$.
If $x\not\in \bigcup_{k\in\Z}f^k(\mathfrak S)$ then there are open neighborhoods $V_n$ of $x$ in $\Lambda$ s.t.
\begin{equation}\label{pre-tau-n}
f^{n-1}(V_n)\in {\rm FB}_{\mathfrak r_f}(q_n),
\textrm{ and }f^n\restriction_{V_n}=({\mathfrak q}_{q_n}\circ\cdots\circ {\mathfrak q}_{q_1})\restriction_{V_n}.
\end{equation}
By the definition of the flow box coordinates,
${\mathfrak q}_{q_i}(\cdot)=\vf^{-{\mathfrak t}_{q_i}(\cdot)}(\cdot)$ for every $i$.
Since $x\not\in \bigcup_{k\in\Z}f^k(\mathfrak S)$, ${\mathfrak t}_{q_i}$
is continuous on a neighborhood of $f^{i-1}(x)$, hence the smaller $V_n$
the closer $-{\mathfrak t}_{q_i}\restriction_{f^{i-1}(V_n)}$ is to $R(f^{i-1}(x))$.
If $V_n$ is small enough and
$$
R_n:=R(x)+R(f(x))+\cdots+R(f^{n-1}(x)),
$$
then $\vf^{R_n}(y)\in {\rm FB}_{\mathfrak r_f}(q_n)$ for all $y\in{V_n}$,
and we can decompose
\begin{equation}\label{tau-n}
({\mathfrak q}_{q_n}\circ\cdots\circ {\mathfrak q}_{q_1})(y)=({\mathfrak q}_{q_n}\circ \vf^{R_n})(y),\text{ for }y\in V_n.
\end{equation}
We emphasize that the power $R_n$ is the same for all $y\in V_n$.

We use (\ref{pre-tau-n})--(\ref{tau-n}) to calculate $df^n_x \vec{v}^s_x$.
First note that $(d{\mathfrak q}_{q_1})_x\vec{n}(x)=\vec{0}$: let $\gamma(t)=\vf^t(x)$, then ${\mathfrak q}_{q_1}[\gamma(t)]={\mathfrak q}_{q_1}(x)$ for all $|t|$ small, so
$\left.\frac{d}{dt}\right|_{t=0}{\mathfrak q}_{q_1}[\gamma(t)]=\vec{0}$.
By (\ref{pre-tau-n}), $df_x^n\vec{v}^s_x=d({\mathfrak q}_{q_n}\circ\cdots\circ {\mathfrak q}_{q_1})\un{e}^s_x$.
By (\ref{tau-n}),
$\|df^n_x \vec{v}^s_x\|\leq \max_i\|d{\mathfrak q}_{q_i}\|\cdot\|d\vf_x^{R_n}\un{e}^s_x\|_{\vf^{R_n}(x)}$, whence
$
\limsup_{n\to\infty}\frac{1}{n}\log\|df_x^n\vec{v}^s_x\|_{f^n(x)}\leq -\chi\limsup_{n\to\infty}\frac{R_n}{n}\leq -\chi\inf R.
$

Applying this argument to  the reverse  flow $\psi^t:=\vf^{-t}$, we find that the other
Lyapunov exponent belongs to $(\chi\inf R,\infty)$.
\end{proof}

\noindent
{\em Remark.\/} If $\mu$ is ergodic then $\limsup_{n\to\infty}\frac{R_n}{n}=\int R d\mu_\Lambda=1$,
and we get the stronger estimate that the Lyapunov exponents of $f$ are outside $(-\chi,\chi)$ almost surely.

\subsection*{Adapted Poincar\'e sections}
Let $\Lambda$ be a standard Poincar\'e section, and let $\dist_\L$ denote its intrinsic Riemannian distance
(with the convention that the distance between different connected components of $\L$ is infinite).
Let $\mu$ be a $\vf$--invariant probability measure, and let $\mu_\Lambda$ be the induced
probability measure on $\Lambda$.
Recall that $f_{\Lambda}:\Lambda\to\Lambda$ may have singularities. The following definition is motivated by the treatment of Pesin theory for maps with singularities in \cite{Katok-Strelcyn}.

\medskip
\noindent
{\sc Adapted Poincar\'e section:\/}
A standard Poincar\'e section $\Lambda$ is  {\em adapted to $\mu$} if:
\begin{enumerate}[$(1)$]
\item $\mu_\Lambda(\mathfrak S)=0$, where $\mathfrak S=\mathfrak S(\Lambda)$ is the singular set of $\Lambda$.
\item $\lim\limits_{n\to\infty}\frac{1}{n}\log \dist_\Lambda(f^{n}_\Lambda(p),\mathfrak S)=0$
for $\mu_\Lambda$--a.e. $p\in\Lambda$.
\item $\lim\limits_{n\to\infty}\frac{1}{n}\log \dist_\Lambda(f^{-n}_\Lambda(p),\mathfrak S)=0$
for $\mu_\Lambda$--a.e. $p\in\Lambda$.
\end{enumerate}
\noindent
Notice that  $(2)$ implies $(1)$, by the Poincar\'e recurrence theorem.

\medskip
We wish to show that any $\vf$--invariant Borel probability measure has adapted Poincar\'e sections.
The idea is to construct a one-parameter family of standard Poincar\'e sections $\Lambda_r$,
and show that $\Lambda_r$ is adapted to $\mu$ for a.e. $r$. The family is constructed in the next lemma.

\begin{lem}\label{LemmaStandardSection}
For every $h_0>0, K_0>1$ there are $p_1,\ldots,p_N\in M$, $0<\rho_0<h_0/K_0$
s.t. for every $r\in [\rho_0,K_0\rho_0]$ the set $\Lambda(p_1,\ldots,p_N;r)$ is a
standard Poincar\'e section with roof function and radius bounded above by $h_0$.
\end{lem}

\medskip
The existence of standard Poincar\'e sections is treated in \cite[\S2]{Bowen-Symbolic-Flows} as a self-evident fact, but we do not think it is completely obvious. We provide a detailed proof of Lemma \ref{LemmaStandardSection} in the appendix.
The next result shows the existence of adapted sections.

\begin{thm}\label{Thm-Adapted-Section}
Every $\vf$--invariant probability measure $\mu$ has adapted Poincar\'e sections
with arbitrarily small roof functions.
\end{thm}

\begin{proof}
We use parameter selection, as in \cite{Ledrappier-Strelcyn}.
Let $\Lambda_r:=\L(p_1,\ldots,p_N;r)$, $a\leq r\leq b$, be a one-parameter family of
standard Poincar\'e sections as in Lemma \ref{LemmaStandardSection}.
We will show that $\Lambda_r$ is adapted to $\mu$ for Lebesgue a.e. $r\in[a,b]$.

Without loss of generality
$a,b,r,\sup R_{\Lambda_r}\leq h_0<\mathfrak r=\frac{1}{10}\bigl[\frac{\min\{1,\mathfrak r_s,\mathfrak r_f,\mathfrak d\}}{S_0}\bigr]$ for all $r\in [a,b]$,
where $\mathfrak r_s,\mathfrak r_f, \mathfrak d$ are given by Lemmas
\ref{LemmaTD}--\ref{LemmaFB_2}, and $S_0:=1+\max\|X_p\|$.
We define the {\em boundary} of  a canonical transverse disc $S_r(p)$ by the formula
$\del S_r(p):=\{\exp_{p}(\vec{v}): \vec{v}\in T_{p} M, \vec{v}\perp X_{p}, \|\vec{v}\|_p=r\}$. Let
$$
\mathfrak S_r:=\bigcup\big\{{\mathfrak q}_{p_i}\bigl[\del S_r(p_j)\bigr]:
1\leq i,j\leq N\, , \dist_M(S_r(p_i),S_r(p_j))\leq h_0 S_0\bigr\},
$$
where ${\mathfrak q}_{p_i}: {\rm FB}_{\mathfrak r_f}(p_i)\to S_{\mathfrak r_f}(p_i)$
is given by Lemma \ref{LemmaFB_1}.
The assumption that $\dist_M(S_r(p_i),S_r(p_j))\leq h_0 S_0$ ensures the inclusion
$\partial S_r(p_j)\subset {\rm FB}_{\mathfrak r_f}(p_i)$, since for all $q\in \partial S_r(p_j)$,
$\dist_M(q,p_i)\leq \diam[S_r(p_j)]+\dist_M(S_r(p_j),S_r(p_i))+\diam[S_r(p_i)]<h_0 S_0+4r<5\mathfrak rS_0<\mathfrak d$, whence $q\in B_\mathfrak d(p_i)\subset {\rm FB}_{\mathfrak r_f}(p_i)$.

\medskip
\noindent
{\sc Claim.} {\em $\mathfrak S_r$ contains the singular set of $\Lambda_r$.}

\medskip
\noindent
{\em Proof.\/} Fix $r$ and write $R=R_{\Lambda_r}$, $f=f_{\Lambda_r}$.
We show that if $p\in\Lambda_r\setminus\mathfrak S_r$ then $f,f^{-1}$ are local
diffeomorphisms on a neighborhood of $p$.
Let $i,j$ be the unique indices s.t.  $p\in S_r(p_i)$ and $f(p)\in S_r(p_j)$. The speed of the flow is
less than $S_0$, so $\dist_M(p,p_j)\leq \dist_M(p,f(p))+\dist_M(f(p),p_j)< h_0 S_0+r<\mathfrak d$.
Thus $p\in {\rm FB}_{\mathfrak r_f}(p_j)$. Similarly, $\dist_M(f(p),p_i)<\mathfrak d$, so
$f(p)\in {\rm FB}_{\mathfrak r_f}(p_i)$. It follows that
$R(p)=-\mathfrak t_{p_j}(p)=|\mathfrak t_{p_j}(p)|\textrm{ and }
f(p)=\mathfrak q_{p_j}(p).$
Similarly, $p=\mathfrak q_{p_i}[f(p)]$. Since $\dist_M(S_r(p_i),S_r(p_j))\leq \dist_M(p,\vf^{R(p)}(p))<h_0 S_0$
and $p\not\in\mathfrak S_r$,
$$
p\not\in \partial S_r(p_i)\textrm{ and }f(p)\not\in \partial S_r(p_j).
$$
So $\exists V\subset \Lambda_r\setminus\partial \Lambda_r$ relatively open s.t. $V\owns p$ and  $\mathfrak q_{p_j}(V)\subset \Lambda_r\setminus\partial\Lambda_r$.
The map ${\mathfrak q}_{p_j}:V\to {\mathfrak q}_{p_j}(V)$  is a diffeomorphism, because $\mathfrak q_{p_j}$ is differentiable and $\mathfrak q_{p_i}\circ\mathfrak q_{p_j}=\id$ on $V$. We will show that $f\upharpoonright_W=\mathfrak q_{p_j}\upharpoonright_W$ on some open $W\subset \Lambda_r\setminus\partial\Lambda_r$ containing $p$.

Since $R(p)=|\mathfrak t_{p_j}(p)|$, the curve $\{\vf^t(p):0< t< |{\mathfrak t}_{p_j}(p)|\}$
does not intersect $\Lambda_r$. The set $\Lambda_r$ is compact and $\vf,{\mathfrak t}_{p_j}$
are continuous, so $p$ has a relatively open neighborhood $W\subset V$ s.t.
$\{\vf^t(q):0< t< |{\mathfrak t}_{p_j}(q)|\}$  does not intersect $\Lambda_r$  for all $q\in W$.
So
$
f\upharpoonright_W={\mathfrak q}_{p_j}\upharpoonright_W$, and we see that $f$ is a local diffeomorphism at $p$.
Similarly,  $f^{-1}$ is a local diffeomorphism at $p$, which proves that $p\not\in\mathfrak S(\Lambda_r)$,
and hence the claim.

\medskip
We now proceed to the proof of the theorem. We begin with some reductions. Let $f_r:=f_{\Lambda_r}$.
%
By the claim it is enough to show that
\begin{equation}\label{DSC}
\mu_{\L_r}\biggl\{p\in\L_r: \liminf\limits_{|n|\to\infty}\frac{1}{|n|}\log\dist_{\Lambda_r}(f^n_r(p),\mathfrak S_r)<0\biggr\}=0\textrm{  for  a.e. $r\in[a,b]$.}
\end{equation}
Indeed, this implies $\exists r$ s.t.
$\liminf\limits_{|n|\to\infty}\frac{1}{|n|}\log\dist_{\Lambda_r}(f^n_r(p),\mathfrak S(\Lambda_r))\geq 0$
for $\mu_{\Lambda_r}$--a.e. $p\in\Lambda_r$, and the limit is non-positive, because
$\dist_{\Lambda_r}(q,\mathfrak S(\Lambda_r))\leq \dist_{\Lambda_r}(q,\partial\Lambda)\leq r$
for all $q\in\Lambda$. Let
$$
A_\alpha(r):=\{p\in\L_b:\exists \textrm{ infinitely many }n\in\Z\textrm{ s.t. } \tfrac{1}{|n|}\log\dist_{\Lambda_b}(f^{n}_b(p),\mathfrak S_r)<-\alpha\}.
$$
We have $\Lambda_r\subset\Lambda_b$, so $\mu_{\Lambda_r}\ll\mu_{\Lambda_b}$,
$\dist_{\Lambda_r}\geq \dist_{\Lambda_b}$, and $f_r(x)=f_b^{n(x)}(x)$ with
$1\leq n(x)\leq \frac{\sup R_r}{\inf R_b}$. Therefore (\ref{DSC}) follows from the statement
\begin{equation}\label{mu(A)=0}
\forall \textrm{  $\alpha>0$ rational }\bigl(\mu_{\L_b}[A_\alpha(r)]=0\textrm{ for  a.e. }r\in[a,b]\bigr).
\end{equation}
Let
$
I_\alpha(p):=\{a\leq r\leq b:p\in A_\alpha(r)\},
$
then $1_{A_\alpha(r)}(p)=1_{I_\alpha(p)}(r)$, whence by Fubini's Theorem
$\int_a^b \mu_{\L_b}[A_\alpha(r)]dr=\int_{\L_b}\mathrm{Leb}[I_\alpha(p)]d\mu_{\L_b}(p)$.
So  (\ref{mu(A)=0}) follows from
\begin{equation}\label{Leb(I)}
\mathrm{Leb}[I_\alpha(p)]=0\textrm{ for all }p\in\L_b.
\end{equation}
In summary, (\ref{Leb(I)}) $\Rightarrow$ (\ref{mu(A)=0}) $\Rightarrow$ (\ref{DSC}) $\Rightarrow$ the theorem.

\medskip
\noindent
{\sc Proof of (\ref{Leb(I)}):} Fix $p\in\L_b$. If $r\in I_\alpha(p)$ then
$\dist_{\L_b}(f^n_b(p),\mathfrak S_r)<e^{-\alpha |n|}$  for infinitely many $n\in\Z$.
The section $\Lambda_b$ is a finite union of canonical transverse discs $S_b(p_i)$,
and $S_b(p_i)\cap\mathfrak S_r$ is a finite union of projections $S_b(p_i)\cap\mathfrak q_{p_i}[\partial S_r(p_j)]$,
each satisfying $\dist_M(S_r(p_i),S_r(p_j))\leq h_0 S_0$. It follows that there are infinitely many
$n\in\Z$ such that  for some {\em fixed} $1\leq i,j \leq N$ s.t. $\dist_M(S_r(p_i),S_r(p_j))\leq h_0 S_0$,
it holds that
\begin{equation}\label{n}
f^n_b(p)\in S_{b}(p_i)\textrm{ and }
\dist_{\L_b}(f^n_b(p), {\mathfrak q}_{p_i}[\del S_r(p_j)])<e^{-\alpha |n|}.
\end{equation}
Since $\dist_M(S_r(p_i),S_r(p_j))\leq h_0 S_0$, $\dist_M(p_i,p_j)\leq h_0 S_0+2r<\mathfrak d$.
This, and our assumptions on $b$ and $h_0$, guarantee that
${S_b(p_j)},f^n_b(p),{\mathfrak q}_{p_i}(\del S_r(p_j))$ are inside
${\rm FB}_{\mathfrak r_f}(p_i)\cap {\rm FB}_{\mathfrak r_f}(p_j)$, which are in the domains of definition
of $\mathfrak q_{p_i}$ and $\mathfrak q_{p_j}$.

Suppose $n$ satisfies (\ref{n}). Let $q$ be the point that minimizes
$\dist_{\L_b}(f^n_b(p),{\mathfrak q}_{p_i}(q))$ over all $q\in \partial S_r(p_j)$, then
\begin{align*}
&e^{-\alpha |n|}>\dist_{\L_b}(f^n_b(p), {\mathfrak q}_{p_i}(q))\geq \dist_{M}(f^n_b(p), {\mathfrak q}_{p_i}(q))\\
&\geq \mathfrak L^{-1}\dist_{M}({\mathfrak q}_{p_j}[f^n_b(p)], {\mathfrak q}_{p_j}[{\mathfrak q}_{p_i}(q)])& \because\mathrm{Lip}({\mathfrak q}_j)\leq \mathfrak L\\
&= \mathfrak L^{-1}\dist_{M}({\mathfrak q}_{p_j}[f^n_b(p)], q)& \because q\in  S_r(p_j)\\
&\geq (2\mathfrak L)^{-1}\dist_{\L_b}({\mathfrak q}_{p_j}[f^n_b(p)], q).& \because\dist_{S_b(p_j)}\leq 2\dist_M
\end{align*}
Thus $\dist_{\L_b}({\mathfrak q}_{p_j}[f^n_b(p)], q)\leq 2\mathfrak L e^{-\alpha |n|}$, which implies that
$$
|\dist_{\L_b}(p_j, {\mathfrak q}_{p_j}[f^n_b(p)])-\dist_{\L_b}(p_j,q)|\leq 2\mathfrak L e^{-\alpha |n|}.
$$
We now use the special geometry of canonical transverse discs:
$q\in \del S_r(p_j)$, so $\dist_{\L_b}(p_j,q)=r$.
Writing $D_{jn}(p):=\dist_{\L_b}(p_j, {\mathfrak q}_{p_j}[f^n_b(p)])$,
we see that for every $n$ which satisfies (\ref{n}), $|D_{jn}(p)-r|\leq 2\mathfrak L e^{-\alpha |n|}$.
Thus every $r\in I_\alpha(p)$ belongs to
$$
\bigcup_{j=1}^N\{r\in [a,b]: \exists \textrm{ infinitely many  $n\in\Z$} \textrm{ s.t.  }
|r- D_{jn}(p)|\leq 2\mathfrak L e^{-\alpha |n|}\}.
$$
By the Borel-Cantelli Lemma, this set has zero Lebesgue measure.
\end{proof}

\medskip
Two standard Poincar\'e sections with the same set of centers are called {\em concentric}.
Since $b/a$ can be chosen arbitrarily large, the last proof shows the following.

\begin{cor}\label{Cor-Adapted}
Let $\mu$ be a $\vf$--invariant probability measure.
For every $h_0>0$ there are two concentric standard Poincar\'e sections $\Lambda_i=\Lambda(p_1,\ldots,p_N;r_i)$ with height functions bounded above by $h_0$, s.t. $\Lambda_1$ is adapted to $\mu$ and $r_2>2 r_1$.
\end{cor}
\noindent
To see this   take $r_1$ close to $a$ s.t. $\Lambda_{r_1}$ is adapted, and $r_2=b$.

\medskip
\noindent
{\em Remark.\/} The adapted Poincar\'e section given in Theorem \ref{Thm-Adapted-Section} and Corollary \ref{Cor-Adapted} depends on the measure $\mu$. It would be  interesting to construct, for a given $\chi>0$,  a Poincar\'e section which is adapted to all
ergodic hyperbolic measures with one Lyapunov exponent bigger than $\chi$ and one Lyapunov exponent smaller than $-\chi$.

\section{Pesin charts for adapted  Poincar\'e sections}\label{SectionPesinCharts}

One of the central tools in Pesin theory is a system of local coordinates which present
a non-uniformly hyperbolic map as a perturbation of a uniformly hyperbolic linear map
\cite{Pesin-Izvestia-1976,Katok-Hasselblatt-Book,Barreira-Pesin-Non-Uniform-Hyperbolicity-Book}.
We will construct such coordinates for the Poincar\'e map of an adapted  Poincar\'e section.
Adaptability is used, as in \cite{Katok-Strelcyn}, to control the size of the coordinate patches
along typical orbits (Lemma \ref{Lemma_Q_decay}).

Suppose $\mu$ is a $\vf$--invariant probability measure on $M$, and assume that
$\mu$ is $\chi_0$--hyperbolic for some $\chi_0>0$. We do not assume ergodicity.
Fix once and for all a standard Poincar\'e section $\Lambda=\Lambda(p_1,\ldots,p_N;r)$
for $\vf$, which is adapted to $\mu$. Set $f:=f_\Lambda, R:=R_{\Lambda}, \mathfrak S:=\mathfrak S(\Lambda)$,
and let $\mu_\Lambda$ be the induced measure on $\Lambda$.

Without loss of generality, there is a larger concentric standard Poincar\'e section
$\wt{\Lambda}:=\L(p_1,\ldots,p_N;\wt{r})$  s.t. $\wt{r}>2r$. Thus $\wt{\Lambda}\supset\Lambda$,
and $\dist_{\wt{\Lambda}}(\Lambda,\partial\wt{\Lambda})>r$. We will use $\wt{\Lambda}$ as a safety
margin in the following definition of the {\em exponential map} of $\Lambda$:
 $$
\mathrm{Exp}_x: \{\vec{v}\in T_x \L: \|\vec{v}\|_x<r\}\to \wt{\L},\ \Exp_x(\vec{v}):=\g_x(\|\vec{v}\|_x),
$$
where $\g_x(\cdot)$ is the geodesic in $\wt{\L}$ s.t. $\g(0)=x$ and $\dot{\g}(0)=\vec{v}$.
This makes sense even near $\partial\Lambda$, because every geodesic of $\L$ can be
prolonged $r$ units of distance into $\wt{\Lambda}$ without falling off the edge.
Notice that geodesics of $\wt{\Lambda}$ are usually not geodesics of $M$, therefore
$\Exp_x$ is usually different from $\exp_x$.
As in \cite[chapter 9]{Spivak}, there are $\rho_{\textrm{dom}}, \rho_{\textrm{im}}\in(0,r)$
s.t. for every $x\in \Lambda$, $\mathrm{Exp}_x$ is a 2--bi-Lipschitz diffeomorphism from
$\{\vec{v}\in T_x \Lambda: \|\vec{v}\|_x<\sqrt{2}\rho_{\mathrm{dom}}\}$ onto a relative
neighborhood of $\{y\in \wt{\Lambda}:\dist_{\wt{\Lambda}}(y,x)<\rho_{\mathrm{im}}\}$.
%

\subsection*{Non-uniform hyperbolicity}
Since $\Lambda$ is adapted to $\mu$, $\mu_\Lambda(\mathfrak S)=0$. By Lemma \ref{Lemma_Lyap_Exp},
for $\mu_\Lambda$--a.e. $x\in \Lambda$, $f$ has one Lyapunov exponent in $(-\infty,-\chi_0\inf R)$
and one Lyapunov exponent in $(\chi_0\inf R,\infty)$. Let $\chi:=\chi_0\inf R$.

\medskip
\noindent
{\sc Non-uniformly hyperbolic set:}
Let $\NUH_\chi(f)$ be the set  of $x\in \Lambda\setminus\bigcup_{n\in\mathbb Z}f^{-n}(\mathfrak S)$
s.t. $T_{f^n(x)}\L=E^u(f^n(x))\oplus E^s(f^n(x))$, $n\in\Z$, where $E^u,E^s$ are one-dimensional
linear subspaces, and:
\begin{enumerate}[(i)]
\item $\lim\limits_{n\to\pm\infty}\frac{1}{n}\log\|df^n_x\vec{v}\|<-\chi$ for all non-zero $\vec{v}\in E^s(x)$.
\item $\lim\limits_{n\to\pm\infty}\frac{1}{n}\log\|df^{-n}_x\vec{v}\|<-\chi$ for all non-zero $\vec{v}\in E^u(x)$.
\item $\lim\limits_{n\to\pm\infty}\frac{1}{n}\log|\sin\measuredangle(E^s(f^n(x)),E^u(f^n(x))) |=0$.
\item $df_x E^s(x)=E^s(f(x))$ and $df_x E^u(x)=E^u(f(x))$.
\end{enumerate}
By the Oseledets Theorem and Lemma \ref{Lemma_Lyap_Exp}, $\mu_\L[\NUH_{\chi}(f)]=1$.

\subsection*{Pesin charts}

These are a system of coordinates on $\NUH_\chi(f)$ which
simplifies the form of $f$. The following definition is slightly different than in Pesin's original
work \cite{Pesin-Izvestia-1976}, but the proofs are essentially the
same.\footnote{The difference is in the choice of $\chi$ in the exponential terms
$e^{2k\chi}$ in the definitions of the Pesin parameters $s(x),u(x)$. In Pesin's work,
$\chi=$ Lyapunov exponent of $x$ minus $\epsilon$, while here it is constant.}
Fix a measurable family of unit vectors $\un{e}^u(x)\in E^u(x),\un{e}^s(x)\in E^s(x)$ on
$\NUH_{\chi}(f)$. Since $\dim E^{u/s}(x)=1$,  $\un{e}^{u/s}(x)$ are determined up to a sign.
To make the choice, let $(\un{e}^1_x,\un{e}^2_x)$ be a continuous choice of basis for $T_x\L$
so that $\<\un{e}^1_x,\un{e}^2_x,X_x\>$ has   positive orientation. Pick $\un{e}^u(x),\un{e}^s(x)$
s.t. $\measuredangle(\un{e}^1_x,\un{e}^s(x))\in [0,\pi)$, and $\measuredangle(\un{e}^s(x),\un{e}^u(x))>0$.

\medskip
\noindent
{\sc Pesin parameters:\/}
Given  $x\in\NUH_{\chi}(f)$, let
\begin{enumerate}[$\circ$]
\item $\a(x):=\measuredangle(\un{e}^s(x),\un{e}^u(x))$,
\item $s(x):=\sqrt{2}\left(\sum_{k=0}^\infty e^{2k\chi}\|df^k_x \un{e}^s(x)\|^2_{f^k(x)}\right)^{\frac{1}{2}}$,
\item $u(x):=\sqrt{2}\left(\sum_{k=0}^\infty e^{2k\chi}\|df^{-k}_x \un{e}^u(x)\|^2_{f^{-k}(x)}\right)^{\frac{1}{2}}$.
\end{enumerate}
The infinite series converge, because $x\in\NUH_\chi(f)$.

\medskip
\noindent
{\sc Oseledets-Pesin reduction:}
Define a  linear transformation $C_{\chi}(x):\R^2\to T_x\Lambda$
by mapping ${1\choose 0}\mapsto s(x)^{-1}\un{e}^s(x)$ and ${0\choose 1}\mapsto u(x)^{-1}\un{e}^u(x)$.

\medskip
This diagonalizes the derivative cocycle $df_x:T_x M\to T_{f(x)} M$ as follows.

\begin{thm}\label{Thm_OP_Reduction}
$\exists C_\vf$ s.t.  $\forall x\in\NUH_{\chi}(f)$,  $
C_{\chi}(f(x))^{-1}\circ df_x\circ C_{\chi}(x)=\bigl({\tiny
\begin{array}{cc}
A_x & 0 \\
0 & B_x
\end{array}}
\bigr)
$, where $C_\vf^{-1}\leq |A_x|\leq e^{-\chi}$, and $e^{\chi}\leq |B_x|\leq C_\vf$.
\end{thm}

\noindent
The proof is a routine modification of the proofs in \cite[theorem 3.5.5]{Barreira-Pesin-Non-Uniform-Hyperbolicity-Book}
or \cite[theorem S.2.10]{Katok-Hasselblatt-Book}, using the uniform bounds on
$df\upharpoonright_{\Lambda\setminus\mathfrak S}$ (Lemma \ref{Lemma_Smooth_Section}).

Our conventions for $\un{e}^s(x),\un{e}^u(x)$ guarantee that $C_\chi(x)$ is orientation-preserving,
and one can show exactly as in \cite[Lemmas 2.4--2.5]{Sarig-JAMS} that
\begin{equation}\label{C-inverse-norm}
\|C_\chi(x)\|\leq 1\textrm{ and }\frac{1}{\sqrt{2}}\frac{\sqrt{s(x)^2+u(x)^2}}{|\sin\a(x)|}\leq \|C_{\chi}(x)^{-1}\|\leq \frac{\sqrt{s(x)^2+u(x)^2}}{|\sin\a(x)|}.
\end{equation}
We see that $\|C_\chi(x)^{-1}\|$ is large exactly when $E^s(x)\approx E^u(x)$ (small $|\sin\alpha(x)|$),
or when it takes a long time to notice the exponential decay of $\frac{1}{n}\log\|df_x^n \un{e}^s(x)\|$
or of $\frac{1}{n}\log\|df_x^{-n} \un{e}^u(x)\|$ (large $s(x)$ or large $u(x)$).
In summary, large $\|C_\chi(x)^{-1}\|$ means bad hyperbolicity.

\medskip
\noindent
{\sc Pesin Maps:}
The {\em Pesin map} at $x\in\NUH_{\chi}(f)$ (not to be confused with the Pesin {\em chart} defined below) is
$\Psi_x:[-\rho_{\dom},\rho_{\dom}]^2\to \wt{\L}$, given by
$$
\Psi_x(u,v)=\Exp_x\biggl[C_{\chi}(x){u\choose v}\biggr].
$$
The map $\Psi_x$  is orientation-preserving, and it maps $[-\rho_{\dom},\rho_{\dom}]^2$
diffeomorphically onto a neighborhood of $x$ in $\wt{\Lambda}\setminus\partial\wt{\Lambda}$.
We have $\Lip(\Psi_x)\leq 2$, because $\|C_\chi(x)\|\leq 1$, but $\Lip(\Psi_x^{-1})$ is not
uniformly bounded, because $\|C_\chi(x)^{-1}\|$ can be arbitrarily large.

\medskip
\noindent
{\sc Maximal size:} Fix some parameter $0<\epsilon<2^{-\frac{3}{2}}$ (which will be calibrated later).
Although $\Psi_x$ is well-defined on all of $[-\rho_{\dom},\rho_{\dom}]^2$,
it will only be useful for us on the smaller set $[-Q_\epsilon(x),Q_\epsilon(x)]^2$, where
\begin{equation}\label{Q_def}
Q_\epsilon(x):=\left\lfloor\e^{3/\beta}\left(\frac{\sqrt{s(x)^2+u(x)^2}}{|\sin\a(x)|}\right)^{-12/\beta}\wedge\bigl(\epsilon\dist_{\Lambda}(x,\mathfrak S) \bigr)\wedge \rho_{\dom}\right\rfloor_\epsilon.
\end{equation}
Here $\mathfrak S$ is the singular set of $\Lambda$, $a\wedge b:=\min\{a,b\}$,
$\beta$ is the constant in the $C^{1+\beta}$ assumption on $\vf$, and
$\lfloor t \rfloor_\epsilon:=\max\{\theta\in I_\epsilon: \theta\leq t\}$
where $I_\epsilon:=\{e^{-\frac{1}{3}\ell\epsilon}:\ell\in\N\}$.

\medskip
The value $Q_\epsilon(x)$ is called the {\em maximal size} (of the Pesin charts defined
below).\footnote{We do not claim that Theorem \ref{Thm_Pesin_Coord} below does not
hold on larger boxes $[-Q',Q']^2$.} Notice that
$Q_\epsilon\leq \e^{3/\beta} \|C_\chi^{-1}\|^{-12/\beta}$, so $Q_\epsilon(x)$
is small when $x$ is close to $\mathfrak S$ or when the hyperbolicity at $x$ is bad.
Another important property of $Q_\epsilon$ is that, thanks to the inequalities
$\|C_\chi\|\leq 1$ and $Q_\epsilon<2^{-\frac{3}{2}}\dist_{\Lambda}(x,\mathfrak S)$,
\begin{equation}\label{Pesin-Inside}
\Psi_x\bigl([-Q_\epsilon(x),Q_\epsilon(x)]^2\bigr)\subset \Lambda\setminus\mathfrak S.
\end{equation}
This is in contrast to $\Psi_x\bigl([-\rho_{\dom},\rho_{\dom}]^2\bigr)$,
which may intersect $\mathfrak S$ or $\wt{\Lambda}\setminus\Lambda$.

\medskip
\noindent
{\sc Pesin Charts:}
The {\em maximal Pesin chart} at $x\in\NUH_\chi(f)$ (with parameter $\e$) is
$\Psi_x:[-Q_\epsilon(x),Q_\epsilon(x)]^2\to\Lambda\setminus\mathfrak S$,
$\Psi_x(u,v)=\mathrm{Exp}_x[C_{\chi}(x){u\choose v}]$. The {\em Pesin chart of size $\eta$} is
$\Psi_x^\eta:=\Psi_x\upharpoonright_{[-\eta,\eta]^2}$ for $0<\eta\leq Q_\e(x)$.

\medskip
The Pesin charts provide a system of local coordinates on a neighborhood of $\NUH_\chi(f)$.
The following theorem says that the Poincar\'e map ``in coordinates"
$$
f_x:=\Psi_{f(x)}^{-1}\circ f\circ\Psi_x:[-Q_\epsilon(x),Q_\epsilon(x)]^2\to\R^2
$$
is close to a uniformly hyperbolic {\em linear} map.
In what  follows, $\un{0}={0\choose 0}$ and  the statement ``{\em for all $\epsilon$ small enough $P$ holds}"
means ``$\exists\epsilon_0>0$ which depends only on $M,\vf,\Lambda,\beta,\chi_0$ s.t.
for all $0<\epsilon<\epsilon_0$, $P$ holds".

\begin{thm}[Pesin]\label{Thm_Pesin_Coord}
For all $\epsilon$ small enough the following holds. For every $x\in\NUH_{\chi}(f)$,
$f_x$ is well-defined and injective on $[-Q_\epsilon(x),Q_\epsilon(x)]^2$,
and can be put there in the form
$f_x(u,v)=\bigl(A_x u+h_x^1(u,v),B_x v+h_x^2(u,v)\bigr)$, where:
\begin{enumerate}[$(1)$]
\item $C_{\vf}^{-1}\leq |A_x|\leq e^{-\chi}$ and $e^{\chi}\leq |B_x|\leq C_{\vf}$,
with $C_{\vf}$ as in Theorem  \ref{Thm_OP_Reduction}.
\item $h_x^i$ are $C^{1+\frac{\beta}{2}}$ functions s.t. $h_x^i(\un{0})=0$, $(\nabla h_x^i)(\un{0})=\un{0}$.
\item $\|h_x^i\|_{C^{1+\frac{\beta}{2}}}<\epsilon$ on $[-Q_\epsilon(x),Q_{\epsilon}(x)]^2$.
\end{enumerate}
A similar statement holds for
$f_x^{-1}:=\Psi_{f^{-1}(x)}^{-1}\circ f^{-1}\circ\Psi_x:[-Q_\epsilon(x),Q_\epsilon(x)]^2\to\R^2$.
\end{thm}

\begin{proof}
Let $U:=\Psi_x([-Q_\epsilon(x),Q_\epsilon(x)]^2)$. By (\ref{Pesin-Inside}), $f$ and $f^{-1}$
are $C^{1+\beta}$ on $U$, with uniform bounds on their $C^{1+\beta}$ norms
(Lemma \ref{Lemma_Smooth_Section}). Now continue as in \cite[Theorem 2.7]{Sarig-JAMS} or
 \cite[Theorem 5.6.1]{Barreira-Pesin-Non-Uniform-Hyperbolicity-Book},
 replacing $M$ by $\Lambda$ and $\exp_p$ by $\Exp_p$.
\end{proof}

\subsection*{Adaptability and temperedness}

The maximal size of Pesin charts may not be bounded below on $\NUH_\chi(f)$.
A central idea in Pesin theory is that it is nevertheless possible to control how fast
$Q_\epsilon$ decays along typical orbits.

Define for this purpose the set $\NUH^\ast_{\chi}(f)$ \label{NUH-star} of all $x\in \NUH_{\chi}(f)$
which on top of the defining properties (i)--(iv) of $\NUH_{\chi}(f)$ also satisfy:
\begin{enumerate}[\ (i)]
\setcounter{enumi}{4}
\item $\lim\limits_{n\to\pm\infty}\frac{1}{n}\log\dist_{\Lambda}(f^n(x),\mathfrak S)=0$.
\item $\lim\limits_{n\to\pm\infty}\frac{1}{n}\log\|C_{\chi}(f^n(x))^{-1}\|=0$.
\item $\lim\limits_{n\to\pm\infty}\frac{1}{n}\log\|C_{\chi}(f^n(x))\un{v}\|=0$ for $\un{v}={1\choose 0}, {0\choose 1}$.
\item $\lim\limits_{n\to\pm\infty}\frac{1}{n}\log|\det C_{\chi}(f^n(x))|=0$.
\end{enumerate}

\begin{lem}\label{Lemma_Q_decay}
$\NUH^\ast_{\chi}(f)$ is an $f$--invariant Borel set of full $\mu_\L$--measure, and for every
$x\in \NUH^\ast_{\chi}(f)$, $\lim\limits_{n\to\pm\infty}\frac{1}{n}\log Q_\epsilon(f^n(x))=0$.
\end{lem}

\begin{proof}
Condition (v) holds $\mu_\Lambda$--a.e. because $\Lambda$ is adapted to $\mu$.
Conditions (vi)--(viii) hold $\mu_\Lambda$--a.e. because of the Oseledets Theorem
(apply the proof of \cite[Lemma 2.6]{Sarig-JAMS} to the ergodic components of $\mu_{\Lambda}$).
By (\ref{C-inverse-norm}), conditions (v)--(vi) imply that
$\lim_{n\to\pm\infty}\frac{1}{n}\log Q_{\epsilon}(f^n(x))=0$.
\end{proof}

\begin{lem}[Pesin's Temperedness Lemma]\label{PropTemp}
There exists a positive Borel function $q_\epsilon:\NUH_{\chi}^\ast(f)\to (0,1)$ s.t. for every
$x\in\NUH_{\chi}^\ast(f)$, $0<q_\epsilon(x)\leq \epsilon Q_{\epsilon}(x)$ and
$e^{-\epsilon/3}\leq \frac{q_\epsilon\circ f}{q_\epsilon}\leq e^{\epsilon/3}$.
\end{lem}

\noindent
This lemma follows from Lemma \ref{Lemma_Q_decay} as in
\cite[Lemma 3.5.7]{Barreira-Pesin-Non-Uniform-Hyperbolicity-Book}.
It implies that
\begin{equation}\label{Jonah}
Q_\epsilon(f^n(x))> e^{-\frac{1}{3}|n|\epsilon}q_\epsilon(x)\textrm{ for all }n\in\Z,
\end{equation}
which gives a control on the decay of $Q_\epsilon$ along typical orbits.


\subsection*{Overlapping Pesin charts}

Theorem \ref{Thm_Pesin_Coord} says that $f_x:=\Psi_{f(x)}^{-1}\circ f \circ\Psi_x$
is close to a linear hyperbolic map. This property is stable under perturbations,
therefore we expect $f_{xy}:=\Psi_{y}^{-1}\circ f \circ\Psi_x$ to be close to a linear
hyperbolic map, whenever $\Psi_y$ is ``sufficiently close" to $\Psi_{f(x)}$.
We now specify the meaning of ``sufficiently close".

Recall that $\Lambda$ is the disjoint union of a finite number of canonical transverse discs $S_r(p_i)$.
Let
$
D_r(p_i):=S_r(p_i)\setminus\partial S_r(p_i)=\{\exp_{p_i}(\vec{v}):\vec{v}\perp X_{p_i},\|\vec{v}\|< r\}
$.
Choose for every $D=D_r(p_i)$ a map  $\Theta:TD\to\R^2$ s.t.:
\begin{enumerate}[(1)]
\item $\Theta:T_x D\to \R^2$ is a linear isometry for all $x\in D$.
\item Let $\vartheta_x=(\Theta\upharpoonright_{T_x D})^{-1}:\R^2\to T_x D$,
then $(x,\underline{u})\mapsto(\mathrm{Exp}_x\circ\vartheta_x)(\underline{u})$ is smooth and
Lipschitz on $\Lambda\x\{\underline{u}\in\R^2:\|\underline{u}\|<\rho_{\dom}\}$ with respect to the metric
$d\bigl((x,\underline{u}),(x',\underline{u}')\bigr):=\dist_\Lambda(x,x')+\|\underline{u}-\underline{u}'\|$.
\item $x\mapsto \vartheta_x^{-1}\circ\mathrm{Exp}_x^{-1}$ is a Lipschitz map from $D$
to $C^2(D,\R^2)$, the space of $C^2$ maps from $D$ to $\R^2$.
\end{enumerate}
\label{SectionOverlap}

Recall that the Pesin map is $\Psi_x(u,v)=\Exp_x[C_\chi(x){u\choose v}]$, and the Pesin chart
of size $0<\eta<Q_\epsilon(x)$ is $\Psi_x^\eta:=\Psi_x\upharpoonright_{[-\eta,\eta]^2}$.

\medskip
\noindent
{\sc Overlapping charts:}  Let $x_1,x_2\in\NUH_\chi(f)$. We say that
$\Psi_{x_1}^{\eta_1}, \Psi_{x_2}^{\eta_2}$ {\em $\epsilon$--overlap}, and write
$\Psi_{x_1}^{\eta_1}\overset{\epsilon}{\approx}\Psi_{x_2}^{\eta_2}$,  if $x_1,x_2$
lie in the same transversal disc of $\Lambda$,
$e^{-\epsilon}<\frac{\eta_1}{\eta_2}<e^{\epsilon}$, and
$\dist_\Lambda(x_1,x_2)+\|\Theta\circ C_{\chi}(x_1)-\Theta\circ C_{\chi}(x_2)\|<\eta_1^4\eta_2^4.$


\begin{prop}\label{Prop_Overlap_Meaning}
The following holds for all $\epsilon$  small enough.
If $\Psi_{x_1}^{\eta_1}\overset{\epsilon}{\approx}\Psi_{x_2}^{\eta_2}$ then:
\begin{enumerate}[$(1)$]
\item {\em $\Psi_{x_i}^{\eta_i}$ chart nearly the same patch:}
$\Psi_{x_i}\bigl([-e^{-2\epsilon}\eta_i,e^{-2\epsilon}\eta_i]^2\bigr)\subset\Psi_{x_j}\bigl([-\eta_j,\eta_j]^2\bigr)$.
\item {\em $\Psi_{x_i}$ define  nearly the same coordinates:}
$\dist_{C^{1+\frac{\beta}{2}}}(\Psi_{x_i}^{-1}\circ\Psi_{x_j},\mathrm{Id})<\epsilon\eta_i^2\eta_j^2$,
where the $C^{1+\frac{\beta}{2}}$ distance is calculated on $[-e^{-\epsilon}\rho_{\dom},e^{-\epsilon}\rho_{\dom}]^2$.
\end{enumerate}
\end{prop}

\begin{cor}\label{Cor_Pesin_Coord}
The following holds for all $\epsilon$ small enough. If $x,y\in\NUH_{\chi}(f)$ and
$\Psi_{f(x)}^{\eta'}\overset{\epsilon}{\approx}\Psi_y^{\eta}$ then
$f_{xy}:=\Psi_y^{-1}\circ f\circ\Psi_x:[-Q_\epsilon(x),Q_{\epsilon}(x)]^2\to\R^2$ is well-defined,
injective, and can be put in the form
$f_{xy}(u,v)=\bigl(A_{xy} u+h_{xy}^1(u,v), B_{xy} v+h_{xy}^2(u,v)\bigr)$, where:
\begin{enumerate}[$(1)$]
\item $C_\vf^{-1}\leq |A_{xy}|\leq e^{-\chi}$ and $e^{\chi}\leq |B_{xy}|\leq C_\vf$, with $C_{\vf}$
as in Theorem~\ref{Thm_OP_Reduction}.
\item $|h_{xy}^i(\un{0})|<\epsilon\eta$, $\|\nabla h_{xy}^i(\un{0})\|<\epsilon\eta^{\beta/3}$.
\item $\|h^i_{xy}\|_{C^{1+\frac{\beta}{3}}}<\epsilon$ for $i=1,2$, where the norm is taken
on $[-Q_\epsilon(x),Q_\epsilon(x)]^2$.
\end{enumerate}
A similar statement holds for $f_{xy}^{-1}:=\Psi_{x}^{-1}\circ f^{-1}\circ\Psi_y$
whenever $\Psi_{f^{-1}(y)}^{\eta'}\overset{\epsilon}{\approx}\Psi_x^\eta$.
\end{cor}
\noindent

The proofs are routine modifications of \cite[Props. 3.2 and 3.4]{Sarig-JAMS} once
we replace $M$ by one of the canonical transverse discs in $\Lambda$ and
$\exp_x$ by $\Exp_x$. For Proposition \ref{Prop_Overlap_Meaning}, we use the definition
of overlap, and for Corollary \ref{Cor_Pesin_Coord} we treat $f_{xy}=(\Psi_y^{-1}\circ\Psi_{f(x)})\circ f_x$
as a small perturbation of $f_x$ and then use Theorem \ref{Thm_Pesin_Coord}.

\part{Symbolic dynamics}

Throughout this part we assume that $M,X$ and $\vf$ satisfy our standing assumptions,
and that $\mu$ is a $\chi_0$--hyperbolic $\vf$--invariant probability measure on $M$.
We fix a standard Poincar\'e section $\Lambda=\Lambda(p_1,\ldots,p_N;r)$ adapted to $\mu$,
and a larger concentric standard section $\wt{\Lambda}:=\Lambda(p_1,\ldots,p_N;\wt{r})$ s.t. $\wt{r}>2r$.
Let $f,R$ and $\mathfrak S$ denote the Poincar\'e map, roof function, and singular set of $\Lambda$,
and let $\chi:=\chi_0\inf R$ (a bound for the Lyapunov exponents of $f$ at $\mu_\Lambda$--a.e. point,
see Lemma~\ref{Lemma_Lyap_Exp}).


In this part of the paper we construct a countable Markov partition for $f$ on a set of full measure
with respect to $\mu_\Lambda$,  and then use it to develop symbolic dynamics for $\vf$.
This was done in \cite{Sarig-JAMS} for surface diffeomorphisms, and the proof would have applied
to our setup verbatim had $\mathfrak S$ been empty. We will indicate the changes needed to treat
the case  $\mathfrak S\neq \emptyset$.

Not many changes are needed, because most of the work is done inside Pesin charts, where
$f$ and $f^{-1}$ are smooth with uniformly bounded $C^{1+\beta}$ norm. One point is worth
mentioning, though: \cite{Sarig-JAMS} uses a uniform bound on $|\ln Q_\epsilon(f(x))/Q_\epsilon(x)|$,
where $Q_\epsilon(x)$ is the maximal size of a Pesin chart. This quantity is no longer bounded when
$\mathfrak S\neq \emptyset$. When this or other effects of $\mathfrak S$ matter, we will give
complete details. Otherwise, we will just sketch the general idea and refer to \cite{Sarig-JAMS} for details.

\section{Generalized pseudo-orbits and shadowing}\label{Section_GPO}

\subsection*{Generalized pseudo-orbits (gpo)}

Fix some small $\epsilon>0$. Recall that a pseudo-orbit with parameter $\epsilon$ is a sequence
of points $\{x_i\}_{i\in\Z}$ satisfying the nearest neighbor conditions $\dist(f(x_i),x_{i+1})<\epsilon$ for all $i\in\Z$.
A gpo is also a sequence of objects satisfying nearest neighbor conditions, but the objects and the
conditions are more complicated, because of the need to record the hyperbolic features of each point.

\medskip
\noindent
{\sc   $\epsilon$--Double charts:} Ordered pairs
$
\Psi_x^{p^u,p^s}\!:=(\Psi_x\upharpoonright_{[-p^u,p^u]^2}, \Psi_x\upharpoonright_{[-p^s,p^s]^2})
$
where $x\in\NUH_{\chi}(f)$ and $0<p^u,p^s\leq Q_{\epsilon}(x)$ (same Pesin chart, different domains).

\medskip
\noindent
{\sc
$\epsilon$--Generalized pseudo-orbit (gpo):} A sequence $\{\Psi_{x_i}^{p^u_i,p^s_i}\}_{i\in\mathbb Z}$
of $\epsilon$--double charts which satisfies the following nearest neighbor conditions for all $i\in\mathbb Z$:
\begin{enumerate}[{(GPO1)}]
\item
 $\Psi_{f(x_i)}^{p^u_{i+1}\wedge p^s_{i+1}}\overset{\epsilon}{\approx}\Psi_{x_{i+1}}^{p^u_{i+1}\wedge p^s_{i+1}}$   and $
 \Psi_{f^{-1}(x_{i+1})}^{p^u_i\wedge p^s_i}\overset{\epsilon}{\approx}\Psi_{x_i}^{p^u_i\wedge p^s_i}$\label{GPO1}, cf. Prop. \ref{Prop_Overlap_Meaning}.
\item $p^u_{i+1}=\min\{e^{\epsilon} p^u_i,Q_{\epsilon}(x_{i+1})\}$ and $
p^s_i=\min\{e^{\epsilon} p^s_{i+1},Q_{\epsilon}(x_i)\}$.\label{GPO2}
\end{enumerate}
A {\em positive gpo} is a one-sided sequence $\{\Psi_{x_i}^{p^u_i,p^s_i}\}_{i\geq 0}$
with {(GPO1)}, {(GPO2)}. A {\em negative gpo} is a one-sided sequence $\{\Psi_{x_i}^{p^u_i,p^s_i}\}_{i\leq 0}$
with {(GPO1)}, {(GPO2)}. Gpos were called ``chains" in \cite{Sarig-JAMS}.

\medskip
\noindent
{\sc Shadowing:}
A gpo $\{\Psi_{x_i}^{p^u_i,p^s_i}\}_{i\in\Z}$ {\em shadows the orbit of $x$}, if
$f^i(x)\in\Psi_{x_i}\bigl([-\eta_i,\eta_i]^2\bigr)$ for all $i\in\Z$, where $\eta_i:=p^u_i\wedge p^s_i$.

\medskip
This notation is heavy, so we will sometimes abbreviate it by writing $v_i$ instead of
$\Psi_{x_i}^{p^u_i,p^s_i}$, and letting $p^u(v_i):=p^u_i, p^s(v_i):=p^s_i, x(v_i):=x_i$.
The nearest neighbor conditions [(GPO\ref{GPO1}) + (GPO\ref{GPO2})] will be expressed
by the notation $v_i\epto v_{i+1}$.

\begin{lem}\label{LemmaEasy}
Suppose $0<p^u_i,p^s_i\leq Q_i$ satisfy $p^u_{i+1}=\min\{e^\epsilon p^u_i, Q_{i+1}\}$ and
$p^s_{i}=\min\{e^\epsilon p^s_{i+1}, Q_{i}\}$ for  $i=0,1$. If $\eta_i:=p^u_i\wedge p^s_i$, then
$\eta_{i+1}/\eta_i\in [e^{-\epsilon},e^{\epsilon}]$.
\end{lem}

See \cite[Lemma 4.4]{Sarig-JAMS} for the proof. Let $v=\Psi_{x}^{p^u,p^s}$ be an $\epsilon$--double chart.


\medskip
\noindent
{\sc Admissible manifolds:}
An {\em $s$--admissible manifold} in $v$ is a set of the form $\Psi_x\{(t,F(t)):|t|\leq p^s\}$,
where $F:[-p^s,p^s]\to \R$ satisfies:
\begin{enumerate}[(${\mathrm {Ad}}$1)]
\item $|F(0)|\leq 10^{-3}(p^u\wedge p^s)$.
\item $|F'(0)|\leq \frac{1}{2}(p^u\wedge p^s)^{\beta/3}$.
\item $F$ is $C^{1+\frac{\beta}{3}}$  and $\sup|F'|+\Hol_{\beta/3}(F)\leq \frac{1}{2}$.
\end{enumerate}
Similarly, a {\em $u$--admissible manifold} in $v$ is a set of the form $\Psi_x\{(F(t),t):|t|\leq p^u\}$,
where $F:[-p^u,p^u]\to \R$  satisfies (Ad1--3).

\medskip
The constant $10^{-3}$ in (Ad1) implies that $s$--admissible manifolds in $v$ intersect
$u$--admissible manifolds in $v$ inside the smaller set $\Psi_x([-10^{-2}\eta,10^{-2}\eta]^2)$,
where $\eta=p^u\wedge p^s$ (see Theorem \ref{Thm_Shadowing} below).
We call $F$ the {\em representing function}, and we denote the collections of all $s/u$--admissible
manifolds in $v$ by $\mathfs M^s(v)$ and $\mathfs M^u(v)$.
The representing function satisfies $\|F\|_\infty\leq Q_\epsilon(x)$, because $p^u,p^s\leq Q_\epsilon(x)$,
$|F(0)|\leq 10^{-3}(p^u\wedge p^s)$ and $|F'|\leq \frac{1}{2}$.\footnote{In fact
$|F'(t)|\leq |F'(0)|+\frac{1}{2}|t|^{\beta/3}\leq \epsilon$ for $t\in\dom(F)$, since
$|t|\leq p^{u/s}\leq Q_\epsilon(x)\leq \epsilon^{3/\beta}$.}
As a result, $s/u$--admissible manifolds
are subsets of $\Psi_x([-Q_\epsilon(x),Q_\epsilon(x)]^2)$, a set where $f$ is smooth, and where
if $\epsilon$ is small enough then $f$ is a perturbation of a uniformly hyperbolic linear map in
Pesin coordinates (Theorem \ref{Thm_Pesin_Coord}). This implies the following fact.

\medskip
\noindent
{\sc Graph Transform Lemma:\/} {\em
For all $\epsilon$ small enough, if $v_i\epto v_{i+1}$, then the forward image of a $u$--admissible manifold
$V^u\in \mathfs M^u(v_i)$ contains a unique $u$--admissible manifold $\mathfs F_{v_i v_{i+1}}^u[V^u]$,
called the {\em (forward) graph transform} of $V^u$.}

\medskip
\noindent
{\em Sketch of proof\/} (see \cite[Prop. 4.12]{Sarig-JAMS}, \cite[Supplement]{Katok-Hasselblatt-Book}, or \cite{Barreira-Pesin-Non-Uniform-Hyperbolicity-Book} for details). Let
$
f_{x_i x_{i+1}}:=\Psi_{x_{i+1}}^{-1}\circ f \circ\Psi_{x_i}:[-Q_\epsilon(x_i),Q_\epsilon(x_i)]^2\to\R^2
$.
By (GPO1) and Corollary \ref{Cor_Pesin_Coord}, $f_{x_i x_{i+1}}$ is $\epsilon$ close in the $C^{1+\frac{\beta}{3}}$ norm on $[-Q_\epsilon(x_i),Q_\epsilon(x_i)]^2$ to a linear map which contracts the $x$--coordinate by at least $e^{-\chi}$ and expands the $y$--coordinate by at least $e^{\chi}$. Direct calculations show that
if $\epsilon$ is much smaller than $\chi$ and  $V^u\in\mathfs M^u(v_i)$, then  $
f(V^u)\supset\Psi_{x_{i+1}}\{(G(t),t):t\in [a,b]\}
$
where $G$ satisfies (Ad1--3), and $[a,b]\supset [-e^{{\chi}/{2}}p^u_i,e^{{\chi}/{2}}p^u_i]$.
By (GPO2), if $\epsilon<\chi/2$ then $[-e^{\chi/2}p^u_i,e^{\chi/2}p^u_i]\supset [-p^u_{i+1},p^u_{i+1}]$,
so $f(V^u)$ restricts to a $u$--admissible manifold in $v_{i+1}$.

\medskip
There is also a {\em (backward) graph transform}
$\mathfs F_{v_{i+1}v_{i}}^s: \mathfs M^s(v_{i+1})\to \mathfs M^s(v_{i})$, ob\-tained by
applying $f^{-1}$ to $s$--admissible manifolds in $v_{i+1}$ and restricting the result to
an $s$--admissible manifold in $v_{i}$. Put a metric on $\mathfs M^u(v_i)$ and $\mathfs M^s(v_{i+1})$
by measuring the sup-norm distance between the representing functions.
Using the form of $f_{x_i x_{i+1}}$ in coordinates, one can  show by direct calculations that
$\mathfs F^u_{v_i v_{i+1}}:\mathfs M^u(v_i)\to \mathfs M^u(v_{i+1})$ and
$\mathfs F^s_{v_{i+1} v_i}:\mathfs M^s(v_{i+1})\to \mathfs M^s(v_{i})$ contract distances
by at least $e^{-\chi/2}$, see \cite[Prop. 4.14]{Sarig-JAMS}.

Suppose $\un{v}^-=\{v_i\}_{i\leq 0}$ is a {\em negative} gpo, and pick arbitrary
$V^u_{-n}\in\mathfs M^u(v_{-n})$ $(n\geq 0)$, then
$
V^u_{0,n}:=({\mathfs F}^u_{v_{-1} v_0}\circ \cdots\circ{\mathfs F}^u_{v_{-n+1} v_{-n+2}}
\circ{\mathfs F}^u_{v_{-n} v_{-n+1}})(V^u_{-n})\in\mathfs M^u(v_0)$.
Using the uniform contraction of $\mathfs F_{v_{-i-1} v_{-i}}$, it is easy to see that $\{V^u_{0,n}\}_{n\geq 1}$
is a Cauchy sequence, and that its limit is independent of the choice of $V^u_{-n}$ \cite[Prop. 4.15]{Sarig-JAMS}.
Thus we can make the following definition for all $\epsilon$ small enough.

\medskip
\noindent
{\sc The unstable manifold of a negative gpo $\un{v}^{-}$:}
$$
V^u[\un{v}^-]:=\lim_{n\to\infty} ({\mathfs F}^u_{v_{-1} v_0}\circ \cdots\circ{\mathfs F}^u_{v_{-n+1} v_{-n+2}}\circ{\mathfs F}^u_{v_{-n} v_{-n+1}})(V^u_{-n})
$$
for some (any) choice of $V^u_{-n}\in\mathfs M^u(v_{-n})$.

\medskip
Working with {\em positive} gpos and {\em backward} graph transforms, we can also make the following definition.

\medskip
\noindent
{\sc The stable manifold of a positive gpo $\un{v}^{+}$:}
$$
V^s[\un{v}^+]:=\lim_{n\to\infty}
({\mathfs F}^s_{v_{1} v_0}\circ \cdots\circ{\mathfs F}^s_{v_{n-1} v_{n-2}}\circ{\mathfs F}^s_{v_{n} v_{n-1}})(V^s_{n})
$$
for some (any) choice of $V^s_{n}\in\mathfs M^s(v_{n})$.

\medskip
The following properties hold:
\begin{enumerate}[(1)]
\item {\sc Admissibility:} $V^u[\un{v}^-]\in\mathfs M^u(v_0)$ and $V^s[\un{v}^+]\in\mathfs M^s(v_0)$.
This is because $\mathfs M^u(v_0),\mathfs M^s(v_0)$ are closed in the supremum norm.
\item {\sc Invariance:} $f^{-1}(V^u[\{v_i\}_{i\leq 0}])\subset V^u[\{v_i\}_{i\leq -1}]$,
$f(V^s[\{v_i\}_{i\geq 0}])\subset V^s[\{v_i\}_{i\geq 1}]$. This is immediate from the definition.
\item {\sc Hyperbolicity:}\label{Hyperbolicity-property} if $x,y\in V^{u}[\un{v}^-]$ then
$\dist_\Lambda(f^{-n}(x),f^{-n}(y))\xrightarrow[n\to\infty]{}0$, and if $x,y\in V^s[\un{v}^+]$ then
$\dist_\Lambda(f^{n}(x),f^{n}(y))\xrightarrow[n\to\infty]{}0$. The rates are exponential.
\item {\sc H\"older property:}\label{Holder-property-stable-vectors}
The maps $\un{v}\mapsto V^u[\{v_i\}_{i\leq 0}],V^s[\{v_i\}_{i\geq 0}]$
are H\"older continuos, i.e. there exist constants $K>0$ and $0<\theta<1$ s.t. for all $n\geq 0$,
if $\{u_i\}_{i\in\Z},\{v_i\}_{i\in\Z}$ are gpo's with $u_i=v_i$ for all $|i|\leq n$ then
\begin{align*}
{\rm dist}_{C^1}( V^u[\{u_i\}_{i\leq 0}], V^u[\{v_i\}_{i\leq 0}])&\leq K\theta^n\\
{\rm dist}_{C^1}( V^s[\{u_i\}_{i\geq 0}], V^s[\{v_i\}_{i\geq 0}])&\leq K\theta^n.
\end{align*}
\end{enumerate}
Above, ${\rm dist}_{C^1}$ is the $C^1$ distance between two admissible manifolds: if
$V_1=\Psi_x\{(t,F_1(t)):|t|\leq p^{s}\}$, $V_2=\Psi_x\{(t,F_2(t)):|t|\leq p^{s}\}$ are $s$--admissible
manifolds then
$$
{\rm dist}_{C^1}(V_1,V_2):=\max|F_1-F_2|+\max|F_1'-F_2'|,
$$
and a similar definition holds for $u$--admissible manifolds.

\medskip
To prove part (3) notice first that by the invariance property, $f^n(V^s[\un{v}^+])$ and $f^{-n}(V^u[\un{v}^-])$
remain inside Pesin charts. Therefore $f^n\upharpoonright_{V^s[\un{v}^+]}$ and
$f^{-n}\upharpoonright_{V^u[\un{v}^-]}$ can be written in Pesin coordinates as compositions
of $n$ uniformly hyperbolic maps on $\R^2$.  One can then use direct calculations as in the
proof of Pesin's Stable Manifold Theorem to prove (3). See e.g. \cite[Prop. 6.3]{Sarig-JAMS}.

Part (4) is proved almost verbatim as in the case of diffeomorphisms \cite[Prop. 4.15(5)]{Sarig-JAMS}.
Here is a crude explanation: the Pesin charts avoid the singular set hence their $C^{1+\beta}$ norms
are uniformly bounded, and since $V^u[\{v_i\}_{i\leq 0}],V^s[\{v_i\}_{i\geq 0}]$ are limits of admissible
manifolds via the graph transform method, they depend H\"older continuously on $\un{v}$.

\subsection*{The Shadowing Lemma}

\begin{thm}\label{Thm_Shadowing}
The following holds for all $\epsilon$ small enough: Every gpo with parameter $\epsilon$ shadows a unique orbit.
\end{thm}

\medskip
\noindent
{\em Sketch of proof.\/}
Let $\un{v}=\{v_i\}_{i\in\Z}$ be a gpo, $v_i=\Psi_{x_i}^{p^u_i,p^s_i}$.
We have to show that there exists a unique $x$ s.t. $f^i(x)\in\Psi_{x_i}([-\eta_i,\eta_i]^2)$
for all $i\in\Z$, where $\eta_i=p^u_i\wedge p^s_i$.
The sets $V^u:=V^u[\{v_i\}_{i\leq 0}]$ and $V^s:=V^s[\{v_i\}_{i\geq 0}]$ are admissible manifolds in $v_0$.
Because of properties (Ad1--3), $V^u$ and $V^s$ intersect at a unique point $x$,
and $x$ belongs to $\Psi_{x_0}([-10^{-2}\eta_0,10^{-2}\eta_0]^2)$
\cite[Prop. 4.11]{Sarig-JAMS}, see also \cite[Cor. S.3.8]{Katok-Hasselblatt-Book}.
By the invariance property,
$$
f^n(x)\in\Psi_{x_n}([-Q_\epsilon(x_n),Q_\epsilon(x_n)]^2)\textrm{ for all }n\in\Z.
$$
We will show that $\un{v}$ shadows $x$, and that $x$ is the only such point.

Any $y$ s.t. $f^n(y)\in\Psi_{x_n}([-Q_\epsilon(x_n),Q_\epsilon(x_n)]^2)$ for all $n\in\Z$ equals $x$.
The map $f_{x_n x_{n+1}}:=\Psi_{x_{n+1}}^{-1}\circ f\circ\Psi_{x_n}$ is uniformly hyperbolic
on $[-Q_\epsilon(x_n),Q_\epsilon(x_n)]^2$. If  $\Psi_{x_0}^{-1}(x)$ and $\Psi_{x_0}^{-1}(y)$
have different $y$--coordinates, then successive application of $f_{x_n x_{n+1}}$ will expand
the difference between the $y$--coordinates of $\Psi_{x_n}^{-1}(f^n(x))$, $\Psi_{x_n}^{-1}(f^n(y))$
exponentially as $n\to\infty$. If $\Psi_{x_0}^{-1}(x),\Psi_{x_0}^{-1}(y)$  have different $x$--coordinates,
then successive application of $f_{x_{-n-1},x_{-n}}^{-1}$ will expand the difference between
the $x$--coordinates of $\Psi_{x_{n}}^{-1}(f^{n}(x))$,$\Psi_{x_{n}}^{-1}(f^{n}(y))$ exponentially
as $n\to-\infty$. But these differences are bounded by $2Q_\epsilon(x_n)$ whence by a constant,
so $\Psi_{x_0}^{-1}(x)=\Psi_{x_0}^{-1}(y)$, whence $x=y$.

Let $y_k$ denote the unique intersection point of $V^u[\{v_i\}_{i\leq k}]$ and $V^s[\{v_i\}_{i\geq k}]$,
then $f^n(y_k),f^{n+k}(x)\in \Psi_{x_{n+k}}([-Q_\epsilon(x_{n+k}),Q_\epsilon(x_{n+k})]^2)$ for all $n\in\Z$.
By the previous paragraph, $y_k=f^k(x)$. Since $y_k$ is the intersection of a $u$--admissible
manifold and an $s$--admissible manifold in $v_k$, $y_k\in\Psi_{x_k}([-\eta_k,\eta_k]^2)$ where
$\eta_k:=p^u_k\wedge p^s_k$. It follows that $f^k(x)\in \Psi_{x_k}([-\eta_k,\eta_k]^2)$ for all $k\in\Z$.
Thus $\un{v}$ shadows the orbit of $x$, and $x$ is unique with this property.
\hfill$\Box$

\subsection*{Which points are shadowed by gpos?}

To appreciate where the difficulty lies, let us try the na\"{i}ve approach:
given $x\in\NUH_\chi(f)$, set $x_i:=f^i(x)$, and look for $p^u_i,p^s_i$ s.t.
$\{\Psi_{x_i}^{p^u_i,p^s_i}\}_{i\in\Z}$ is a gpo. (GPO1) is automatic, but without
additional information on $Q_\epsilon(f^i(x))$, it is not clear that there exist
$p^{u}_i, p^s_i$ satisfying (GPO2).
This is where we will use the adaptedness of $\Lambda$:
$\lim\limits_{|n|\to\infty}\frac{1}{n}\log Q_{\epsilon}(f^n(x))=0$ for a.e. $x$, whence by
(\ref{Jonah}) there exist $q_\epsilon(x)>0$ s.t.
$Q_\epsilon(f^n(x))> e^{-\frac{1}{3}\epsilon|n|}q_\epsilon(x)>e^{-\epsilon|n|}q_\epsilon(x)$ for all $n\in\Z$.
So the following suprema range over non-empty sets:
\begin{align*}
p^u_i&:=\sup\{t>0: Q_\epsilon(f^{i-n}(x))\geq e^{-{\epsilon}n}t\textrm{ for all }n\geq 0\},\\
p^s_i&:=\sup\{t>0: Q_\epsilon(f^{i+n}(x))\geq e^{-{\epsilon}n}t \textrm{ for all }n\geq 0\}.
\end{align*}
It is easy to see that $p^u_i,p^s_i$ satisfy (GPO2), so
$\{\Psi_{f^i(x)}^{p^u_i,p^s_i}\}_{i\in\Z}$ is a gpo shadowing $x$.

If we want to use the  previous construction to shadow a set of full measure of orbits, then we need uncountably many ``letters" $\Psi_x^{p^u,p^s}$.  The following proposition achieves this with  a countable discrete collection.
 Recall the definition of $\NUH_\chi^\ast(f)$ from Lemma \ref{Lemma_Q_decay}, and let
\begin{equation}\label{NUH-sharp}
\NUH_{\chi}^\#(f):=\{x\in\NUH_{\chi}^\ast(f):\limsup_{n\to\infty}q_\epsilon(f^n(x)), \limsup_{n\to-\infty}q_\epsilon(f^{n}(x))\neq 0\}.
\end{equation}
Thus $\NUH_{\chi}^\#(f)$ is $f$--invariant of full measure (by the Poincar\'e recurrence theorem).

\begin{prop}\label{Prop_Coarse_Gr}
The following holds for all $\epsilon$ small enough. There exists a countable collection
of $\epsilon$--double charts $\mathfs A$ with the following properties:
\begin{enumerate}[$(1)$]
\item {\sc Discreteness:} Let $D(x):=\dist_\Lambda\bigl(\{x,f(x),f^{-1}(x)\},\mathfrak S\bigr)$,
then for every $t>0$ the set $\{\Psi_x^{p^u,p^s}\in\mathfs A:D(x), p^u,p^s>t\}$ is finite.
\item {\sc Sufficiency:} For every $x\in\NUH_{\chi}^\#(f)$ there is a gpo $\{v_n\}_{n\in\Z}\in\mathfs A^\Z$
which shadows $x$, and which satisfies $p^u(v_n)\wedge p^s(v_n)\geq e^{-\epsilon/3}q_{\epsilon}(f^n(x))$ for all $n\in\Z$.
\item {\sc Relevance:} For every $v\in\mathfs A$ there is a gpo $\un{v}\in\mathfs A^\Z$ s.t. $v_0=v$
and $\un{v}$ shadows a point in $\NUH_\chi(f)$.
\end{enumerate}
\end{prop}

\begin{proof}
The proof for diffeomorphisms in \cite[Props. 3.5 and 4.5]{Sarig-JAMS} does not extend to our case,
because it uses a uniform bound $F^{-1}\leq Q_\epsilon\circ f/Q_\epsilon\leq F$ which does not
hold in the presence of singularities. We bypass this difficulty as follows.

Let $X:=[\Lambda\setminus\bigcup_{i=-1,0,1}f^i(\mathfrak S)]^3\x(0,\infty)^3\x\mathrm{GL}(2,\R)$,
together with the product topology, and let $Y\subset X$ denote the subset of $(\un{x},\un{Q},\un{C})\in X$
of the form
\begin{align*}
\un{x}&=(x,f(x),f^{-1}(x)), \textrm{ where }x\in\NUH_\chi^\ast(f),\\
\un{Q}&=(Q_\e(x),Q_\epsilon(f(x)),Q_\epsilon(f^{-1}(x))),\\
\un{C}&=(C_{\chi}(x),C_{\chi}(f(x)),C_{\chi}(f^{-1}(x))).
\end{align*}
\noindent
Cut $Y$ into the countable disjoint union $Y=\biguplus_{(k,\un{\ell})\in\mathbb N_0^4} Y_{k,\un{\ell}}$
where $\N_0=\N\cup\{0\}$ and
$$
Y_{k,\un{\ell}}:=\biggl\{(\un{x},\un{Q},\un{C})\in Y:  \begin{array}{c}x\in\NUH_{\chi}^\ast(f),
\ e^{-(k+1)}< Q_\epsilon(x)\leq e^{-k},\text{ and}\\
e^{-(\ell_i+1)}< \dist_\L(f^i(x),\mathfrak S)\leq e^{-\ell_i}\ \ (i=0,1,-1)\end{array}\biggr\}.
$$

\medskip
\noindent
{\sc Precompactness Lemma.} {\em  $Y_{k,\un{\ell}}$ are precompact in $X$.}

\noindent
{\em Proof.\/}
Suppose $(\un{x},\un{Q},\un{C})\in Y_{k,\un{\ell}}$. By (\ref{C-inverse-norm}) and (\ref{Q_def}),
$\|C_{\chi}(x)^{-1}\|\leq \epsilon^{\frac{1}{4}}e^{\frac{\beta (k+1)}{12}}$. One can show as in
\cite[page 403]{Sarig-JAMS}, that  $C^{-1}\leq \|C_{\chi}(y)^{-1}\|/\|C_{\chi}(f(y))^{-1}\|\leq C$
for all $y\in\NUH_{\chi}(f)$ for some global constant $C$. It follows that
$$
\|C_{\chi}(f^i(x))^{-1}\|\leq C\epsilon^{\frac{1}{4}}e^{\frac{\beta (k+1)}{12}} \textrm{ for }i=0,1,-1.
$$
Since $C_{\chi}(\cdot)$ is a contraction, $\un{C}\in G_k\x G_k \x G_k$, where $G_k$ is the compact set
$\{A\in\mathrm{GL}(2,\R): \|A\|\leq 1, \|A^{-1}\|\leq C\epsilon^{\frac{1}{4}}e^{\frac{\beta (k+1)}{12}}\}$.

Next we bound $\un{Q}$ in a compact set. By (\ref{C-inverse-norm}),
if $(\un{x},\un{Q},\un{C})\in Y_{k,\un{\ell}}$, then
$$
\frac{Q_\epsilon(f^i(x))}{Q_{\epsilon}(x)}\geq \biggl(\frac{\sqrt{2}\|C_{\chi}(f^i(x))^{-1}\|}{\|C_{\chi}(x)^{-1}\|}\biggr)^{-\frac{12}{\beta}}\wedge \left(\frac{\dist_{\Lambda}(f^i(x),\mathfrak S)}{\dist_{\Lambda}(x,\mathfrak S)}\right)\geq  2^{-\frac{6}{\beta}}C^{-\frac{12}{\beta}}\wedge e^{-(\ell_i+1)},
$$
whence $Q_\epsilon(f^i(x))\geq 2^{-\frac{6}{\beta}}C^{-\frac{12}{\beta}}e^{-(\ell_i+1)}e^{-(k+1)}$.
By definition, $Q_\epsilon(f^i(x))\leq \rho_{\dom}$, so $\un{Q}\in (F_{k,\un{\ell}})^3$ with
$F_{k,\un{\ell}}\subset \R$ compact.

Finally, $\un{x}\in E_{\ell_0}\x f(E_{\ell_0})\x f^{-1}(E_{\ell_0})$, with
$E_{\ell_0}:=\{y\in\Lambda:\dist_{\Lambda}(y,\mathfrak S)\geq e^{-\ell_0-1}\}$.
The set $E_{\ell_0}$ is compact because $\Lambda$ is compact, and $f(E_{\ell_0}), f^{-1}(E_{\ell_0})$
are compact because $f\upharpoonright_{E_{\ell_0}},f^{-1}\upharpoonright_{E_{\ell_0}}$ are continuous.

In summary, $Y_{k,\un{\ell}}\subset \prod_{i=0,1,-1}f^i(E_{\ell_0})\x (F_{k,\un{\ell}})^3\x (G_k)^3$,
a compact subset of $X$. So $Y_{k,\un{\ell}}$ is precompact in $X$, proving the lemma.

\medskip
Since $Y_{k,\un{\ell}}$ is precompact in $X$, $Y_{k,\un{\ell}}$ contains a finite set $Y_{k,\un{\ell}}(m)$
s.t. for every $(\un{x},\un{Q},\un{C})\in Y_{k,\un{\ell}}$ there exists some
$(\un{y},\un{Q}',\un{C}')\in Y_{k,\un{\ell}}(m)$ s.t. for every $|i|\leq 1$:
\begin{enumerate}[(a)]
\item $\dist_\Lambda(f^i(x),f^i(y))+\|\Theta\circ C_{\chi}(f^i(x))-\Theta\circ C_{\chi}(f^i(y))\|<e^{-8(m+3)}$.
\item $e^{-\epsilon/3}<Q_\epsilon(f^i(x))/Q_\epsilon(f^i(y))<e^{\epsilon/3}$.
\end{enumerate}
($\Theta$ is defined at the end of \S\ref{SectionPesinCharts}.)

\medskip
\noindent
{\sc Definition of $\mathfs A$:}
The set of $\epsilon$--double charts $\Psi_{x}^{p^u,p^s}$ s.t. for some $k,\ell_0,\ell_{1},\ell_{-1},m$:
\begin{enumerate}[(\textrm{A}1)]
\item $x$ is the first coordinate of some $(\un{x},\un{Q},\un{C})\in Y_{k,\un{\ell}}(m)$.
\item  $0<p^u,p^s\leq Q_\epsilon(x)$ and $p^u,p^s\in I_{\epsilon}=\{e^{-\frac{1}{3}\ell\epsilon}:\ell\in\N\}$.
\item  $p^u\wedge p^s\in [e^{-m-2},e^{-m+2}]$.
\end{enumerate}

\medskip
\noindent
{\sc Proof that  $\mathfs A$ is discrete:} Fix $t>0$. Suppose $\Psi_x^{p^u,p^s}\in\mathfs A$ and
let $k,\un{\ell},m$ be as above. If $D(x),p^u,p^s>t$, then:
\begin{enumerate}[$\circ$]
\item $k\leq |\log t|$ because $t<p^u\leq Q_\epsilon(x)\leq e^{-k}$.
\item $\ell_i\leq |\log t|$ because $t<D(x)\leq \dist_\Lambda(f^i(x),\mathfrak S)\leq e^{-\ell_i}$.
\item $m\leq |\log t|+2$ because $t<p^u\wedge p^s\leq e^{-m+2}$.
\end{enumerate}
So
$
\#\{x:\Psi_x^{p^u,p^s}\in\mathfs A, D(x),p^u,p^s>t\}\leq \sum\limits_{k,\ell_0,\ell_1,\ell_{-1}=0}^{\big\lceil|\log t|\big\rceil}
\sum\limits_{m=0}^{\big\lceil|\log t|\big\rceil+2}\# Y_{k,\un{\ell}}(m)<\infty$.
Also,
$
\#\{(p^u,p^s):\Psi_x^{p^u,p^s}\in\mathfs A, D(x),p^u,p^s>t\}\leq (\#(I_\epsilon\cap [t,1]))^2<\infty$.
Thus $\#\{\Psi_x^{p^u,p^s}\in\mathfs A: D(x),p^u,p^s>t\}<\infty$, proving that $\mathfs A$ is discrete.

\medskip
The proof of sufficiency requires some preparation. A sequence $\{(p^u_n,p^s_n)\}_{n\in\Z}$ is called
{\em $\epsilon$--subordinated} to a sequence $\{Q_n\}_{n\in\Z}\subset I_{\epsilon}$,
if $0<p^u_n,p^s_n\leq Q_n$; $p^u_n,p^s_n\in I_\epsilon$; $p^u_{n+1}=\min\{e^\epsilon p^u_n, Q_{n+1}\}$
and $p^s_{n}=\min\{e^\epsilon p^s_{n+1}, Q_{n}\}$ for all $n$.

\medskip
\noindent
{\sc First Subordination Lemma.\/} {\em Let $\{q_n\}_{n\in\Z}, \{Q_n\}_{n\in\Z}\subset I_{\epsilon}$.
If for every $n\in\Z$ $0<q_n\leq Q_n$ and $e^{-\epsilon}\leq q_n/q_{n+1}\leq e^{\epsilon}$, then
there exists $\{(p^u_n,p^s_n)\}_{n\in\Z}$ which is $\epsilon$--subordinated to $\{Q_n\}_{n\in\Z}$,
and such that $p^u_n\wedge p^s_n\geq q_n$ for all $n$.}

\medskip
\noindent
{\sc Second Subordination Lemma.\/} {\em Suppose $\{(p^u_n,p^s_n)\}_{n\in\Z}$ is
$\epsilon$--subordinated to $\{Q_n\}_{n\in\Z}$. If $\limsup\limits_{n\to\infty}(p^u_n\wedge p^s_n)>0$
and $\limsup\limits_{n\to-\infty}(p^u_n\wedge p^s_n)>0$, then $p^u_n$ (resp. $p^s_n$)
is equal to $Q_n$ for infinitely many $n>0$, and for infinitely many $n<0$.}

\medskip
\noindent
These are Lemmas 4.6 and 4.7 in \cite{Sarig-JAMS}.

\medskip
\noindent
{\sc Proof of sufficiency:}
Fix $x\in\NUH^\#_{\chi}(f)$. Recall the definition of $q_{\epsilon}(\cdot)$ from Pesin's
Temperedness Lemma (Lemma \ref{PropTemp}), and choose  $q_n\in I_\epsilon$ s.t.
$q_n/q_\epsilon(f^n(x))\in [e^{-\epsilon/3},e^{\epsilon/3}]$.
Necessarily $e^{-\epsilon}\leq q_n/q_{n+1}\leq e^{\epsilon}$.

By the first subordination lemma there exists $\{(q^u_n,q^s_n)\}_{n\in\Z}$ s.t. $\{(q^u_n,q^s_n)\}_{n\in\Z}$
is $\epsilon$--subordinated to $\{e^{-\epsilon/3}Q_\epsilon(f^n(x))\}_{n\in\Z}$, and
$q^u_n\wedge q^s_n\geq q_n$ for all $n\in\Z$. Let $\eta_n:=q^u_n\wedge q^s_n$.
By Lemma \ref{LemmaEasy}, $e^{-\epsilon}\leq \eta_{n+1}/\eta_n\leq e^{\epsilon}$.
Since $\eta_n\geq q_n\geq e^{-\epsilon/3}q_\epsilon(f^n(x))$ and $x\in\NUH_{\chi}^\#(f)$,
$\limsup_{n\to\pm\infty}\eta_n>0$.

Choose non-negative integers $m_n, k_n,\un{\ell}_n=(\ell^n_0,\ell^n_1,\ell^n_{-1})$ s.t. for all $n\in\Z$:
\begin{enumerate}[$\circ$]
\item $\eta_n\in [e^{-m_n-1},e^{-m_n+1}]$.
\item $Q_\epsilon(f^n(x))\in (e^{-k_n-1},e^{-k_n}]$.
\item $\dist_{\Lambda}(f^{n+i}(x),\mathfrak S)\in (e^{-\ell^n_i-1},e^{-\ell^n_i}]$ for $i=0,1,-1$.
\end{enumerate}
Choose an element of $Y_{k_n,\un{\ell}_n}$ with first coordinate $f^n(x)$, and approximate it
by some element of $Y_{k_n,\un{\ell}_n}(m_n)$ with first coordinate $x_n$ s.t. for $i=0,1,-1$:
\begin{enumerate}[(a$_n$)]
\item $\dist_\Lambda(f^i(f^n(x)),f^i(x_n))+\|\Theta\circ C_{\chi}(f^i(f^n(x)))- \Theta\circ C_{\chi}(f^i(x_n))\|<e^{-8(m_n+3)}$.
\item $e^{-\epsilon/3}<Q_{\epsilon}(f^i(f^n(x)))/Q_{\epsilon}(f^i(x_n))<e^{\epsilon/3}$.
\end{enumerate}
By $(\mathrm b_n)$ with $i=0$, $Q_{\epsilon}(x_n)\geq e^{-\epsilon/3}Q_{\epsilon}(f^n(x))\geq \eta_n$.
By the first subordination lemma, there exists $\{(p^u_n,p^s_n)\}_{n\in\Z}$ $\epsilon$--subordinated
to $\{Q_\epsilon(x_n)\}_{n\in\Z}$ such that $p^u_n\wedge p^s_n\geq \eta_n$ for all $n\in\Z$.
Necessarily, $p^u_n\wedge p^s_n\geq e^{-\epsilon/3}q_{\epsilon}(f^n(x))$. Let
$$
\un{v}:=\{\Psi_{x_n}^{p^u_n,p^s_n}\}_{n\in\Z}.
$$
We will show that $\un{v}\in\mathfs A^\Z$, $\un{v}$ is a gpo, and $\un{v}$ shadows the orbit of $x$.

\medskip
\noindent
{\em Proof that $\Psi_{x_n}^{p^u_n,p^s_n}\in\mathfs A$:\/} (A1), (A2) are clear, so we focus on (A3).
It is enough to show that $1\leq (p^u_n\wedge p^s_n)/\eta_n\leq e$. The lower bound is by construction.
For the upper bound, recall that $\limsup_{n\to\pm\infty}\eta_n>0$, so by the second subordination
lemma $q^u_n=e^{-\epsilon/3}Q_\epsilon(f^n(x))$ for infinitely many $n<0$. By (b$_n$) with $i=0$,
$q^u_n\geq e^{-\epsilon}Q_\epsilon(x_n)\geq e^{-\epsilon}p^u_n$ for infinitely many $n<0$.
If $q^u_n\geq e^{-\epsilon}p^u_n$ then $q^u_{n+1}\geq e^{-\epsilon}p^u_{n+1}$:
\begin{align*}
&q^u_{n+1}=\min\{e^{\epsilon}q^u_n,e^{-\epsilon/3}Q_\epsilon(f^{n+1}(x))\}\\
&\geq \min\{e^{\epsilon} e^{-\epsilon}p^u_n,e^{-2\epsilon/3}Q_\epsilon(x_{n+1})\}, \textrm{ by (b$_{n+1}$) with $i=0$}\\
&\geq e^{-\epsilon}\min\{e^{\epsilon} p^u_n,Q_\epsilon(x_{n+1})\}=e^{-\epsilon}p^u_{n+1}.
\end{align*}
It follows that $q^u_n\geq e^{-\epsilon}p^u_n$ for all $n\in\Z$.
Similarly $q^s_n\geq e^{-\epsilon}p^s_n$ for all $n\in\Z$, whence
$\eta_n\geq e^{-\epsilon}(p^u_n\wedge p^s_n)$ for all $n\in\Z$, giving us (A3).

\medskip
\noindent
{\em Proof that
$\{\Psi_{x_n}^{p^u_n,p^s_n}\}_{n\in\Z}$ is a gpo.\/} (GPO2) is true by construction,
so we just need to check (GPO1). We write  (a$_n$) with $i=1$, and (a$_{n+1}$) with $i=0$:
\begin{enumerate}[$\circ$]
\item $\dist_\Lambda(f^{n+1}(x),f(x_n))+\|\Theta\circ C_{\chi}(f^{n+1}(x))- \Theta\circ C_{\chi}(f(x_n))\|<e^{-8(m_n+3)}$.
\item $\dist_\Lambda(f^{n+1}(x),x_{n+1})+\|\Theta\circ C_{\chi}(f^{n+1}(x))- \Theta\circ C_{\chi}(x_{n+1})\|<e^{-8(m_{n+1}+3)}$.
\end{enumerate}
So $x_{n+1},f(x_n),f^{n+1}(x)$ are all in the same canonical transverse disc, and
\begin{equation}\label{Ineq_m_n}
\dist_\Lambda(f(x_n),x_{n+1})+\|\Theta\circ C_{\chi}(f(x_n))-\Theta\circ C_{\chi}(x_{n+1})\|<e^{-8(m_n+3)}+e^{-8(m_{n+1}+3)}.
\end{equation}
The proof of (A3) shows that  $\xi_{n}:=p^u_{n}\wedge p^s_{n}\in [e^{-m_{n}-2},e^{-m_{n}+2}]$.
Also $\xi_n/\xi_{n+1}\in [e^{-\epsilon},e^{\epsilon}]$, because $\{(p^u_n,p^s_n)\}_{n\in\Z}$ is
$\epsilon$--subordinated (see Lemma \ref{LemmaEasy}). So the right hand side of (\ref{Ineq_m_n})
is less than $e^{-8}(1+e^{8\epsilon})\xi_{n+1}^8<(p^u_{n+1}\wedge p^s_{n+1})^8$.
Thus $\Psi_{f(x_n)}^{p^u_{n+1}\wedge p^s_{n+1}}\overset{\epsilon}{\approx}\Psi_{x_{n+1}}^{p^u_{n+1}\wedge p^s_{n+1}}$.
A similar argument with ($\mathrm a_n$) and $i=-1$, and with ($\mathrm a_{n-1}$) and $i=0$ shows that
$\Psi_{f^{-1}(x_n)}^{p^u_{n-1}\wedge p^s_{n-1}}\overset{\epsilon}{\approx}\Psi_{x_{n-1}}^{p^u_{n-1}\wedge p^s_{n-1}}$.
So (GPO1) holds, and $\un{v}$ is a gpo.

\medskip
\noindent
{\em Proof that  $\un{v}$ shadows $x$\/:} By (a$_n$) with $i=0$,
$\Psi_{x_{n}}^{p^u_{n}\wedge p^s_{n}}\overset{\epsilon}{\approx}\Psi_{f^{n}(x)}^{p^u_{n}\wedge p^s_{n}}$
for all $n\in\Z$. By Proposition \ref{Prop_Overlap_Meaning},
$f^n(x)=\Psi_{f^n(x)}(\un{0})\in \Psi_{x_n}([-p^u_{n}\wedge p^s_{n},p^u_{n}\wedge p^s_{n}]^2)$.
So $\un{v}$ shadows $x$.

\medskip
\noindent
{\sc Arranging relevance:}
Call an element $v\in\mathfs A$ {\em relevant}, if there is a gpo $\un{v}\in\mathfs A^\Z$ s.t. $v_0=v$
and $\un{v}$ shadows a point in $\NUH_\chi(f)$. In this case every $v_i$ is relevant, because
$\NUH_\chi(f)$ is $f$--invariant. So $\mathfs A':=\{v\in\mathfs A: v\textrm{ is relevant}\}$ is sufficient.
It is discrete, because $\mathfs A'\subseteq\mathfs A$ and $\mathfs A$ is discrete.
The theorem follows with $\mathfs A'$.
\end{proof}

\subsection*{The inverse shadowing problem}

The same orbit can be shadowed by many different gpos. The ``inverse shadowing problem"
is to control the set of gpos $\{\Psi_{x_i}^{p^u_i,p^s_i}\}_{i\in\Z}$ which shadow the orbit of a
given point $x$. (GPO1) and (GPO2) were designed to make this possible. We need the following condition.

\medskip
\noindent
{\sc Regularity:\/} Let $\mathfs A$ be as in Proposition \ref{Prop_Coarse_Gr}.
A gpo $\un{v}\in\mathfs A^\Z$ is called {\em regular}, if $\{v_i\}_{i\geq 0}$, $\{v_i\}_{i\leq 0}$
have constant subsequences.

\begin{prop}\label{PropOnto}
Almost every $x\in\Lambda$ is shadowed by a regular gpo in $\mathfs A^\Z$.
\end{prop}

\begin{proof}
We will show that this is the case for all $x\in\NUH_\chi^\#(f)$.
Since $\mathfs A$ is sufficient (Proposition \ref{Prop_Coarse_Gr}(2)), for every $x\in\NUH_\chi^\#(f)$ there is a gpo $\un{v}=\{\Psi_{x_k}^{p^u_k,p^s_k}\}_{k\in\Z}\in\mathfs A^\Z$ which shadows $x$ s.t. for all $k\in\Z$,
$\eta_k:=p^u_k\wedge p^s_k\geq e^{-\epsilon/3}q_\epsilon(f^k(x))$.

Since $p^{u/s}\leq Q_\epsilon(\cdot)\leq \epsilon \dist_\Lambda(\cdot,\mathfrak S)$,
\begin{equation}\label{guitar}
\dist_\Lambda(x_k,\mathfrak S)\geq \epsilon^{-1}e^{-\epsilon/3}q_\epsilon(f^k(x))\text{ for all $k\in\Z$}.
\end{equation}
Since $v_k\xrightarrow{\epsilon}v_{k+1}$, $\Psi_{f(x_k)}^{\eta_{k+1}}\overset{\epsilon}\approx\Psi_{x_{k+1}}^{\eta_{k+1}}$,
whence $f(x_k)\in\Psi_{x_{k+1}}([-Q(x_{k+1}),Q(x_{k+1})]^2)$. Since $\Lip(\Psi_{x_{k+1}})\leq 2$,
$\dist_{\Lambda}(f(x_k),x_{k+1})\leq 2\sqrt{2}Q(x_{k+1})\leq 3\epsilon \dist_{\Lambda}(x_{k+1},\mathfrak S)$.
By the triangle inequality, relation (\ref{guitar}), and the inequality
$e^{-\epsilon/3}\leq q_\epsilon\circ f/q_{\epsilon}\leq e^{\epsilon/3}$,
\begin{align*}
&\dist_\Lambda(f(x_k),\mathfrak S)\geq \dist_\Lambda(x_{k+1},\mathfrak S)-
\dist_\Lambda(f(x_{k}),x_{k+1})\geq (1-3\epsilon)\dist_{\Lambda}(x_{k+1},\mathfrak S)\\
&\geq (1-3\epsilon)\epsilon^{-1}e^{-\epsilon/3}q_\epsilon(f^{k+1}(x))
>(1-3\epsilon)\epsilon^{-1}e^{-\epsilon}q_\epsilon(f^{k}(x))>q_{\epsilon}(f^k(x)),
\end{align*}
provided $\epsilon$ is small enough.
Similarly, $\dist_\Lambda(f^{-1}(x_k),\mathfrak S)>q_\epsilon(f^{k}(x))$, and we obtain
that $\min\{D(x_k),p^u_k,p^s_k\}\geq e^{-\epsilon/3}q_{\epsilon}(f^k(x))$ for all $k\in\Z$.

Since $x\in\NUH_\chi^\#(f)$, $\exists k_i,\ell_i\uparrow\infty$ and $c>0$ s.t.
$q_\epsilon(f^{-k_i}(x))\geq c$, $q_\epsilon(f^{\ell_i}(x))\geq c$. Since $\mathfs A$ is discrete,
there must be some constant subsequences $v_{-k_{i_j}}, v_{\ell_{i_j}}$.
\end{proof}

The next theorem says, in a precise way, that if $\un{u}$ is a regular gpo which shadows $x$,
then $u_i$ is determined ``up to bounded error". Together with the discreteness of $\mathfs A$,
this implies that for every $i$ there are only finitely many choices for
$u_i$.\footnote{But the set of all  possible full sequences $\un{u}$ can  be uncountable.}

\begin{thm}\label{Thm_Inverse}
The following holds for all $\epsilon$ small enough. Let $\un{u},\un{v}$ be regular gpos which shadow
the orbit of the same point $x$. If $u_i=\Psi_{x_i}^{p^u_i,p^s_i}$ and $v_i=\Psi_{y_i}^{q^u_i,q^s_i}$, then:
\begin{enumerate}[$(1)$]
\item $\dist_{\Lambda}(x_i,y_i)<10^{-1}\max\{p^u_i\wedge p^s_i,q^u_i\wedge q^s_i\}$.
\item ${\dist_{\Lambda}(f^k(x_i),\mathfrak S)}/{\dist_{\Lambda}(f^k(y_i),\mathfrak S)}\in [e^{-\sqrt{\epsilon}},e^{\sqrt{\epsilon}}]$ for $k=0,1,-1$.
\item $Q_{\epsilon}(x_i)/Q_{\epsilon}(y_i)\in [e^{-\sqrt[3]{\epsilon}},e^{\sqrt[3]{\epsilon}}]$.
\item $(\Psi_{y_i}^{-1}\circ\Psi_{x_i})=(-1)^{\sigma_i}\mathrm{Id}+\un{c}_i+\Delta_i\textrm{ on }[-\epsilon,\epsilon]^2$,
where $\|\un{c}_i\|<10^{-1}(q^u_i\wedge q^s_i)$, $\sigma_i\in\{0,1\}$ are constants,
and $\Delta_i:[-\epsilon,\epsilon]^2\to\R^2$ is a vector field s.t. $\Delta_i(\un{0})=\un{0}$ and
$\|(d\Delta_i)_{\un{v}}\|<\sqrt[3]{\epsilon}$ on $[-\epsilon,\epsilon]^2$.
\item $p^u_i/q^u_i, p^s_i/q^s_i\in [e^{-\sqrt[3]{\epsilon}},e^{\sqrt[3]{\epsilon}}]$.
\end{enumerate}
\end{thm}

\begin{proof}
Denote the unstable and stable manifolds of $\un{u}$ and $\un{v}$ by $U^u,U^s$ and $V^u,V^s$.
By the proof of the shadowing lemma,
$
U^u\cap U^s=V^u\cap V^s=\{x\}
$.

\medskip
\noindent
{\sc Part (1).\/}  $U^{u/s}$ are admissible manifolds. By  (Ad1--3), their intersection point must
satisfy $x=\Psi_{x_0}(\un{\xi})$ where $\|\un{\xi}\|_\infty\leq 10^{-2}(p^u_0\wedge p^s_0)$,
see \cite[Prop. 4.11]{Sarig-JAMS}. Since $\Lip(\Psi_{x_0})\leq 2$,
$\dist_{\Lambda}(x_0,x)\leq 50^{-1}(p^u_0\wedge p^s_0)$.
Similarly $\dist_{\Lambda}(y_0,x)\leq 50^{-1}(q^u_0\wedge q^s_0)$, whence
$\dist_\Lambda(x_0,y_0)\leq 25^{-1}\max\{p^u_0\wedge p^s_0,q^u_0\wedge q^s_0\}$.

\medskip
\noindent
{\sc Part (2).\/} In what follows $a=b\pm c$ means $b-c\leq a\leq b+c$.
\begin{align*}
&\dist_{\Lambda}(x_0,\mathfrak S)=\dist_{\Lambda}(x,\mathfrak S)\pm\dist_{\Lambda}(x,x_0)=\dist_{\Lambda}(x,\mathfrak S)\pm 50^{-1}(p^u_0\wedge p^s_0), \textrm{ by part 1}\\
&=\dist_{\Lambda}(x,\mathfrak S)\pm 50^{-1}Q_{\epsilon}(x_0),\ \ \textrm{because }p^u_0,p^s_0\leq Q_{\epsilon}(x_0)\\
&=\dist_{\Lambda}(x,\mathfrak S)\pm 50^{-1}\epsilon\dist_{\Lambda}(x_0,\mathfrak S),
\textrm{ by the definition of $Q_\epsilon(x_0)$}.
\end{align*}
Therefore $\frac{\dist_{\Lambda}(x,\mathfrak S)}{\dist_{\Lambda}(x_0,\mathfrak S)}= 1\pm 50^{-1}\epsilon$.
Similarly $\frac{\dist_{\Lambda}(x,\mathfrak S)}{\dist_{\Lambda}(y_0,\mathfrak S)}= 1\pm 50^{-1}\epsilon$.
It follows that if $\epsilon$ is small enough then
$\frac{\dist_{\Lambda}(x_0,\mathfrak S)}{\dist_{\Lambda}(y_0,\mathfrak S)}\in [e^{-{\epsilon}},e^{{\epsilon}}]$.
Applying this argument to suitable shifts of $\un{u}$ and $\un{v}$, we obtain
$\frac{\dist_{\Lambda}(x_{1},\mathfrak S)}{\dist_{\Lambda}(y_{1},\mathfrak S)}\in [e^{-{\epsilon}},e^{{\epsilon}}]$
and $\frac{\dist_{\Lambda}(x_{-1},\mathfrak S)}{\dist_{\Lambda}(y_{-1},\mathfrak S)}\in [e^{-{\epsilon}},e^{{\epsilon}}]$.

Since $u_0\epto u_1$,
$\Psi_{f(x_0)}^{p^u_1\wedge p^s_1}\overset{\epsilon}{\approx}\Psi_{x_1}^{p^u_1\wedge p^s_1}$.
So $x_1,f(x_0)\in \Psi_{x_1}\bigl([-Q_{\epsilon}(x_1),Q_{\epsilon}(x_1)]^2\bigr)$.
Since $\mathrm{Lip}(\Psi_{x_1})\leq 2$,
$\dist_{\Lambda}\bigl(x_1,f(x_0)\bigr)\leq 2\sqrt{2}Q_{\epsilon}(x_1)< 6\epsilon\dist_{\Lambda}(x_1,\mathfrak S)$.
Thus
\begin{align*}
&\dist_{\Lambda}(f(x_0),\mathfrak S)=\dist_{\Lambda}(x_1,\mathfrak S)\pm \dist_{\Lambda}(x_1,f(x_0))
=\dist_{\Lambda}(x_1,\mathfrak S)\pm 6\epsilon\dist_{\Lambda}(x_1,\mathfrak S)\\
&=e^{\pm 7\epsilon}\dist_{\Lambda}(x_1,\mathfrak S), \text{ provided $\epsilon$ is small enough.}
\end{align*}
Similarly $\dist_{\Lambda}(f(y_0),\mathfrak S)=e^{\pm 7\epsilon}\dist_{\Lambda}(y_1,\mathfrak S)$. Since $\frac{\dist_{\Lambda}(x_{1},\mathfrak S)}{\dist_{\Lambda}(y_{1},\mathfrak S)}\in [e^{-{\epsilon}},e^{{\epsilon}}]$,
$$\frac{\dist_{\Lambda}(f(x_0),\mathfrak S)}{\dist_{\Lambda}(f(y_0),\mathfrak S)}\in [e^{-15\epsilon},e^{15\epsilon}]\subset [e^{-\sqrt{\epsilon}},e^{\sqrt{\epsilon}}],\textrm{ provided $\epsilon$ is small enough.}
$$
Similarly $\frac{\dist_{\Lambda}(f^{-1}(x_0),\mathfrak S)}{\dist_{\Lambda}(f^{-1}(y_0),\mathfrak S)}\in[e^{-\sqrt{\epsilon}},e^{\sqrt{\epsilon}}]$.

\medskip
\noindent
{\sc Part (3).\/} One shows as in  \cite[\S 6 and \S 7]{Sarig-JAMS} that for all $\epsilon$  small enough,
\begin{equation}\label{polyphemus}
\frac{\sin\alpha(x_i)}{\sin\alpha(y_i)}\in [e^{-\sqrt{\epsilon}},e^{\sqrt{\epsilon}}]\ \ \textrm{ and }
\ \ \frac{s(x_i)}{s(y_i)}, \frac{u(x_i)}{u(y_i)}\in [e^{-4\sqrt{\epsilon}},e^{4\sqrt{\epsilon}}].
\end{equation}
The proof carries over without change, because all the calculations are done on $f^n(U^s), f^n(V^s)$ ($n\geq 0$)
and $f^{n}(U^u), f^{n}(V^u)$ $(n\leq 0)$, and these sets stay inside Pesin charts, away from $\mathfrak S$.
By (\ref{polyphemus}), for all $\epsilon$ small enough we have
$$
\biggl(\frac{\sqrt{s(x_0)^2+u(x_0)^2}}{|\sin\a(x_0)|}\biggr)^{-\frac{12}{\beta}}
\biggl(\frac{\sqrt{s(y_0)^2+u(y_0)^2}}{|\sin\a(y_0)|}\biggr)^{\frac{12}
{\beta}}\in [e^{-\sqrt[3]{\epsilon}},e^{-\sqrt[3]{\epsilon}}].
$$
By part (2) and the definition of $Q_\epsilon$,
$\frac{Q_{\epsilon}(x_0)}{Q_{\epsilon}(y_0)}\in [e^{-\sqrt[3]{\epsilon}},e^{\sqrt[3]{\epsilon}}]$.

\medskip
\noindent
{\sc Part (4).\/} This is done exactly as in \cite[\S 9]{Sarig-JAMS}, except that one needs to add the
constraint $\epsilon<\rho_{\dom}$ to be able to use the smoothness of $p\mapsto \Exp_p$ on $\Lambda$.

\medskip
\noindent
{\sc Part (5).\/} This is done exactly as in the proof of \cite[Prop. 8.3]{Sarig-JAMS}, except that step 1
there should be replaced by part (3) here.
\end{proof}

\noindent
{\em Remark.\/} Regularity is needed in parts (3), (4), (5). Parts (3), (4) also use the full force of
(Ad1--3), and Part (5) is based on (GPO2). See \cite{Sarig-JAMS}.

\section{Countable Markov partitions and symbolic dynamics}

Sina{\u \i} and Bowen gave several methods for constructing Markov partitions for uniformly
hyperbolic diffeomorphisms
\cite{Sinai-Construction-of-MP,Sinai-MP-U-diffeomorphisms,Bowen-MP-Axiom-A,Bowen-LNM}.
One of these constructions, due to Bowen, uses pseudo-orbits and
shadowing \cite{Bowen-LNM}. The theory of gpos we developed in the previous section allows us to
apply this method to adapted Poincar\'e sections. The result is a Markov partition for $f:\Lambda\to\Lambda$.
It is then a standard procedure to code $f:\Lambda\to\Lambda$ by a topological Markov shift,
and $\vf:M\to M$ by a topological Markov flow.

\subsection*{Step 1: A Markov extension}

Let $\mathfs A$ be the countable set of double charts we constructed in Proposition \ref{Prop_Coarse_Gr},
and let $\mathfs G$ denote the countable directed graph with set of vertices $\mathfs A$
and set of edges $\{(v,w)\in\mathfs A\x\mathfs A:v\epto w\}$.

\begin{lem}\label{Lemma_Finite_Degree}
Every vertex of $\mathfs G$ has finite ingoing degree, and finite outgoing degree.
\end{lem}

\begin{proof}
We fix $v\in\mathfs A$, and bound the number of $w$ s.t. $v\epto w$, using the
discreteness and relevance of $\mathfs A$ (cf. Prop. \ref{Prop_Coarse_Gr}).

By the relevance property, $v\epto w$ extends to a path $v\epto w\epto u$.
Write $v=\Psi_x^{p^u,p^s},w=\Psi_y^{q^u,q^s}$,  $u=\Psi_z^{r^u,r^s}$, then:
\begin{enumerate}[$\circ$]
\item $\frac{r^u\wedge r^s}{q^u\wedge q^s}, \frac{q^u\wedge q^s}{p^u\wedge p^s} \in [e^{-\epsilon},e^{\epsilon}]$,
by Lemma \ref{LemmaEasy}.
\item $\dist_\Lambda(y,\mathfrak S)\geq \epsilon^{-1} e^{-\epsilon}(p^u\wedge p^s)$,
because $\dist_\Lambda(y,\mathfrak S)\geq\epsilon^{-1}Q_\epsilon(y)$,  $Q_\epsilon(y)\geq q^u\wedge q^s$.
\item $\dist_\Lambda(f(y),\mathfrak S)\geq e^{-2\epsilon}(\epsilon^{-1}-3)(p^u\wedge p^s)$, because
\begin{align*}
&\dist_\Lambda(f(y),\mathfrak S)\geq \dist_{\Lambda}(z,\mathfrak S)-\dist_{\Lambda}(z,f(y))\\
&\geq \epsilon^{-1}Q_\epsilon(z)-2\sqrt{2}Q_\epsilon(z)
\hspace{0.3cm} \because f(y)\in \Psi_z([-Q_\epsilon(z),Q_\epsilon(z)]^2)\text{ and } \mathrm{Lip}(\Psi_z)\leq 2\\
&\geq (\epsilon^{-1}-3)(r^u\wedge r^s)\geq e^{-2\epsilon}(\epsilon^{-1}-3)(p^u\wedge p^s).
\end{align*}
\item $\dist_\Lambda(f^{-1}(y),\mathfrak S)\geq e^{-2\epsilon}(\epsilon^{-1}-3)(p^u\wedge p^s)$, for similar reasons.
\end{enumerate}
So $D(y):=\dist_\Lambda(\{y,f(y),f^{-1}(y)\},\mathfrak S)\geq t:=e^{-2\epsilon}(\epsilon^{-1}-3)(p^u\wedge p^s)$.

By the discreteness of $\mathfs A$ (and assuming $\epsilon<\frac{1}{3}$), $\#\{w\in\mathfs A:v\epto w\}<\infty$.
The finiteness of the ingoing degree is proved in the same way.
\end{proof}

\noindent
{\sc The Markov Extension:}
Let $\Sigma(\mathfs G)$ denote the set of two-sided paths on $\mathfs G$:
$$
\Sigma(\mathfs G):=\{\un{v}\in \mathfs A^\Z: v_i\epto v_{i+1}\textrm{ for all }i\in\Z\}.
$$
We equip $\Sigma(\mathfs G)$ with the metric $d(\un{u},\un{v})=\exp[-\min\{|n|:u_n\neq v_n\}]$,
and with the action of the {\em left shift map}
$
\sigma:\Sigma(\mathfs G)\to\Sigma(\mathfs G),\ \sigma:\{v_i\}_{i\in\Z}\mapsto \{v_{i+1}\}_{i\in\Z}.
$
The set $\Sigma(\mathfs G)$ is exactly the collection of gpos in $\mathfs A^\Z$,
hence $\pi:\Sigma(\mathfs G)\to \Lambda$ given by
$$
\pi(\un{v}):=\textrm{unique point whose $f$--orbit is shadowed by $\un{v}$}
$$
is well-defined. Necessarily $f\circ\pi=\pi\circ\sigma$, so $\sigma:\Sigma(\mathfs G)\to\Sigma(\mathfs G)$
is an extension of $f:\Lambda\to\Lambda$ (at least on a subset of full measure, by Prop. \ref{PropOnto}).

It is easy to see, using the finite degree of the vertices of $\mathfs G$, that $(\Sigma(\mathfs G),d)$
is a locally compact, complete and separable metric space. The left shift map is a bi-Lipschitz homeomorphism.
The subset of {\em regular gpos}
$$
\Sigma^\#(\mathfs G):=\{\un{v}\in\Sigma(\mathfs G): \{v_i\}_{i\leq 0},\{v_i\}_{i\geq 0}\textrm{ contain constant subsequences}\}
$$
has full measure with respect to every $\sigma$--invariant Borel probability measure.

As we saw in the proof of the shadowing lemma (Theorem \ref{Thm_Shadowing}),
$\pi(\un{v})$ is the unique intersection of $V^u(\un{v}^-)$ and $V^s(\un{v}^+)$ where
$\un{v}^\pm$ are the half gpos determined by $\un{v}$. The proof shows that the
following holds for all $\epsilon$ small enough:
\begin{enumerate}[(1)]
\item {\sc H\"older continuity:} $\pi$ is H\"older continuous (because $\mathfs F^u, \mathfs F^s$
are contractions, see \cite[Thm 4.16(2)]{Sarig-JAMS}).
\item {\sc Almost surjectivity:} $\mu_\Lambda(\Lambda\setminus\pi[\Sigma^\#(\mathfs G)])=0$ (Proposition \ref{PropOnto}).
\item {\sc Inverse Property:} for all $x\in \Lambda, i\in\Z$,
$\#\{v_i: \un{v}\in\Sigma^\#(\mathfs G),\pi(\un{v})=x\}<\infty.$
(Theorem \ref{Thm_Inverse} and the discreteness of $\mathfs A$.)
\end{enumerate}

The inverse property does {\em not} imply that $\pi$ is finite-to-one or even countable-to-one.
The following steps will lead us to an a.e. {\em finite-to-one} Markov extension.

\subsection*{Step 2: A Markov cover}

Given $v\in\mathfs A$, let $_0[{v}]:=\{\un{v}\in\Sigma(\mathfs G):v_0=v\}$.
This is a partition of $\Sigma(\mathfs G)$. The projection to $\Lambda$,
$$
\mathfs Z:=\{Z(v):v\in\mathfs A\},\textrm{ where }Z(v):=\{\pi(\un{v}):\un{v}\in\Sigma^\#(\mathfs G),\ v_0=v\},
$$
is not a partition. It could even be the case that $Z(v)=Z(w)$ for $v\neq w$
(in this case, we agree to think of $Z(v),Z(w)$ as different elements of $\mathfs Z$).
Here are some important properties of $\mathfs Z$.

\medskip
\noindent
{\sc Covering property:\/}\label{covering-property} {\em $\mathfs Z$ covers a set of full $\mu_\Lambda$--measure.}

\medskip
\noindent
{\em Proof.\/} $\mathfs Z$ covers $\NUH_\chi^\#(f)$.

\medskip
\noindent
{\sc Local finiteness:} {\em
For all $Z\in\mathfs Z$,  $\#\{Z'\in\mathfs Z:Z'\cap Z\neq\emptyset\}<\infty$.
Even better: $\#\{v\in\mathfs A: Z(v)\cap Z\neq\emptyset\}<\infty$ for all $Z\in\mathfs Z$.}

\medskip
\noindent
{\em Proof.\/} Write $Z=Z(\Psi_x^{p^u,p^s})$. If $Z(\Psi_y^{q^u,q^s})\cap Z\neq \emptyset$,
then $q^u\wedge q^s\geq e^{-\sqrt[3]{\epsilon}}(p^u\wedge p^s)$ and
$\dist_\Lambda(\{y,f(y),f^{-1}(y)\},\mathfrak S)\geq e^{-\sqrt{\epsilon}}\dist_\Lambda(\{x,f(x),f^{-1}(x)\},\mathfrak S)$
(Theorem \ref{Thm_Inverse}). Since $\mathfs A$ is discrete, there are only finitely many such
$\Psi_y^{q^u,q^s}$ in $\mathfs A$.

\medskip
\noindent
{\sc Product structure:} {\em Suppose $v\in\mathfs A$ and $Z=Z(v)$.
For every $x\in Z$ there are sets $W^u(x,Z)$ and $W^s(x,Z)$ called the {\em $s$--fibre}
and {\em $u$--fibre of $x$ in $Z$} s.t.:
\begin{enumerate}[$(1)$]
\item $Z=\bigcup_{x\in Z}W^u(x,Z)$, $Z=\bigcup_{x\in Z}W^s(x,Z)$.
\item Any two $s$--fibres in $Z$ are either equal or disjoint, and the same for $u$--fibres.
\item For every $x,y\in Z$, $W^u(x,Z)\cap W^s(y,Z)$ consists of a single point.
\end{enumerate}}
\noindent
{\em Notation: $[x,y]_Z:=$ unique point in $W^u(x,Z)\cap W^s(x,Z)$.}

\medskip
\noindent
{\em Proof.\/} Recall from \S\ref{Section_GPO} the notation for the stable and unstable manifolds
of positive and negative gpos. Fix $Z=Z(v)$ in $\mathfs Z$, $x\in Z$, and let:
\begin{enumerate}[$\circ$]
\item $V^s(x,Z):=V^s[\{v_i\}_{i\geq 0}]\textrm{ for some (any) }\un{v}\in\Sigma^\#(\mathfs G)\textrm{ s.t. }v_0=v\textrm{ and }\pi(\un{v})=x$.
\item $V^u(x,Z):=V^u[\{v_i\}_{i\leq 0}]\textrm{ for some (any) }\un{v}\in\Sigma^\#(\mathfs G)\textrm{ s.t. }v_0=v\textrm{ and }\pi(\un{v})=x$.
\item $W^s(x,Z):=V^s(x,Z)\cap Z$.
\item $W^u(x,Z):=V^u(x,Z)\cap Z$.
\end{enumerate}
To see that the definition is proper, suppose $\un{u},\un{v}$ are two regular gpos such that $u_0=v_0=v$.
If $V^{t}[\un{u}],V^t[\un{v}]$ intersect at some point $z$ for $t=s$ or $u$, then  $V^{t}[\un{u}]=V^t[\un{v}]$,
because both are equal to the piece of the  local stable/unstable manifold of $z$ at
$\Psi_{x(v)}([-p^t(v),p^t(v)]^2)$.  See \cite[Prop. 6.4]{Sarig-JAMS} for details.
In particular, $\pi(\un{u})=\pi(\un{v})\Rightarrow V^t[\un{u}]=V^t[\un{v}]$ for $t=u,s$.

This argument also shows that any two $t$--fibres $(t=s$ or $u$) are equal or disjoint,
hence (2) holds. (1) is because $W^u(x,Z)$, $W^s(x,Z)$ both contain $x$.
For (3), write $W^u(x,Z)=V^u[\un{u}]\cap Z$ and $W^s(y,Z)=V^s[\un{v}]\cap Z$
where $\un{u},\un{v}\in\Sigma^\#(\mathfs G)$ satisfy $u_0=v_0=v$.
Let $\un{w}:=(\ldots,u_{-2},u_{-1},\dot{v},v_1,v_2,\ldots)$ with the dot indicating the zeroth coordinate.
Clearly $\pi(\un{w})\in W^u(x,Z)\cap W^s(y,Z)$. Since $W^u(x,Z)\cap W^s(y,Z)\subset V^u(x,Z)\cap V^s(y,Z)$
and a $u$--admissible manifold intersects an  $s$--admissible at most
once \cite[Cor. S.3.8]{Katok-Hasselblatt-Book}, \cite[Prop. 4.11]{Sarig-JAMS},
$W^u(x,Z)\cap W^s(y,Z)=\{\pi(\un{w})\}$.


\medskip
\noindent
{\sc Symbolic Markov property:} {\em If $x=\pi(\un{v})$ with $\un{v}\in\Sigma^\#(\mathfs G)$, then}
$$
f\bigl[W^s(x,Z(v_0))\bigr]\subset W^s(f(x),Z(v_1))\textrm{ and }f^{-1}\bigl[W^u(f(x),Z(v_1))\bigr]\subset W^u(x,Z(v_0)).
$$

\medskip
\noindent
{\em Proof.\/} Fix $y\in W^s(x,Z(v_0))$. Choose $\un{u}\in\Sigma^\#(\mathfs G)$ s.t.
$u_0=v_0$ and $y=\pi(\un{u})$. Write $u_i=\Psi_{y_i}^{q^u_i,q^s_i}$ and $\eta_i:=q^u_i\wedge q^s_i$, then
$
f^{-n}(y)\in \Psi_{y_{-n}}([-\eta_{-n}, \eta_{-n}]^2)\textrm{ for all }n\geq 0.
$
Write $v_i=\Psi_{x_i}^{p^u_i,p^s_i}$ and $\xi_i:=p^u_i\wedge p^s_i$.
Since $y\in W^s(x,Z(v_0))\subset V^s[\un{v}^+]$,
$f^n(y)\in f^n(V^s[\un{v}^+])\subset V^s[\sigma^n\un{v}^+]\subset \Psi_{x_{n}}([-\xi_n, \xi_n]^2)\textrm{ for all }n\geq 0$.
It follows that the gpo
$\un{w}=(\ldots,u_{-2},u_{-1},\dot{v},v_1,v_2,\ldots)
$ shadows $y$ (the dot indicates the position of the zeroth coordinate).
Necessarily, $f(y)=f[\pi(\un{w})]=\pi[\sigma(\un{w})]\in V^s[\{v_i\}_{i\geq 1}]\cap Z(v_1)=W^s(f(x),Z(v_1))$.
Thus $f(y)\in W^s(f(x),Z(v_1))$.

Since $y\in W^s(x,Z(v_0))$ was arbitrary, $f[W^s(x,Z(v_0))]\subset W^s(f(x),Z(v_1))$.
The other inequality is  symmetric.

\medskip
\noindent
{\sc Overlapping charts property:} {\em The following holds for all $\epsilon$ small enough.
Suppose $Z,Z'\in\mathfs Z$ and $Z\cap Z'\neq\emptyset$.
\begin{enumerate}[$(1)$]
\item If $Z=Z(\Psi_{x_0}^{p^u_0,p^s_0}),Z'=Z(\Psi_{y_0}^{q^u_0,q^s_0})$, then
$Z\subset \Psi_{y_0}([-(q^u_0\wedge q^s_0),(q^u_0\wedge q^s_0)]^2)$.
\item For all $x\in Z,y\in Z'$, $V^u(x,Z)$ intersects $V^s(y,Z')$ at a unique point.
\item For any $x\in Z\cap Z'$, $W^u(x,Z)\subset V^u(x,Z')$ and $W^s(x,Z)\subset V^s(x,Z')$.
\end{enumerate}
}

\medskip
\noindent
{\em Sketch of proof.\/} If $Z\cap Z'\neq \emptyset$, then there are $\un{u},\un{v}\in\Sigma^\#(\mathfs G)$
s.t. $u_0=\Psi_{x_0}^{p^u_0,p^s_0}$, $v_0=\Psi_{y_0}^{q^u_0,q^s_0}$, and $\pi(\un{u})=\pi(\un{v})$.
By Theorem \ref{Thm_Inverse}(4), $\Psi_{y_0}^{-1}\circ \Psi_{x_0}$ is close to $\pm\mathrm{Id}$.
This is enough to prove (1)--(3), see Lemmas 10.8 and 10.10 in \cite{Sarig-JAMS} for details.

\subsection*{Step 3 (Bowen, Sina{\u \i}): A countable Markov partition}

We refine $\mathfs Z$ into a partition without destroying the Markov property or the product structure.
The refinement procedure we use below is due to Bowen \cite{Bowen-LNM}, building on earlier
work of Sina{\u \i} \cite{Sinai-Construction-of-MP,Sinai-MP-U-diffeomorphisms}.
It was designed for finite Markov covers, but works equally well for locally finite infinite covers.
Local finiteness is essential: a general non-locally finite cover may not have a countable refining
partition as can be seen in the example of the cover $\{(\alpha,\beta):\alpha,\beta\in\Q\}$ of $\R$.

Enumerate $\mathfs Z=\{Z_i:i\in\N\}$. For every $Z_i,Z_j\in\mathfs Z$ s.t. $Z_i\cap Z_j\neq\emptyset$, let
\begin{align*}
T^{us}_{ij}&:=\{x\in Z_i:W^u(x,Z_i)\cap Z_j\neq\emptyset\ ,\ W^s(x,Z_i)\cap Z_j\neq\emptyset\},\\
T^{u\emptyset}_{ij}&:=\{x\in Z_i:W^u(x,Z_i)\cap Z_j\neq\emptyset\ ,\ W^s(x,Z_i)\cap Z_j=\emptyset\},\\
T^{\emptyset s}_{ij}&:=\{x\in Z_i:W^u(x,Z_i)\cap Z_j=\emptyset\ ,\ W^s(x,Z_i)\cap Z_j\neq\emptyset\},\\
T^{\emptyset\emptyset}_{ij}&:=\{x\in Z_i:W^u(x,Z_i)\cap Z_j=\emptyset\ ,\ W^s(x,Z_i)\cap Z_j=\emptyset\}.
\end{align*}
This is a partition of $Z_i$.
Let $\mathfs T:=\bigl\{T^{\alpha\beta}_{ij}:i,j\in\N, \alpha\in\{u,\emptyset\}, \beta\in\{s,\emptyset\}\bigr\}$.
This is a countable set, and $\mathfs T\supset \mathfs Z$ (since $T^{us}_{ii}=Z_i$, $\forall i$).
Necessarily, $\mathfs T$ covers $\NUH_{\chi}^\#(f)$.

\medskip
\noindent
{\sc The Markov partition:}
$\mathfs R$:= the collection of sets which can be put in the form
$
R(x):=\bigcap\{T\in\mathfs T:T\owns x\}
$
for some $x\in \bigcup_{i\geq 1}Z_i$.

\begin{prop}\label{prop-Z-and-R}
$\mathfs R$ is a countable pairwise disjoint cover of $\NUH_{\chi}^\#(f)$.
It refines $\mathfs Z$, and every element of $\mathfs Z$ contains only finitely
many elements of $\mathfs R$.
\end{prop}

\begin{proof}
See \cite{Bowen-LNM} or \cite[Prop. 11.2]{Sarig-JAMS}. The local finiteness of $\mathfs Z$
is needed to show that $\mathfs R$ is countable: it implies that $\#\{T\in\mathfs T:T\owns x\}<\infty$ for all $x$.
%
%
\end{proof}

The following proposition says that $\mathfs R$ is a {\em Markov partition}
in the sense of Sina{\u \i} \cite{Sinai-MP-U-diffeomorphisms}. First, some definitions.
The {\em $u$--fibre} and {\em $s$--fibre} of $x\in R\in\mathfs R$ are
\begin{align*}
W^u(x,R)&:=\bigcap\{W^u(x,Z_i)\cap T^{\alpha\beta}_{ij}:T^{\alpha\beta}_{ij}\in\mathfs T\textrm{ contains }R\},\\
W^s(x,R)&:=\bigcap\{W^s(x,Z_i)\cap T^{\alpha\beta}_{ij}:T^{\alpha\beta}_{ij}\in\mathfs T\textrm{ contains }R\}.
\end{align*}

\begin{prop} The following properties hold.
\begin{enumerate}[$(1)$]
\item {\sc Product structure:\/} Suppose $R\in\mathfs R$.
\begin{enumerate}[{\rm (a)}]
\item If $x\in R$, then the $s$ and $u$ fibres of $x$ contain $x$, and are contained in $R$,
therefore $R=\bigcup_{x\in R}W^u(x,R)$ and $R=\bigcup_{x\in R}W^s(x,R)$.
\item For all $x,y\in R$, either the $u$--fibres of $x,y$ in $R$ are equal, or they are disjoint. Similarly for $s$--fibres.
\item For all $x,y\in R$, $W^u(x,R)$ and $W^s(y,R)$ intersect at a unique point,
denoted by $[x,y]$ and called the {\em Smale bracket of $x,y$}.  
\end{enumerate}
\item {\sc Hyperbolicity:} For all $z_1,z_2\in W^s(x,R)$, $\dist_{\Lambda}(f^n(z_1),f^n(z_2))\xrightarrow[n\to\infty]{}0$,
and for all $z_1,z_2\in W^u(x,R)$, $\dist_{\Lambda}(f^{-n}(z_1),f^{-n}(z_2))\xrightarrow[n\to\infty]{}0$.
The rates are exponential.
\item {\sc Markov property:} Let $R_0,R_1\in\mathfs R$. If $x\in R_0$ and $f(x)\in R_1$,
then $f[W^s(x,R_0)]\subset W^s(f(x),R_1)$ and $f^{-1}[W^u(f(x),R_1)]\subset W^u(x,R_0)$.
\end{enumerate}
\end{prop}

\begin{proof} This follows from the Markov properties of $\mathfs Z$ as in \cite{Bowen-LNM}.
See \cite[Prop. 11.5--11.7]{Sarig-JAMS} for a proof using the notation of this paper.
%
%
%
%
%
%
\end{proof}

\subsection*{Step 4: Symbolic coding for $f:\Lambda\to\Lambda$ \cite{Adler-Weiss-PNAS,Sinai-MP-U-diffeomorphisms}}

Let $\mathfs R$ denote the partition we constructed in the previous section.
Suppose $R,S\in\mathfs R$. We say that {\em $R$ connects to $S$}, and write $R\to S$,
whenever $\exists x\in R$ s.t. $f(x)\in S$. Equivalently, $R\to S$ iff $R\cap f^{-1}(S)\neq \emptyset$.

\noindent
\medskip
{\sc The dynamical graph of $\mathfs R$:} This is the directed graph $\widehat{\mathfs G}$
with set of vertices $\mathfs R$ and set of edges $\{(R,S)\in\mathfs R\x\mathfs R: R\to S \}.$

\medskip
\noindent
{\sc Fundamental observation \cite{Adler-Weiss-PNAS,Sinai-MP-U-diffeomorphisms}:}
Suppose $m\leq n$ are integers, and $R_m\to\cdots\to R_n$ is a finite path on $\widehat{\mathfs G}$, then
$$
_\ell[R_m,\ldots,R_n]:=f^{-\ell}(R_m)\cap f^{-\ell-1}(R_{m+1})\cap\cdots\cap f^{-\ell-(n-m)}(R_n)\neq \emptyset.
$$

{\em Proof.\/} This can be seen by induction on $n-m$ as follows: If $n-m=0$ or $1$ there is nothing to prove.
Assume by induction that the statement holds for $m-n$, then $\exists x\in {_\ell[}R_m,\ldots,R_n]$
and $\exists y\in R_n$ s.t. $f(y)\in R_{n+1}$. Let $z:=[f^n(x),y]$,
then $f^{-n}(z)\in  {_\ell[}R_m,\ldots,R_{n+1}]$ by the Markov property. \hfill $\square$


\medskip
The sets $_\ell[R_m,\ldots,R_n]$ can be related to cylinders in $\Sigma^\#(\mathfs G)$ as follows.
Define for every path $v_m\to\cdots\to v_n$ on $\mathfs G$ (not $\widehat{\mathfs G}$) the set
$$
Z_\ell(v_m,\ldots,v_n):=\{\pi(\un{u}):\un{u}\in\Sigma^\#(\mathfs G), u_i=v_i\text{ for }i=\ell,\ldots,\ell+n-m\}.
$$

\begin{lem}\label{Lem_R_Z}
For all doubly infinite path $\cdots\to R_0\to R_{1}\to\cdots$ on $\widehat{\mathfs G}$ there
is a gpo $\un{v}\in\Sigma(\mathfs G)$ s.t. for every $n$, $R_n\subset Z(v_n)$ and
$_{-n}[R_{-n},\ldots,R_n]\subset Z_{-n}(R_{-n},\ldots,R_n)$.
\end{lem}

The proof proceeds as follows: for each $n\geq 0$ take $x_n\in _{-n}[R_{-n},\ldots,R_n]$,
write $x_n=\pi(\un v^{(n)})$ for $\un v^{(n)}\in\Sigma^\#(\mathfs G)$, and then apply a
diagonal argument to construct $\un v$. See \cite[Lemma 12.2]{Sarig-JAMS} for the details.

\begin{prop}
Every vertex of $\wh{\mathfs G}$ has finite outgoing and ingoing degrees.
\end{prop}

\begin{proof}
Fix $R_0\in\mathfs R$. For every  path $R_{-1}\to R_0\to R_1$ in $\widehat{\mathfs G}$,
find a path $v_{-1}\to v_0\to v_1$ in $\mathfs G$ s.t. $Z(v_i)\supset R_i$ for $|i|\leq 1$.
Since $\mathfs Z$ is locally finite, there are finitely many possibilities for $v_0$. Since
every vertex of $\mathfs G$ has finite degree, there are also only finitely many possibilities
for $v_{-1},v_1$. Since every element in $\mathfs Z$ contains at most finitely many elements
in $\mathfs R$, there is a finite number of possibilities for $R_{-1},R_1$.
\end{proof}

Let $\Sigma(\wh{\mathfs G}):=\{\textrm{doubly infinite paths on $\wh{\mathfs G}$}\}
=\{\un{R}\in \mathfs R^\Z:R_i\to R_{i+1},\forall i\}$, equipped with the metric
$d(\un{R},\un{S}):=\exp[-\min\{|i|:R_i\neq S_i\}]$, and the action of the left shift map
$\sigma:\Sigma(\wh{\mathfs G})\to \Sigma(\wh{\mathfs G})$, $\sigma(\un{R})_i=R_{i+1}$. Let
$$
{\Sigma}^\#(\widehat{\mathfs G}):=\left\{\un{R}\in\Sigma(\widehat{\mathfs G}):
\{R_i\}_{i\leq 0}, \{R_i\}_{i\geq 0}\textrm{ contain constant subsequences} \right\}.
$$

Since $\pi:\Sigma(\mathfs G)\to\Lambda$ is H\"older continuous, there are constants
$C>0,\theta\in (0,1)$ s.t. for every finite path $v_{-n}\to\cdots\to v_n$ on $\mathfs G$,
$\diam(Z_{-n}(v_{-n},\ldots,v_n))\leq C\theta^n$. By  Lemma \ref{Lem_R_Z},
$\diam(_{-n}[R_{-n},\ldots,R_n])\leq C\theta^n$ for every finite path $R_{-n}\to\cdots\to R_n$
on $\widehat{\mathfs G}$. This allows us to make the following definition.

\medskip
\noindent
{\sc Symbolic dynamics for $f$:\/} Let
 $\wh{\pi}:\Sigma(\wh{\mathfs G})\to\Lambda$ be defined by
 $$\wh{\pi}(\un{R}):=\textrm{ The unique point in }\bigcap_{n=0}^\infty\ov{ _{-n}[R_{-n},\ldots,R_n]}.$$

\begin{thm}\label{Thm_Symbolic_Dynamics_For_f}
The following holds for all $\e$ small enough.
\begin{enumerate}[$(1)$]
\item $\wh{\pi}\circ\sigma=f\circ\wh{\pi}$.
\item $\wh{\pi}:\Sigma(\wh{\mathfs G})\to\Lambda$ is H\"older continuous.
\item $\wh{\pi}[\Sigma^\#(\wh{\mathfs G})]$ has full $\mu_{\Lambda}$--measure.
\item Every $x\in \widehat{\pi}[\Sigma^\#(\widehat{\mathfs G})]$ has finitely many pre-images
in $\Sigma^\#(\widehat{\mathfs G})$. If $\mu$ is ergodic, this number is equal a.e. to a constant.
\item Moreover, there exists $N:\mathfs R\x\mathfs R\to\N$ s.t. if $x=\widehat{\pi}(\un{R})$ where
$R_i=R$ for infinitely many $i<0$ and $R_i=S$ for infinitely many $i>0$, then
$\#\{\un{S}\in\Sigma^\#(\widehat{\mathfs G}):\pi(\un{S})=x\}\leq N(R,S)$.
\end{enumerate}
\end{thm}

\begin{proof}
(1) If $\un{R}\in\Sigma(\wh{\mathfs G})$, then $\wh{\pi}[\sigma(\un{R})]=f[\wh{\pi}(\un{R})]$:
\begin{align*}
\{\widehat{\pi}[\sigma(\un{R})]\}&=\bigcap_{n\geq 0}\overline{_{-n}[R_{-n+1},\ldots,R_{n+1}]}\supseteq \bigcap_{n\geq 0}\overline{_{-(n+2)}[R_{-n-1},\ldots,R_{n+1}]}
\end{align*}
\begin{align*}
&\overset{!}{=} \bigcap_{n\geq 0} \overline{f\bigl(\,_{-n-1}[R_{-n-1},\ldots,R_{n+1}]\,\bigr)} \overset{!}\supset\bigcap_{n\geq 0} f\bigl(\,\overline{_{-n-1}[R_{-n-1},\ldots,R_{n+1}]}\,\bigr)\\
&\overset{!}{=}f\left(\bigcap_{n\geq 0} \overline{_{-n-1}[R_{-n-1},\ldots,R_{n+1}]}\,\right)=\{f[\widehat{\pi}(\un{R})]\}.
\end{align*}
The equalities $\overset{!}{=}$ are because $f$ is invertible. To justify $\overset{!}{\supset}$,
it is enough to show that $f$ is continuous on an open neighborhood of
$C:=\overline{_{-n-1}[R_{-n-1},\ldots,R_{n+1}]}$. Fix some $v_0=\Psi_{x_0}^{p^u_0,p^s_0}$
s.t. $Z(v_0)\supset R_0$, then
$C\subset \overline{R_0}\subset \ov{Z(v_0)}\subset \Psi_{x_0}([-Q_\epsilon(x_0),Q_{\epsilon}(x_0)]^2)
\subset \Lambda\setminus\mathfrak S$. So $f$ is continuous on $C$.

\medskip
\noindent
(2) is because of the inequality $\diam(_{-n}[R_{-n},\ldots,R_n])\leq C\theta^n$ mentioned above.

\medskip
\noindent
(3) is because for every $x\in \NUH_{\chi}^\#(f)$, $x=\wh{\pi}(\un{R})$ where $R_i:=$ unique
element of $\mathfs R$ which contains $f^i(x)$. Clearly $\un{R}\in \Sigma(\wh{\mathfs G})$.
To see that $\un{R}\in \Sigma^\#(\wh{\mathfs G})$, we use Lemma \ref{Lem_R_Z} to construct
a regular gpo $\un{v}\in \Sigma^\#(\mathfs G)$ s.t. $x=\pi(\un{v})$, with $R_i\subset Z(v_i)$.
Every $Z(v)$ contains at most finitely many elements of $\mathfs R$ (Proposition \ref{prop-Z-and-R}).
Therefore, the regularity of $\un{v}$ implies the regularity of $\un{R}$.

\medskip
\noindent
(4) follows from (5) and the $f$--invariance of
$x\mapsto \#\{\un{R}\in\Sigma^\#(\wh{\mathfs G}): \widehat{\pi}(\un{R})=x\}$.

\medskip
\noindent
(5) is proved using Bowen's method \cite[pp. 13--14]{Bowen-Regional-Conference}, see also
\cite[p. 229]{Parry-Pollicott-Asterisque}. The proof is the same as in \cite{Sarig-JAMS}, but since
the presentation there has an error, we decided to include the complete details in the appendix.
\end{proof}

The next lemma is used in \cite{Ledrappier-Lima-Sarig}.
Recall from Lemma \ref{Lemma_Lyap_Exp} that there is a set $\Lambda_\chi^*$ of
full $\mu_\Lambda$--measure s.t. every $x\in\Lambda_\chi^*$ has tangent unit vectors
$\vec{v}^s_x,\vec{v}^u_x\in T_x\Lambda$ s.t.
$\lim_{n\to\infty}\frac{1}{n}\log\|df^n_x\vec{v}^s_x\|_{f^n(x)}<-\chi$ and
$\lim_{n\to\infty}\frac{1}{n}\log\|df^n_x\vec{v}^u_x\|_{f^n(x)}>\chi$.
The maps $x\in\Lambda_\chi^*\mapsto \vec{v}^s_x,\vec{v}^u_x$ are not necessarily
H\"older continuous with respect to the Riemannian metric (they may not even be globally defined).
But the symbolic metric is so much stronger than the Riemannian metric that the following  holds.

\begin{lem}\label{Lemma-Holder-property-stable-direction}
The maps $\un{R}\in\Sigma(\wh{\mathfs G})\mapsto \vec{v}^s_{\wh\pi(\un R)},\vec{v}^u_{\wh\pi(\un R)}$
are H\"older continuous with respect to the symbolic metric.
\end{lem}

\noindent
Lemma \ref{Lemma-Holder-property-stable-direction}
is a version of \cite[Prop. 12.6]{Sarig-JAMS} in our setup, and is proved similarly.

\subsection*{Step 5: Symbolic coding for $\vf:M\to M$}\label{section-coding-flow}

Let $\widehat{\pi}:\Sigma(\widehat{\mathfs G})\to \Lambda$ be the symbolic coding for
$f:\Lambda\to\Lambda$ given by Theorem \ref{Thm_Symbolic_Dynamics_For_f}.
Recall that $R:\Lambda\to(0,\infty)$ denotes the roof function of $\Lambda$. By the choice
of $\Lambda$, $R$ is bounded away from zero and infinity, and there is a global constant
$\mathfrak C$ s.t. $\sup_{x\in\Lambda\setminus\mathfrak S}\|dR_x\|<\mathfrak C$,
see Lemma \ref{Lemma_Smooth_Section}. Let
$$
r:\Sigma(\widehat{\mathfs G})\to (0,\infty),\  r:=R\circ\wh{\pi}.
$$
This function is also bounded away from zero and infinity, and since
$\wh{\pi}:\Sigma(\widehat{\mathfs G})\to\Lambda$ is H\"older and Pesin charts are connected
subsets of $\Lambda\setminus\mathfrak S$, $r$ is H\"older continuous.

Let $\sigma_r:\wh{\Sigma}_r\to\wh{\Sigma}_r$ denote the topological Markov flow with roof function
$r$ and base map $\sigma:\Sigma(\wh{\mathfs G})\to \Sigma(\wh{\mathfs G})$
(see page \pageref{TopMarkovFlowDefiPage} for definition).
Recall that the {\em regular part} of $\wh{\Sigma}_r$ is
$
\wh{\Sigma}_r^\#:=\{(\un{x},t): \un{x}\in \Sigma^\#(\widehat{\mathfs G}),\ 0\leq t<r(\un{x})\}
$.
This is a $\sigma_r$--invariant set, which contains all the periodic orbits of $\sigma_r$.
By the Poincar\'e recurrence theorem, $\wh{\Sigma}_r^\#$ has full measure with respect to
every $\sigma_r$--invariant probability measure. Let
$$
\wh{\pi}_r:\wh{\Sigma}_r\to M\,,\ \wh{\pi}_r(\un{x},t):=\vf^t[\widehat{\pi}(\un{x})].
$$
The following claims follow directly from Theorem \ref{Thm_Symbolic_Dynamics_For_f}:
\begin{enumerate}[$(1)$]
\item  $\wh{\pi}_r\circ\sigma^t_r=\vf^t\circ\wh{\pi}_r$ for all $t\in\R$.
\item $\wh{\pi}_r[\wh{\Sigma}^\#_r]$ has full measure with respect to $\mu$.
\item Every $p\in \wh{\pi}_r[\wh{\Sigma}^\#_r]$ has finitely many pre-images in $\wh{\Sigma}_r^\#$.
In case $\mu$ is ergodic, $p\mapsto \#(\wh{\pi}_r^{-1}(p)\cap\wh{\Sigma}^\#_r)$ is $\vf$--invariant,
whence constant almost everywhere.
\item Moreover, there exists $N:\mathfs R\x\mathfs R\to\N$ s.t. if $p=\wh{\pi}_r(\un{x},t)$
where $x_i=R$ for infinitely many $i<0$ and $x_i=S$ for infinitely many $i>0$, then
$\#\{(\un{y},s)\in\wh{\Sigma}^\#_r:\wh{\pi}_r(\un{y},s)=p\}\leq N(R,S)$.
\end{enumerate}
This proves all parts of Theorem \ref{ThmSymbolicDynamics}, except for the H\"older continuity of $\wh{\pi}_r$.

\subsection*{Step 6: H\"older continuity of $\wh{\pi}_r$}

Every topological Markov flow is continuous with respect to a natural metric, introduced by Bowen and Walters.
We will show that $\wh{\pi}_r:\wh{\Sigma}_r\to M$ is H\"older continuous with respect to this metric.
First we recall the definition of the Bowen-Walters metric.
Let $\sigma_r:\Sigma_r\to\Sigma_r$ denote a general topological Markov flow
(cf. page \pageref{TopMarkovFlowDefiPage}).
Suppose first that $r\equiv 1$ (constant suspension). Let $\psi:\Sigma_1\to\Sigma_1$ be the
suspension flow, and make the following definitions \cite{Bowen-Walters-Metric}:
\begin{enumerate}[$\circ$]
\item {\em Horizontal segments:} ordered pairs $[z,w]_h\in \Sigma_1\x\Sigma_1$ where
$z=(\un{x},t)$ and $w=(\un{y},t)$ have the same height $0\leq t<1$.
The {\em length} of a horizontal segment $[z,w]_h$ is defined to be
$\ell([z,w]_h):=(1-t)d(\un{x},\un{y})+td(\sigma(\un{x}),\sigma(\un{y}))$,
where $d$ is the metric on $\Sigma$ given by $d(\un{x},\un{y}):=\exp[-\min\{|n|:x_n\neq y_n\}]$.
\item {\em Vertical segments:} ordered pairs $[z,w]_v\in\Sigma_1\x\Sigma_1$ where $w=\psi^t(z)$
for some $t$. The {\em length} of a vertical segment $[z,w]_v$ is
$\ell([z,w]_v):=\min\{|t|>0:w=\psi^t(z)\}$.
\item {\em Basic paths} from $z$ to $w$:
$\gamma:=(z_0=z\xrightarrow[]{t_0}z_1\xrightarrow[]{t_1}\cdots\xrightarrow[]{t_{n-2}}z_{n-1}\xrightarrow[]{t_{n-1}}z_n=w)$
with $t_i\in\{h,v\}$ s.t. $[z_{i},z_{i+1}]_{t_{i}}$ is a horizontal segment when $t_{i}=h$, and a
vertical segment when $t_{i}=v$. Define $\ell(\gamma):=\sum_{i=0}^{n-1} \ell([z_i,z_{i+1}]_{t_i})$.
\item
{\em Bowen-Walters Metric on $\Sigma_1$:} $d_1(z,w):=\inf\{\ell(\gamma)\}$ where
$\gamma$ ranges over all basic paths from $z$ to $w$.
\end{enumerate}

Next we consider the general case $r\not\equiv 1$. The idea is to use a canonical bijection from
$\Sigma_r$ to $\Sigma_1$, and declare that it is an isometry.

\medskip
\noindent
{\sc Bowen-Walters Metric on $\Sigma_r$ \cite{Bowen-Walters-Metric}:}
$d_r(z,w):=d_1(\vartheta_r(z),\vartheta_r(w))$, where $\vartheta_r:\Sigma_r\to\Sigma_1$
is the map $\vartheta_r(\un{x},t):=(\un{x},t/r(\un{x}))$.
\label{Bowen-Walters-Metric-Page}

\begin{lem}\label{Lemma-BW}
Assume $r$ is bounded away from zero and infinity, and H\"older continuous.
Then $d_r$ is a metric, and there are constants $C_1,C_2,C_3>0$, $0<\kappa<1$ which only
depend on $r$ such that for all $z=(\un{x},t),w=(\un{y},s)$ in $\Sigma_r$:
\begin{enumerate}[$(1)$]
\item $d_r\bigl(z,w\bigr)\leq C_1[d(\un{x},\un{y})^\kappa+|t-s|]$.
\item Conversely:
\begin{enumerate}[{\rm (a)}]
\item If $\left|\frac{t}{r(\un{x})}-\frac{s}{r(\un{y})}\right|\leq\frac{1}{2}$ then
$d(\un{x},\un{y})\leq C_2 d_r(z,w)$ and $|s-t|\leq C_2 d_r(z,w)^\kappa$.
\item If $\frac{t}{r(\un{x})}-\frac{s}{r(\un{y})}>\frac{1}{2}$ then
$d(\sigma(\un{x}),\un{y})\leq C_2 d_r(z,w)$ and $|t-r(x)|,s\leq C_2 d_r(z,w)$.
\end{enumerate}
\item For all $|\tau|<1$, $d_r(\sigma_r^{\tau}(z),\sigma_r^{\tau}(w))\leq C_3 d_r(z,w)^\kappa$.
\end{enumerate}
\end{lem}
\noindent
See the appendix for a proof.

\begin{lem}
The map $\wh{\pi}_r:\wh{\Sigma}_r\to M$ is H\"older continuous with respect to the Bowen-Walters metric.
\end{lem}

\begin{proof}
Fix $(\un{x},t),(\un{y},s)\in\wh{\Sigma}_r$. If $\left|\frac{t}{r(\un{x})}-\frac{s}{r(\un{y})}\right|\leq\frac{1}{2}$ then
\begin{align*}
&\dist_M(\wh{\pi}_r(\un{x},t),\wh{\pi}_r(\un{y},s))=\dist_M(\vf^t(\wh{\pi}(\un{x})),\vf^s(\wh{\pi}(\un{y})))\\
&\leq \dist_M(\vf^t(\wh{\pi}(\un{x})),\vf^s(\wh{\pi}(\un{x})))+\dist_M(\vf^s(\wh{\pi}(\un{x})),\vf^s(\wh{\pi}(\un{y})))\\
&\leq \max_{p\in M}\|X_p\|\cdot |t-s|+\textrm{Lip}(\vf^s)\Hol(\wh{\pi})d(\un{x},\un{y})^\delta
\end{align*}
where $X_p$ is the vector field of $\vf$, and $\delta$ is the H\"older exponent
of $\wh{\pi}:\Sigma(\wh{\mathfs G})\to \Lambda$.
The first summand is bounded by $\const d_r(z,w)^{\kappa}$, by Lemma \ref{Lemma-BW}(2)(a).
The second summand is bounded by $\const d_r(z,w)^{\delta}$, because $\vf$ is a flow of a
Lipschitz (even $C^{1+\beta}$) vector field, therefore there are global constants $a,b$ s.t.
$\Lip(\vf^s)\leq b e^{a|s|}$ \cite[Lemma 4.1.8]{Abraham-Marsden-Ratiu} and so
$\Lip(\vf^s)\leq b^{a\sup R}=O(1)$. It follows that
$\dist_M(\wh{\pi}_r(\un{x},t),\wh{\pi}_r(\un{y},s))\leq \const d_r((\un{x},t),(\un{y},s))^{\min\{\kappa,\delta\}}$.

Now assume that $\frac{t}{r(\un{x})}-\frac{s}{r(\un{y})}>\frac{1}{2}$. Since
$\vf^{r(\un{x})}[\wh{\pi}(\un{x})]=\wh{\pi}[\sigma(\un{x})]$, we have
\begin{align*}
&\dist_M(\wh{\pi}_r(\un{x},t),\wh{\pi}_r(\un{y},s))=\dist_M(\vf^t[\wh{\pi}(\un{x})],\vf^s[\wh{\pi}(\un{y})])\\
&\leq \dist_M(\vf^t[\wh{\pi}(\un{x})],\vf^{r(\un{x})}[\wh{\pi}(\un{x})])+
\dist_M(\wh{\pi}[\sigma(\un{x})],\wh{\pi}[\un{y}])+\dist_M(\wh{\pi}[\un{y}],\vf^s[\wh{\pi}(\un{y})])\\
&\leq \max_{p\in M}\|X_p\|\cdot(|t-r(\un{x})|+|s|)+\Hol(\wh{\pi}) d(\sigma(\un{x}),\un{y})^\delta\leq \const d_r((\un{x},t),(\un{y},s))^{\delta},
\end{align*}
by Lemma \ref{Lemma-BW}(2)(b).

In both cases we find that
$\dist_M(\wh{\pi}_r(\un{x},t),\wh{\pi}_r(\un{y},s))\leq \const d_r((\un{x},t),(\un{y},s))^\gamma$,
where $\gamma:=\min\{\kappa,\delta\}$.
\end{proof}

\part{Applications}

\section{Measures of maximal entropy}

We use the symbolic coding of Theorem \ref{ThmSymbolicDynamics} to show that a geodesic flow on
a closed smooth surface with positive topological entropy can have at most countably many ergodic
measures of maximal entropy. This application requires dealing with non-ergodic measures.

\begin{lem}\label{Lemma_NonAtomic_Ergodic_Decomp}
Let $\vf$ be a continuous flow on a compact metric space $X$. If $\vf$ has uncountably many ergodic
measures of maximal entropy, then $\vf$ has at least one measure of maximal entropy with
non-atomic ergodic decomposition.
\end{lem}

\begin{proof}
%
%
Let $M_\vf(X)$ denote the space of $\vf$--invariant probability measures, together with the weak star topology.
This is a compact metrizable space \cite[Thm 6.4]{Walters-Book}.
The following claims are standard, but we could not find them in the literature.

\medskip
\noindent
{\sc Claim 1.\/} {\em Suppose $E\subset X$ is Borel measurable, then
$\mu\mapsto \mu(E)$ is Borel measurable.}

\medskip
\noindent
{\em Proof.\/} Let $\mathfs M:=\{E\subset X:E\textrm{ is Borel, and }\mu\mapsto \mu(E)\textrm{ is Borel measurable}\}$.
Let $\mathfs A$ denote the collection of Borel sets $E$ for which there are $f_n\in C(X)$ s.t.
$0\leq f_n\leq 1$ and $f_n(x)\xrightarrow[n\to\infty]{} 1_E(x)$ everywhere.
\begin{enumerate}[$\circ$]
\item $\mathfs A$ is an algebra, because if $0\leq f_n,g_n\leq 1$ and $f_n\to 1_A, g_n\to 1_B$,
then $f_n g_n\to 1_{A\cap B}$, $(1-f_n)\to 1_{X\setminus A}$, and $(f_n+g_n-f_n g_n)\wedge 1\to 1_{A\cup B}$.
\item $\mathfs A$ generates the Borel $\sigma$--algebra $\mathfs B(X)$, because it contains
every open ball $B_r(x_0)$: take $f_n(x):=\vf_n[d(x,x_0)]$ where $\vf_n\in C(\R)$ and
$1_{[0,r-\frac{1}{n}]}\leq \vf_n\leq 1_{[0,r)}$.
\item $\mathfs M\supset\mathfs A$: if $A\in\mathfs A$, then by the dominated convergence theorem
$\mu(A)=\lim\limits_{n\to\infty} \int f_n d\mu$ for the $f_n\in C(X)$ s.t. $0\leq f_n\leq 1$ and $f_n\to 1_A$.
Since $\mu\mapsto \int f_n d\mu$ is continuous, $\mu\mapsto \mu(A)$ is Borel measurable.
\item $\mathfs M$ is closed under increasing unions and decreasing intersections.
\end{enumerate}
By the monotone class theorem \cite[Prop. 3.1.14]{Srivastva-Borel-Sets-Book}, $\mathfs M$ contains
the $\sigma$--algebra generated by $\mathfs A$, whence $\mathfs M=\mathfs B(X)$. The claim follows.

\medskip
\noindent
{\sc Claim 2.\/}
{\em  $E_{\vf}(X):=\{\mu\in M_{\vf}(X): \mu\textrm{ is ergodic}\}$ is a Borel subset of $M_{\vf}(X)$.}

\medskip
\noindent
{\em Proof.\/}
Fix a countable dense collection $\{f_n\}_{n\geq 1}\subset C(X)$, $0\leq f_n\leq 1$, then
$\mu$ is ergodic iff
$\limsup_{k\to\infty}\int|\frac{1}{k}\int_0^k f_n\circ\vf^t dt -\int f_n d\mu|d\mu=0$ for every $n$.
This is a countable collection of Borel conditions.

\medskip
\noindent
{\sc Claim 3.\/} {\em The entropy map $\mu\mapsto h_{\mu}(\vf)$ is Borel measurable.}

\medskip
\noindent
{\em Proof.\/} Let $T:=\vf^1$ (the time-one map of the flow $\vf$), then $h_{\mu}(\vf)=h_{\mu}(T)$.
Thus $h_\mu(\vf)=h_\mu(T)=\lim\limits_{n\to\infty} h_{\mu}(T,\alpha_n)$ for any sequence of finite
Borel partitions $\alpha_n$ s.t. $\max\{\diam(A): A\in\alpha_n\}\xrightarrow[n\to\infty]{}0$ \cite[Thm 8.3]{Walters-Book}.
The claim follows, since it easily follows from claim 1 that $\mu\mapsto h_{\mu}(T,\alpha_n)$ is Borel measurable.

\medskip
Let $E_{\max}(X)$ denote the set of ergodic measures with maximal entropy. By claims 2 and 3,
this is a Borel subset of $M_\vf(X)$. By the assumptions of the lemma, $E_{\max}(X)$ is uncountable.
Every uncountable Borel subset of a compact metric space carries a non-atomic Borel probability measure,
because it contains a subset homeomorphic to the Cantor set \cite[Thm 3.2.7]{Srivastva-Borel-Sets-Book}.
Let $\nu$ be a non-atomic Borel probability measure s.t. $\nu[E_{\max}(X)]=1$, and let
$m:=\int_{E_{\max}(X)}\mu\,  d\nu(\mu)$. This is a $\vf$--invariant measure with non-atomic ergodic decomposition.
Since the entropy map is affine \cite[Thm 8.4]{Walters-Book}, $m$ has maximal entropy.
\end{proof}

\begin{thm}\label{Thm-countable-max-meas}
Suppose $\vf$ is a $C^{1+\beta}$ flow with positive speed and positive topological entropy on a $C^\infty$
closed three dimensional manifold, then $\vf$ has at most countably many ergodic measures of maximal entropy.
\end{thm}

\begin{proof}
Let $h:=$ topological entropy of $\vf$, and assume by way of contradiction that $\vf$ has uncountably
many ergodic measures of maximal entropy. By Lemma \ref{Lemma_NonAtomic_Ergodic_Decomp},
$\vf$ has a measure of maximal entropy $\mu$ with a non-atomic ergodic decomposition.

By the variational principle \cite[Thm 8.3]{Walters-Book}, $h_\mu(\vf)=h$.
By the affinity of the entropy map \cite[Thm 8.4]{Walters-Book}, almost every ergodic component
$\mu_x$ of $\mu$ has entropy $h$. Fix some $0<\chi_0<h$.
By the Ruelle entropy inequality \cite{Ruelle-Entropy-Inequality}, a.e. ergodic component
$\mu_x$ is $\chi_0$--hyperbolic. Consequently, $\mu$ is $\chi_0$--hyperbolic.

This places us in the setup considered in part 2, and allows us to apply Theorem
\ref{ThmSymbolicDynamics} to $\mu$. We obtain a coding ${\pi}_r: {\Sigma}_r\to M$ s.t.
$\mu[ {\pi}( {\Sigma}^\#_r)]=1$ and ${\pi}_r: {\Sigma}^\#_r\to M$ is finite-to-one
(though not necessarily bounded-to-one).

\medskip
\noindent
{\sc Lifting Procedure:\/} {\em Define a measure $ \wh{\mu}$ on $ {\Sigma}_r$ by setting for
$E\subset  {\Sigma}_r$ Borel
\begin{equation}\label{mu-hat}
\wh{\mu}(E):=\int_{ {\pi}_r( {\Sigma}^\#_r)} \biggl(\frac{1}{| {\pi}_r^{-1}(p)\cap  {\Sigma}^\#_r|}\sum_{ {\pi}_r(\un{x},t)=p} 1_E(\un{x},t)\biggr)d\mu(p),
\end{equation}
then $\wh{\mu}$ is  a $\sigma_r$--invariant measure, $\wh{\mu}\circ\pi_r^{-1}=\mu$,
and $h_{\wh{\mu}}(\sigma_r)=h_\mu(\vf)$.}

\medskip
\noindent
{\em Proof.\/} We start by clearing away all the measurability concerns. Let $X:={\Sigma}_r$ and
$Y:=M\uplus X$ (disjoint union). Define $f:{\Sigma}_r\to Y$ by $f\upharpoonright_{{\Sigma}_r^\#}=\pi_r$
and $f\upharpoonright_{{\Sigma}_r\setminus{\Sigma}_r^\#}={\rm Id}$, then $f:X\to Y$ is a
countable-to-one Borel map between polish spaces. Such maps send Borel sets to Borel sets
\cite[Thm 4.12.4]{Srivastva-Borel-Sets-Book}, so $\pi_r({\Sigma}^\#_r)=f({\Sigma}^\#_r)$ is Borel.

Next we show that the integrand in (\ref{mu-hat}) is Borel. Let $B:=\{(x,f(x)):f(x)\in M\}$.
This is a Borel subset of $X\x Y$, because the graph of a Borel function is
Borel \cite[Thm 4.5.2]{Srivastva-Borel-Sets-Book}. For every $y\in Y$, $B_y:=\{x\in X:(x,y)\in B\}$
is countable, because either $y\in M$ and $B_y:=\pi_r^{-1}(y)\cap{\Sigma}^\#_r$, or $y\not\in M$
and then $B_y=\emptyset$. By Lusin's theorem \cite[Thm 5.8.11]{Srivastva-Borel-Sets-Book},
there are countably many partially defined Borel functions $\vf_n:M_n\to X$ s.t.
$B=\bigcup_{n=1}^\infty \{(\vf_n(y),y):y\in M_n\}$. Write $B=\biguplus_{n=1}^\infty \{(\vf_n(y),y):y\in M_n'\}$,
$M_n':=\{y\in M_n: k<n, y\in M_k\Rightarrow \vf_n(y)\neq \vf_k(y)\}$. Then for every $y\in M$,
$$
\pi_r^{-1}(y)=\{\vf_n(y):n\geq 1, y\in M_n'\},\textrm{ and }m\neq n\Rightarrow \vf_m(y)\neq\vf_n(y).
$$
Thus, the integrand in (\ref{mu-hat}) equals
$\sum_{n=1}^\infty 1_{M_n'}(p)1_E(\vf_n(p))\big/\sum_{n=1}^\infty 1_{M_n'}(p)
$, a Borel measurable function.

Now that we know that (\ref{mu-hat}) makes sense it is a trivial matter to see that it defines a measure $\wh{\mu}$.
This measure is $\sigma_r$--invariant because of the $\vf$--invariance of $\mu$ and the commutation
relation $\pi_r\circ\sigma_r=\vf\circ\pi_r$. It has the same entropy as $\mu$, because finite-to-one factor
maps preserve entropy \cite{Abramov-Rokhlin}.

\medskip
\noindent
{\sc Projection Procedure:\/} {\em Every  $\sigma_r$--invariant probability measure $\wh{m}$ on
${\Sigma}_r$ proj\-ects to a $\vf$--invariant probability measure $m:=\wh{m}\circ\pi_r^{-1}$ on
$M$ with the same entropy.}

\medskip
\noindent
{\em Proof.\/} By the Poincar\'e recurrence theorem, every $\sigma_r$--invariant probability measure
is carried by ${\Sigma}^\#_r$, therefore $\pi_r:({\Sigma}_r,\wh{m})\to (M,m)$ is a finite-to-one factor map.
Such maps preserve entropy.

\medskip
Combining the lifting procedure and the projection procedure we see that the supremum of the entropies
of $\vf$--invariant measures on $M$ equals the supremum of the entropies of $\sigma_r$--invariant
measures on $\Sigma_r$, and therefore $\wh{\mu}$ given by (\ref{mu-hat}) is a measure of maximal
entropy for $\sigma_r$.

\medskip
\noindent
{\sc Claim.\/}\label{ClaimThm6.2} {\em $\sigma_r$ has at most countably many ergodic measures of maximal entropy.}

\medskip
\noindent
{\em Proof.\/} We recall the well-known relation between measures of maximal entropy for
$\sigma_r$ and equilibrium measures for the shift map $\sigma:{\Sigma}\to{\Sigma}$ \cite{Bowen-Ruelle-SRB}:
$S:=\Sigma\x\{0\}$ is a Poincar\'e section for $\sigma_r:{\Sigma}_r\to {\Sigma}_r$, therefore every
measure of maximal entropy $\wh{\mu}$ for $\sigma_r$ can be put in the form
$\wh{\mu}=
\frac{1}{{\int_{{\Sigma}}r d\wh{\mu}_\Sigma}}{\int_{{\Sigma}} \int_0^{r(\un{x})}\delta_{(\un{x},t)}dt\, d\wh{\mu}_\Sigma(\un{x})}
$
where $\wh{\mu}_\Sigma$ is a shift-invariant measure on $\Sigma$. The denominator is well-defined,
because $r$ is bounded away from zero and infinity. If $\wh{\mu}$ is ergodic, then $\wh{\mu}_\Sigma$ is ergodic.

By the Abramov formula,
$h_{\wh{\mu}}(\sigma_r)=h_{\wh{\mu}_\Sigma}(\sigma)\big/\int_\Sigma r d\wh{\mu}_\Sigma$.
Similar formulas hold for all other $\sigma_r$--invariant probability measures $m$ and the measures
$m_\Sigma$ they induce on ${\Sigma}$. Since $\wh{\mu}$ is a measure of maximal entropy,
$h_{m_\Sigma}(\sigma)\big/\int_\Sigma r dm_\Sigma=h_{m}(\sigma_r)\leq h$ (the maximal possible entropy)
for all $\sigma$--invariant measures $m_\Sigma$. This is equivalent to saying that
$h_{m_\Sigma}(\sigma)+\int_\Sigma (-hr) dm_\Sigma\leq 0$, with equality iff $h_m(\sigma_r)=h$.
Thus, if $\wh{\mu}$ is a measure of maximal entropy for $\sigma_r$, then $\wh{\mu}_\Sigma$
is an equilibrium measure for $-h r$, where $h$ is the value of the maximal entropy. Also,
the topological pressure $P(-hr):=\sup\{h_{\nu}(\sigma)-h\int_\Sigma r d\nu\}=0$, where the
supremum ranges over all $\sigma$--invariant probability measures $\nu$ on $\Sigma$.

Recall that $r:{\Sigma}\to\R$ is H\"older continuous. By \cite{Buzzi-Sarig}, a H\"older continuous potential
on a topologically transitive countable Markov shift has at most one equilibrium measure. If the condition
of topological transitivity is dropped, then there are at most countably many such measures,
one for each transitive component with maximal topological entropy \cite{Gurevich-Topological-Entropy}
(see the proof of \cite[Thm. 5.3]{Sarig-JAMS}). It follows that there are at most countably many possibilities for
$\wh{\mu}_\Sigma$, and therefore at most countably many possibilities for $\wh{\mu}$.

\medskip
We can now obtain the contradiction which proves the theorem. Consider the ergodic decomposition
of $\wh{\mu}$ defined by (\ref{mu-hat}). Almost every ergodic component is a measure of maximal entropy
(because the entropy function is affine). By the claim there are at most countably many different such measures.
Therefore the ergodic decomposition of $\wh{\mu}$ is atomic: $\wh{\mu}=\sum p_i \wh{\mu}_i$ with
$\wh{\mu}_i$ ergodic and $p_i\in(0,1)$ s.t. $\sum p_i=1$. Projecting to $M$, and noting that factors
of ergodic measures are ergodic, we find that $\mu=\sum p_i \mu_i$ where $\mu_i:=\wh{\mu}_i\circ\pi_r^{-1}$
are ergodic. This is an atomic ergodic decomposition for $\mu$. But the ergodic decomposition is unique,
and we assumed that $\mu$ has a non-atomic ergodic decomposition.
\end{proof}

%

\section{Mixing for equilibrium measures on topological Markov flows}
%

Let $\sigma_r:\Sigma_r\to\Sigma_r$ be a topological Markov flow, together with
the Bowen-Walters metric. Let $\Phi:\Sigma_r\to\R$ be bounded and continuous.

\medskip
\noindent
{\sc The topological pressure of $\Phi$:} {
$P(\Phi):=\sup\{h_{\mu}(\sigma_r)+\int\Phi d\mu\}$, where the supremum ranges over
all $\sigma_r$--invariant probability measures $\mu$ on $\Sigma_r$.}

\medskip
\noindent
{\sc Equilibrium measure for $\Phi$:} {
A $\sigma_r$--invariant probability measure $\mu$ on $\Sigma_r$ s.t.
$h_{\mu}(\sigma_r)+\int\Phi d\mu=P(\Phi)$.}

\begin{thm}\label{theorem dichotomy for suspension}
Suppose $\mu$ is an  equilibrium measure of a bounded H\"older continuous potential for
a topological Markov flow $\sigma_r:\Sigma_r\to\Sigma_r$. If $\sigma_r$ is topologically transitive,
then the following are equivalent:
\begin{enumerate}[$(1)$]
\item If $e^{i\theta r}=h/h\circ\sigma$ for some H\"older continuous $h:\Sigma\to S^1$
and $\theta\in\R$, then $\theta=0$ and $h=\const$.
\item $\sigma_r$ is weak mixing.
\item $\sigma_r$ is  mixing.
\end{enumerate}
\end{thm}

$(3)\Rightarrow(2)\Rightarrow(1)$ because if $e^{i\theta r}=h/h\circ\sigma$, then
$F(x,t)=e^{-i\theta t}h(x)$ is an eigenfunction of the flow.
$(1)\Rightarrow(2)\Rightarrow(3)$ are known in the special case when $\Sigma$ is
a subshift of finite type: Parry and Pollicott proved $(1)\Rightarrow(2)$ \cite{Parry-Pollicott-Asterisque},
and Ratner proved $(2)\Rightarrow(3)\Rightarrow$ Bernoulli \cite{Ratner-Flows-Bernoulli}, \cite{Ratner-Flows-K}.
Dolgopyat showed us a different proof of $(2)\Rightarrow(3)$ (private communication).
These proofs can be pushed through to the countable alphabet case with some effort,
using the thermodynamic formalism for countable Markov shifts \cite{Buzzi-Sarig}.
The details can be found in  \cite[Thm 4.6]{Ledrappier-Lima-Sarig}.

\medskip
The following theorem is a symbolic analogue of Plante's necessary and sufficient condition
for a transitive Anosov flow to be a constant suspension of an Anosov
diffeomorphism \cite{Plante-Anosov-Flows-Dichotomy}, see also \cite{Bowen-Symbolic-Flows}.

\begin{thm}\label{theorem-dichotomy-2}
Let $\sigma_r:\Sigma_r\to\Sigma_r$ be a topologically transitive topological Markov flow.
Either every equilibrium measure of a bounded H\"older continuous potential is mixing,
or there is $\Sigma_r'\subset\Sigma_r$ of full measure s.t. $\sigma_r:\Sigma_r'\to\Sigma_r'$
is topologically conjugate to a topological Markov flow with constant roof function.
\end{thm}

\begin{proof}
If $\Sigma$ is a finite set, then $\Sigma_r$ equals a single closed orbit, and the claim is trivial.
From now on assume that $\Sigma$ is infinite.

Assume $\sigma_r$ is not mixing, then $\exp[i\theta r]=h/h\circ\sigma$ with $h:\Sigma\to S^1$
H\"older continuous and $\theta\neq 0$. Write $\theta=2\pi/c$ and put $h$ in the form $h=\exp[i\theta U]$,
where $U:\Sigma\to\R$ is H\"older continuous. Necessarily $r+U\circ\sigma-U\in c\Z$.
We are free to change $U$ on every partition set by a constant in $c\Z$ to make sure $U$
is bounded and positive. Fix $N>2\|U\|_\infty/\inf(r)$.

\medskip
\noindent
{\sc Construction:\/}
{\em There is a cylinder  $A=_{-m}[y_{-m},\ldots,y_n]$ s.t.:
\begin{enumerate}[{\em (i)}]
\item $m,n>0$ and $y_{-m}=y_n$.
\item $n_A(\cdot)>N$ on $A$, where $n_A(\un{x}):=\inf\{n\geq 1: \sigma^n(\un{x})\in A\}$.
\item $\un{x},\un{x}'\in A\Rightarrow |U(\un{x})-U(\un{x}')|<N\inf(r)$.
\end{enumerate}}

\medskip
To find $A$, take $\un{y}\in\Sigma$ with dense orbit. Since $\Sigma$ is infinite, $\sigma^k(\un{y})$ are distinct.
Therefore, $\un{y}$ has a cylindrical neighborhood $C$ s.t. $\sigma^k(C)\cap C=\emptyset$ for $k=1,\ldots,N$.
Choose $m,n>0$ so large that ${_{-m}[}y_{-m},\ldots,y_n]\subset C$, and $|U(\un{x})-U(\un{x}')|<N\inf(r)$
for all $\un{x},\un{x}'\in C$. Since $\un{y}$ has a dense orbit, every symbol appears in $\un{y}$ infinitely
often in the past and in the future, therefore we can choose $m,n$ so that $y_{-m}=y_n$.
The cylinder $A={_{-m}[}y_{-m},\ldots,y_n]$ satisfies (i), (ii) and (iii), because $A\subset C$.

\medskip
Since $\mu$ is ergodic and globally supported, the following set has full $\mu$--measure:
$\Sigma_r':=\{z\in \Sigma_r: \sigma_r^t(z)\in A\x\{0\}\textrm{ infinitely often in the past and in the future}\}.$

\medskip
\noindent
{\sc Step 1:\/} {\em  $\sigma_r:\Sigma_r'\to\Sigma_r'$ is topologically conjugate to a topological Markov
flow $\sigma_{r^\ast}:\Sigma^\ast_{r^\ast}\to \Sigma^\ast_{r^\ast}$ whose roof function $r^*$  takes values in $c\Z$.}

\medskip
\noindent
{\em Proof.\/} $A\x\{0\}$ is a Poincar\'e section for $\sigma_r:\Sigma_r'\to\Sigma_r'$.
The roof function is
$r_A:=r+r\circ\sigma+\cdots+r\circ\sigma^{n_A-1}$.
By  (ii),  $\inf(r_A)>N\inf(r)$, so  $0<U<\inf(r_A)$.

Let
$S:=\{\sigma_r^{U(\un{x})}(\un{x},0):(\un{x},0)\in\Sigma_r'\}$. This is a Poincar\'e section for
$\sigma_r:\Sigma_r'\to\Sigma_r'$, and its roof function is
$r_A^\ast:=r_A+U\circ\sigma^{n_A}-U$ (this is always positive because $U<\inf(r_A)$).
All the values of $r_A^\ast$ belong to $c\Z$, as can be seen from the identity
$r_A^\ast=\sum_{k=0}^{n_A-1}(r+U\circ\sigma-U)\circ\sigma^k$.
We claim that the section map of $S$ is topologically conjugate to a topological Markov shift.
Let $V$ denote the collection of  sets of the form
$$
\<\un{B}\>:=\{\sigma_r^{U(\un{x})}(\un{x},0): \un{x}\in {_{-m}[}\un{A},\un{B},\un{A}]\},
$$
where $\un{A}=(y_{-m},\ldots,y_n)$ is the word defining $A$, and $\un{B}$ is any other word s.t.
$_{-m}[\un{A},\un{B},\un{A}]\neq \emptyset$ for which the only appearances of $\un{A}$ in
$(\un{A},\un{B},\un{A})$ are at the beginning and at the end.

It is easy to see that $\sigma_r^{U(\un{x})}(\un{x},0)\in S$ iff
$\un{x}=(\ldots, \un{A},\un{B}^1,\un{A},\un{B}^2,\un{A},\ldots)$ with
$\<\un{B}^i\>\in V$, and that any sequence $\{\<\un{B}^i\>\}_{i\in\Z}\in V^\Z$ appears this way.
Let $\pi:S\to V^\Z$ be the map $\pi(\un{x})=\{\<\un{B}^i\>\}_{i\in\Z}$. Since $\un{A}$ appears in
$(\un{A},\un{B}^i,\un{A})$ only at the beginning and the end, $\pi\circ\sigma_r^{r_A^{\ast}}=\sigma \circ \pi$,
with $\sigma=$ the left shift on $V^\Z$. So the section map of $S$ is topologically conjugate to
the shift on $V^\Z$. Let $\Sigma^\ast:=V^\mathbb Z$.
The roof function with respect to this new coding is $r^\ast:=r_A^\ast\circ\pi^{-1}$.
Direct calculations show that the H\"older continuity of $r$ implies the H\"older continuity of $r^\ast$.
So $\sigma_{r^\ast}:\Sigma^\ast_{r^\ast}\to \Sigma^\ast_{r^\ast}$ is a TMF, and
$\sigma_r:\Sigma_r'\to \Sigma_r'$ is topologically conjugate to $\sigma_{r^\ast}$.

\medskip
\noindent
{\sc Step 2:} {\em  $\sigma_{r^\ast}:\Sigma^\ast_{r^\ast}\to \Sigma^\ast_{r^\ast}$ is topologically
conjugate to a topological Markov flow $\sigma_{\wt{r}}:\wt{\Sigma}_{\wt{r}}\to \wt{\Sigma}_{\wt{r}}$
where $\wt{r}$ takes values in $c\Z$, and $\wt{r}(\un{x})=\wt{r}(x_0)$.}

\medskip
\noindent
{\em Proof.\/} Since $r^\ast$ is H\"older continuous and takes values in $c\Z$, there must be some
$n_0>0$ s.t. $r^\ast$ is constant on every cylinder of the form $_{-n_0}[a_{-n_0},\ldots,a_{n_0}]$.
Take $\wt{\pi}(\un{x},t):=(\{\un{x}^i\}_{i\in\Z},t)$, where $\un{x}^i:=(x_{-n_0+i},\ldots,x_{n_0+i})$.
The reader can check that the collection of $\{\un{x}^i\}_{i\in\Z}$ thus obtained is a topological
Markov shift $\wt{\Sigma}$, and that $\wt{r}(\{\un{x}^i\}_{i\in\Z})$ only depends on the first symbol $\un{x}^0$.

\medskip
\noindent
{\sc Step 3:\/} {\em  $\sigma_{\wt{r}}:\wt{\Sigma}_{\wt{r}}\to \wt{\Sigma}_{\wt{r}}$ is topologically conjugate
to a topological Markov flow $\sigma_{\wh{r}}:\wh{\Sigma}_{\wh{r}}\to \wh{\Sigma}_{\wh{r}}$ where
$\wh{r}$ is constant equal to $c$.}

\medskip
\noindent
{\em Proof.\/} The set $\{(\un{x},kc): \un{x}\in\wt{\Sigma}, k\in\Z,\ 0\leq kc<\textrm{value of $\wt{r}$ on }_0[x_0]\}$
is a Poincar\'e section for the suspension flow with constant roof function (equal to $c$).
The section map is conjugate to a topological Markov shift $\wh{\Sigma}$ which we now describe.

Let $\wt{\mathfs G}=\mathfs G(\wt{V},\wt{E})$ be the graph of $\wt{\Sigma}$.
Let $\wh{\Sigma}=\Sigma(\wh{\mathfs G})$, where $\wh{\mathfs G}$ has the set of vertices
$\wh{V}:=\{{v\choose k}:v\in\wt{V}, 0\leq kc< \textrm{value of $\wt{r}$ on $_0[v]$}\}$ and edges
${v\choose k}\to {v\choose k+1}$ when ${v\choose k+1}\in \wt{V}$, and
${v\choose k}\to {w\choose 0}$ when ${v\choose k+1}\not\in \wt{V}$ and $v\to w$ in $\wt{E}$.
The conjugacy $\wh{\pi}:\wt{\Sigma}_{\wt{r}}\to \wh{\Sigma}_{\wh{r}}$ is
$\wh{\pi}(\un{x},t):=(\sigma^{\lfloor t/c\rfloor}(\un{y}),t-\lfloor t/c\rfloor c)$, where $\un{y}$ is given by
$(\ldots;{x_0\choose 0},{x_0\choose 1},\ldots,$
${x_0\choose \wt{r}(x_0)/c-1};{x_1\choose 0},{x_1\choose 1},\ldots,{x_1\choose \wt{r}(x_1)/c-1};\ldots)$
with ${x_0\choose 0}$ at the zeroth coordinate.
\end{proof}

\section{Counting simple closed orbits}\label{SectionClosedOrbits}

Let $\pi(T):=\#\{[\gamma]:\text{$\gamma$ is a simple closed geodesic s.t. }\ell[\gamma]\leq T\}$.
In this section we prove the following generalization of Theorem \ref{ThmClosedOrbits}.

\begin{thm}\label{ThmClosedOrbits(General)}
Suppose $\vf$ is a $C^{1+\beta}$ flow with positive speed and positive topological entropy $h$
on a $C^\infty$ closed three dimensional manifold $M$. If $\vf$ has a measure of maximal entropy,
then $\pi(T)\geq C\frac{e^{hT}}{T}$ for all $T$ large enough and $C>0$.
\end{thm}

\noindent
This implies Theorem \ref{ThmClosedOrbits}, because every $C^\infty$ flow admits a measure of
maximal entropy. Indeed, by a theorem of Newhouse \cite{Newhouse-Entropy}, $\vf^1:M\to M$
admits a measure of maximal entropy $m$, and $\mu:=\int_0^1 m\circ\vf^{t}dt$ has maximal entropy for $\vf$.

\subsection*{Discussion}
Theorem \ref{ThmClosedOrbits(General)} strengthens Katok's bound
$\liminf_{T\to\infty}\frac{1}{T}\log\pi(T)\geq h$,
see \cite{KatokIHES,Katok-Closed-Geodesics} for general flows, and it improves Macarini
and Schlenk's bound $\liminf_{T\to\infty}\frac{1}{T}\log\pi(T)>0$ for the class of Reeb flows in \cite{Macarini-2011}.
If one assumes more on the flow, then much better bounds for $\pi(T)$ are known:
\begin{enumerate}[$(1)$]
\item {\em Geodesic flows on closed hyperbolic surfaces:\/} $\pi(T)\sim e^{t}/t$ \cite{Huber-Closed-Geodesics}.
\item {\em Topologically weak mixing Anosov flows} (e.g. geodesic flows on closed surfaces
with negative curvature): $\pi(T)\sim Ce^{hT}/T$ \cite{Margulis-Closed-Orbits} where
$C=1/h$ (C. Toll, unpublished). See \cite{Pollicott-Sharp-PNT-li-Paper} for estimates of the error term.
The earliest estimates for $\pi(T)$ in variable curvature are due to Sina{\u \i} \cite{Sinai-Closed-Geodesics}.
\item {\em Topologically weak mixing Axiom A flows:\/} $\pi(T)\sim e^{hT}/hT$ \cite{Parry-Pollicott-PNT}.
See \cite{Pollicott-Sharp-Error-Term} for an estimate of the error term.
\item {\em Geodesic flows on compact rank one manifolds\/:}
$C_1\frac{e^{hT}}{T}\leq \pi_0(T)\leq C_2\frac{e^{hT}}{T}$
for some $C_1, C_2>0$, where $\pi_0(T)$ counts the homotopy classes of simple closed geodesics with
length less than $T$ \cite{Knieper-Closed-Orbits,Knieper-Handbook-Chapter}.
\item {\em Geodesic flows for certain non-round spheres:\/} for certain metrics constructed by
\cite{Donnay,Burns-Gerber}, $\pi(T)\sim e^{hT}/hT$ \cite{Weaver}.
\end{enumerate}

We cannot give upper bounds for $\pi(T)$ as in (1)--(5), because in the general setup we consider
there can be compact invariant sets with lots of closed geodesics but zero topological entropy
(e.g. embedded flat cylinders). Such sets have zero measure for any ergodic measure with positive
entropy, and they lie outside the ``sets of full measure" that we can control using the methods of this paper.
Adding to our pessimism is the existence of $C^r$ $(1<r<\infty)$ surface diffeomorphisms with
super-exponential growth of periodic points \cite{Kaloshin-Super-Exponential}.
The suspension of these examples gives $C^r$ flows with super-exponential growth of closed orbits.
To the best of our knowledge, the problem of doing this in $C^\infty$ is still open.


\subsection*{Preparations for the proof of Theorem \ref{ThmClosedOrbits(General)}}

Fix an ergodic measure of maximal entropy for $\vf$, and apply Theorem \ref{ThmSymbolicDynamicsErgodic}
with this measure. The result is a topological Markov flow $\sigma_r:{\Sigma}_r\to{\Sigma}_r$
together with a H\"older continuous map $\pi_r:{\Sigma}_r\to M$, satisfying (1)--(6) in
Theorem \ref{ThmSymbolicDynamicsErgodic}.

We saw in the proof of Theorem \ref{Thm-countable-max-meas} (see page \pageref{ClaimThm6.2})
that if $\vf$ has a measure of maximal entropy, then $\sigma_r$ has a measure of maximal entropy.
By the ergodic decomposition, $\sigma_r$ has an ergodic measure of maximal entropy.
Fix such a measure $\mu$, and write
$\mu=\frac{1}{\int_\Sigma r d\nu}\int_\Sigma\left( \int_0^{r(\un{x})}\d_{(\un{x},t)}dt\right)d\nu(x)$.
The induced measure $\nu$ is an ergodic shift invariant measure on $\Sigma$. When we proved
Theorem \ref{Thm-countable-max-meas}, we saw that $\nu$ is an equilibrium measure for $\phi=-hr$.
Like all ergodic shift invariant measures, $\nu$ is supported on a topologically transitive topological
Markov shift $\Sigma'\subseteq \Sigma$ \cite{Aaronson-Denker-Urbanski}. There is no loss of generality in
assuming that $\sigma:\Sigma\to\Sigma$ is topologically transitive (otherwise we work with $\Sigma'$).

\subsection*{Proof of  Theorem \ref{ThmClosedOrbits(General)} when  $\mu$ is mixing.}

Fix $0<\epsilon<10^{-1}\inf(r)$.
Since $r$ is H\"older, there are $H>0$ and $0<\alpha<1$ s.t. $|r(\un{x})-r(\un{y})|\leq H d(\un{x},\un{y})^\alpha$.
Recall that $d(\un{x},\un{y})=\exp[-\min\{|n|:x_n\neq y_n\}]$. For every $\ell\geq 1$,
if $x_{-n_0}^{n_0+\ell}=y_{-n_0}^{n_0+\ell}$ then
$$
|r_\ell(\un{x})-r_\ell(\un{y})|\leq \sum_{i=0}^{\ell-1}H d(\sigma^i(\un{x}),\sigma^i(\un{y}))^\alpha\leq H\sum_{i=0}^{\ell-1} e^{-\alpha \min\{n_0+i,n_0+\ell-i\}}<\frac{2He^{-\alpha n_0}}{1-e^{-\alpha}}\cdot
$$
Choose $n_0$ s.t.
$\sup\{|r_\ell(\un{x})-r_\ell(\un{y})|:x_{-n_0}^{n_0+\ell}=y_{-n_0}^{n_0+\ell},\ell\geq 1\}<\epsilon$.
Fix some cylinder $A:={_{-n_0}[}a_{-n_0},\ldots,a_{n_0}]$ s.t. $\nu(A)\neq 0$, and let
$\Upsilon(T):=\biguplus_{n= 1}^\infty\Upsilon(T,n)$, where
$$
\Upsilon(T,n):=\{(\un{y},n): \un{y}\in A,\sigma^n(\un{y})=\un{y}, |r_n(\un{y})-T|<2\epsilon\}.
$$
Given $(\un{y},n)\in\Upsilon(T,n)$, let
$\gamma_{\un{y},n}:[0,r_n(\un{y})]\to M$, $\gamma_{\un{y},n}(t)=\pi_r[\sigma_r^t(\un{y},0)]$.
This is a closed orbit with length $\ell(\gamma_{\un{y},n})=r_n(\un{y})\in [T-2\epsilon,T+2\epsilon]$.
But $\gamma_{\un{y},n}(t)$ is not necessarily simple, because $\pi$ is not injective.
Let $\gamma_{\un{y},n}^s:=\gamma_{\un{y},n}\upharpoonright_{[0,\ell(\gamma_{\un{y},n})/N]}$,
where $N=N(\un{y},n):=\#\{0\leq t<\ell(\gamma_{\un{y},n}):\gamma_{\un{y},n}(t)=\gamma_{\un{y},n}(0)\}$.
Then $\gamma_{\un{y},n}^s$ is a simple closed orbit. We have $N=1$ iff $\gamma_{\un{y},n}$ is simple, and
$N<\ell(\gamma_{\un{y},n})/\inf(r)$, because an orbit with length less than $\inf(r)$ cannot be closed.

We obtain a map $\Theta:\Upsilon(T)\to\{\textrm{$[\gamma]$: $\gamma$ is a simple closed orbit
s.t. $\ell(\gamma)\leq T+2\epsilon$}\}$,
$$
\Theta:(\un{y},n)\mapsto [\gamma_{\un{y},n}^s].
$$
The map $\Theta$ is not one-to-one, but there is a uniform bound on its non-injectivity:
\begin{equation}\label{Theta-Bound}
1\leq \frac{\#\Theta^{-1}([\gamma_{\un{y},n}^s])}{n}\leq c_0.
\end{equation}
Here is the proof. Suppose $(\un{y},n),(\un{z},m) \in \Upsilon(T)$ and
$[\gamma_{\un{y},n}^s]=[\gamma_{\un{z},m}^s]$, then:
\begin{enumerate}[$\circ$]
\item $N(\un{y},n)=N(\un{z},m)$: $\left[\frac{T-2\epsilon}{N(\un{y},n)},\frac{T+2\epsilon}{N(\un{y},n)}\right]$ and
$\left[\frac{T-2\epsilon}{N(\un{z},m)},\frac{T+2\epsilon}{N(\un{z},m)}\right]$ both contain
$\ell=\ell(\gamma_{\un{y},n}^s)$ $=\ell(\gamma_{\un{z},m}^s)$.
But $N(\un{y},n),N(\un{z},m)<\frac{T+2\epsilon}{\inf(r)}$, and
$\left[\frac{T-2\epsilon}{j},\frac{T+2\epsilon}{j}\right]$ are pairwise disjoint for $j=1,\ldots,[\frac{T+2\epsilon}{\inf(r)}]$,
because $\epsilon<\frac{1}{10}\inf(r)$.
\item $[\gamma_{\un{y},n}]=[\gamma_{\un{z},m}]$, because $[\gamma_{\un{y},n}^s]=[\gamma_{\un{z},m}^s]$
and $N(\un{y},n)=N(\un{z},m)$.
\item $n=m$, because $n$, $m$ are the number of times $\gamma_{\un{y},n}$, $\gamma_{\un{z},m}$
enter $\Lambda_0:=\pi_r({\Sigma}\x\{0\})$, and equivalent closed orbits enter $\Lambda_0$ the same number of times.
\item $\pi_r(\un{z},0)=\pi_r(\sigma^k(\un{y}),0)$ for some $k=0,\ldots,n-1$, because
$\pi_r(\un{z},0)\in \gamma_{\un{y},n}\cap\Lambda_0$.
\item $\un{y},\un{z}\in{\Sigma}^\#$, and $y_i=a_0$ for  infinitely many $i<0$ and infinitely many $i>0$.
\end{enumerate}
By Theorem \ref{ThmSymbolicDynamicsErgodic}(5) there is a constant $c_0:=N(a_0,a_0)$
s.t. if $x_i=a_0$ for infinitely many $i>0$ and  infinitely many $i<0$, then
$\#\{\un{z}\in{\Sigma}^\#:\wh{\pi}(\un{z})=\wh{\pi}(\un{x})\}\leq c_0$. Thus
$\#\Theta^{-1}([\gamma_{\un{y},n}^s])\leq
\#[{\Sigma}_r^\#\cap \bigcup_{k=0}^{n-1}\pi_r^{-1}\{\pi_r(\sigma^k(\un{y}),0)\}]\leq c_0 n.$
Also $\#\Theta^{-1}([\gamma_{\un{y},n}^s])\geq n$, because
$[\gamma_{\sigma^k(\un{y}),n}^s]=[\gamma_{\un{y},n}^s]$ for $k=0,\ldots,n-1$.
This proves (\ref{Theta-Bound}).

\medskip
By the inequality (\ref{Theta-Bound}) and the fact shown above that
$[\gamma_{\un{y},n}^s]=[\gamma_{\un{z},m}^s]\Rightarrow m=n$,
\begin{align*}\label{S(T)-def}
&\ \#\{[\gamma]:\textrm{$\gamma$ simple closed orbit s.t. }\ell(\gamma)\leq T+2\epsilon\}\geq \notag\\
&\geq \#\{[\gamma_{\un{y},n}^s]:(\un{y},n)\in\Upsilon(T)\}=\sum_{n=1}^\infty\#\{[\gamma_{\un{y},n}^s]:(\un{y},n)\in\Upsilon(T,n)\}\notag\\
&\asymp \sum_{n=1}^\infty \frac{\#\Upsilon(T,n)}{n}\ ,\textrm{ where $A_n\asymp B_n$ means $\exists C,N_0$ s.t. $\forall n>N_0$, $C^{-1}\leq \frac{A_n}{B_n}\leq C$} \notag\\
&=\sum_{n=1}^\infty \left(\frac{1}{n}\sum_{\sigma^n(\un{y})=\un{y}}1_A(\un{y}) 1_{[-2\epsilon,2\epsilon]}(r_n(\un{y})-T)\right) \notag\\
&\asymp \frac{e^{hT}}{T}\sum_{n=1}^\infty \sum_{\sigma^n(\un{y})=\un{y}}1_A(\un{y}) 1_{[-2\epsilon,2\epsilon]}(r_n(\un{y})-T) e^{-h r_n(\un{y})}\\
&= \frac{e^{hT}}{T}S(T), \text{ where  }S(T):=\sum_{n=1}^\infty \sum_{\sigma^n(\un{y})=\un{y}}1_A(\un{y}) 1_{[-2\epsilon,2\epsilon]}(r_n(\un{y})-T) e^{-h r_n(\un{y})}.
\end{align*}
To prove the theorem, it is enough to show that $\liminf S(T)>0$.

Recall that $\nu$ is an equilibrium measure for $\phi=-hr$ and $P(-hr)=0$ (see the claim on
page \pageref{ClaimThm6.2}). The structure of such measures was found in \cite{Buzzi-Sarig}.
We will not repeat the characterization here, but we will simply note that it implies the following
uniform estimate \cite[page 1387]{Buzzi-Sarig}: $\exists C(a)>1$ s.t. for every cylinder of the form
${_0[}\un{b}]={_0[}\xi_{0},\ldots,\xi_{n}]$ with $\xi_n=a$,
$C(a)^{-1}\leq \frac{\nu({_0[}\un{b}])}{\exp(-h r_n(\un{y}))}\leq C(a)\textrm{ for all $\un{y}\in {_0[}\un{b}]$}$.
It follows that there is a constant $G=G(A)$ s.t.
$$
G^{-1}\leq\frac{\exp[-h r_n(\un{y})]}{\nu({_{-n_0}[a_{-n_0},\ldots,a_{-1};y_0,\ldots,y_{n-1};a_0,\ldots,a_{n_0}}])}\leq G
$$
for every $\un{y}\in {_{-n_0}[}a_{-n_0},\ldots,a_{-1};y_0,\ldots,y_{n-1};a_0,\ldots,a_{n_0}]$.
Let
$$
U_T:= \bigcup_{(\un{y},n)\in\Upsilon(T)}{_{-n_0}[a_{-n_0},\ldots,a_{-1};y_0,\ldots,y_{n-1};a_0,\ldots,a_{n_0}}],
$$
then $S(T)\asymp \nu[U_T]=(\epsilon^{-1}\int r d\nu)\mu(U_T\x [0,\epsilon])$.

We claim that
\begin{equation}\label{mixing-inclusion}
U_T\x [0,\epsilon]\supset (A\x [0,\epsilon])\cap \sigma_r^{-T}(A\x [0,\epsilon]).
\end{equation}
Once this is shown, we can use the mixing of $\mu$ to get $\liminf \mu(U_T\x [0,\epsilon])>0$,
whence $\liminf S(T)>0$.
Suppose $(\un{x},t),\sigma_r^T(\un{x},t)\in A\x [0,\epsilon]$, and write
$\sigma_r^T(\un{x},t)=(\sigma^n(\un{x}),t+T-r_n(\un{x}))$. Since $\un{x}\in A\cap \sigma^{-n}(A)$,
there exists $\un{y}\in A$ s.t. $\sigma^n(\un{y})=\un{y}$ and $y_{-n_0}^{n+n_0}=x_{-n_0}^{n+n_0}$.
By the choice of $n_0$, $|r_n(\un{x})-r_n(\un{y})|<\epsilon$. Since $t, t+T-r_n(\un{x})\in [0,\epsilon]$,
$|r_n(\un{x})-T|<\epsilon$, whence $|r_n(\un{y})-T|<2\epsilon$. So $(\un{x},t)\in U_T\times [0,\epsilon]$.
This proves (\ref{mixing-inclusion}).

\subsection*{Proof of  Theorem \ref{ThmClosedOrbits(General)} when $\mu$ is not mixing}

In this case, Theorem \ref{theorem-dichotomy-2} gives us a set of full measure
$\Sigma_r'\subset {\Sigma}_r$ s.t. $\sigma_r:\Sigma_r'\to\Sigma_r'$ is topologically conjugate
to a constant suspension over a topologically transitive topological Markov shift
${\sigma}_c:\wt{\Sigma}\x [0,c)\to \wt{\Sigma}\x [0,c)$. Let $\vartheta:\wt{\Sigma}\x [0,c)\to{\Sigma}_r'$
denote the topological conjugacy, and let $\wt{p}:\wt{\Sigma}\x [0,c)\to M$ be the map $\wt{p}:=\pi_r\circ\vartheta$.
The map $\wt{p}$ has the same finiteness-to-one properties of $\pi_r$, because looking carefully
at the proof of Theorem \ref{theorem-dichotomy-2}, we can see that if $\un{x}\in \wt{\Sigma}$
contains some symbol $v$ infinitely many times in its future (resp. past), then $\vartheta(\un{x},t)=(\un{y},s)$
where $\un{y}$ contains some symbol $a=a(v)$ infinitely many times in its future (resp. past).

Since $\sigma_r:{\Sigma}_r'\to{\Sigma}_r'$ has a measure of maximal entropy,
${\sigma}_c:\wt{\Sigma}\x[0,c)\to\wt{\Sigma}\x [0,c)$ has a measure of maximal entropy.
By the Abramov formula, ${\sigma}:\wt{\Sigma}\to \wt{\Sigma}$ has a measure of maximal entropy,
and the value of this entropy is $hc$.
Gurevich characterized the countable state topological Markov shifts which possess measures of maximal entropy \cite{Gurevich-Topological-Entropy,Gurevich-Measures-Of-Maximal-Entropy}.
His work shows that there are $p\in\N$, $C>0$, and a vertex $v$ s.t.
$\#\{\un{x}\in \wt{\Sigma}: x_0=v, {\sigma}^{np}(\un{x})=\un{x}\}\asymp e^{np\cdot hc}$.
Such $\un{x}$ determines a simple closed orbit $\gamma_{\un{x},np}^s:\left[0,\tfrac{npc}{N(\un{x},np)}\right]\to M$,
where $\gamma_{\un{x},np}^s(t)=\wt{p}[\wt{\sigma}_c^t(\un{x},0)]$ and
$N(\un{x},np):=\#\{0\leq t<npc:\wt{p}[\wt{\sigma}_c^t(\un{x},0)]=\wt{p}[(\un{x},0)]\}$.
We therefore get a map
$\Theta:\{(\un{x},n): x_0=v, \sigma^{np}(\un{x})=\un{x}\}\to \{[\gamma]: \textrm{$\gamma$ simple s.t. }\ell(\gamma)\leq npc\}$,
$$
\Theta(\un{x},n):=[\gamma^s_{\un{x},np}].
$$
Again $\Theta$ is not one-to-one, but again one can show that
$1\leq \frac{\#\Theta^{-1}([\gamma^s_{\un{x},np}])}{np}\leq C(v)$, where $C(v)=N(a(v),a(v))$.
Thus
$\#\{[\gamma_{\un{y},np}^s]:\un{y}\in\wt{\Sigma},y_0=v, \sigma^{np}(\un{y})=\un{y}  \}\asymp\frac{e^{np\cdot hc}}{n}$.
Since $\ell(\gamma_{\un{y},np}^s)\leq npc$,
$\#\{[\gamma]:\gamma\textrm{ is simple s.t. }\ell(\gamma)\leq T_n\}\geq \const \times \frac{e^{hT_n}}{T_n}$ for $T_n=npc$.
It follows that
$\#\{[\gamma]:\gamma\textrm{ is simple s.t. }\ell(\gamma)\leq T\}\geq \const\times \frac{e^{hT}}{T}$ for large $T$.\hfill$\Box$

\section*{Appendix A: Standard proofs}

\subsection*{Proof of Lemma \ref{LemmaTD}}
 $M$ is closed (compact and boundaryless) and smooth, so there is a constant
 $r_{\mathrm{inj}}>0$ s.t. for every $p\in M$,
$\exp_p:\{\vec{v}\in T_p M: \|\vec{v}\|_p\leq r_{\mathrm{inj}}\}\to M$ is $\sqrt{2}$--bi-Lipschitz
onto its image (see e.g. \cite{Spivak}, chapter 9). Fix $0<r<r_{\mathrm{inj}}$, and complete
$\vec{n}_p:=\frac{X_p}{\|X_p\|}$ to an orthonormal basis $\{\vec{n}_p,\vec{u}_p,\vec{v}_p\}$ of $T_p M$. Then
$$
J_p(x,y):=\exp_p(x\vec{u}_p+y\vec{v}_p)
$$
is a $C^\infty$ diffeomorphism from $U_r:=\{(x,y)\in \R^2:x^2+y^2\leq r\}$ onto $S_r(p)$ for all
$0<r<r_{\mathrm{inj}}$, proving that   $S=S_r(p)$ is a $C^\infty$ embedded disc.

We claim that $\dist_M(\cdot,\cdot)\leq \dist_S(\cdot,\cdot)\leq 2\dist_M(\cdot,\cdot)$.
The first inequality is obvious. For the second, suppose $z_1,z_2\in S$.
There are $\vec{v}_1,\vec{v}_2\perp X_p$ s.t. $\|\vec{v}_i\|_p\leq r$ and  $z_i=\exp_p(\vec{v}_i)$. Let
$\g(t):=\exp_p[t\vec{v}_2+(1-t)\vec{v}_1]$, $t\in[0,1]$.
Clearly $\g\subset S$, whence
$\dist_S(z_1,z_2)\leq \textrm{length}(\g)$. Since $\exp_p$ has bi-Lipschitz constant
$\sqrt{2}$, $\dist_S(z_1,z_2)\leq \sqrt{2}\|\vec{v}_1-\vec{v}_2\|_p\leq (\sqrt{2})^2\dist_M(z_1,z_2)$.

We  bound $\measuredangle(X_q,T_q S)$ for $q\in S$. If $\vec{u},\vec{v},\vec{w}\in\R^3\setminus\{\un{0}\}$,
then
$
|\measuredangle(\vec{u},\Span\{\vec{v},\vec{w}\})|\geq |\sin\measuredangle(\vec{u},\Span\{\vec{v},\vec{w}\})|\geq\left|\frac{\<\vec{u},\vec{v},\vec{w}\>}{\|\vec{u}\|\cdot\|\vec{v}\|\cdot\|\vec{w}\|}\right|,
$
where $\<\vec{u},\vec{v},\vec{w}\>$ is the signed volume of the parallelepiped with sides
$\vec{u},\vec{v},\vec{w}$. So for every $q\in S_r(p)$,
$$|\measuredangle(X_{q}, T_{q}S)|\geq A(p,\un{x}):=\left|\frac{\left<X_{J_p(\un{x})},(dJ_p)_{\un{x}}\frac{\del}{\del x},(dJ_p)_{\un{x}}\frac{\del}{\del y}\right>_{J_p(\un{x})}}{\|X_{J_p(\un{x})}\|_{J_p(\un{x})}\cdot\|(dJ_p)_{\un{x}}\frac{\del}{\del x}\|_{J_p(\un{x})}\cdot\|(dJ_p)_{\un{x}}\frac{\del}{\del y}\|_{J_p(\un{x})} }\right|,
$$
where $\un{x}=\un{x}(q)$ is characterized by $q=J_p(\un{x})$. By definition $A(p,\un{0})=1$, so
there is an open neighborhood $V_p$ of $p$ and $\delta_p>0$ s.t. $A(q,\un{x})>\frac{1}{2}$ on
$W_p:=V_p\x B_{\delta_p}(\un{0})$. Working in $M\x\R^3$, we cover $K:=M\times\{\un{0}\}$ by
a finite collection $\{W_{p_1},\ldots,W_{p_N}\}$, and let $r_{\mathrm{leb}}$ be a Lebesgue number.
Then $A(p,\un{x})>\frac{1}{2}$ for every $p\in M$ and $\|\un{x}\|<r_{\mathrm{leb}}$.
The lemma follows with $\mathfrak r_s:=\frac{1}{2}\min\{1,r_{\mathrm{inj}},r_{\mathrm{leb}}\}$.\hfill$\Box$

\subsection*{Uniform Inverse Function Theorem}

{\em Let $F:U\to V$ be a differentiable map between two open subsets of $\R^d$ s.t.
$\det(dF_{\un{x}})\neq 0$ for all $\un{x}\in U$. Suppose there are  $K,H,\beta$ s.t.
$\|dF_{\un{x}}\|, \|(dF_{\un{x}})^{-1}\|\leq K$ and
$\|dF_{\un{x}_1}-dF_{\un{x}_2}\|\leq H\|\un{x}_1-\un{x}_2\|^{\beta}$ for all $\un{x}_1,\un{x}_2\in U$.
%
If $\un{x}\in U$, $B_{\epsilon}(\un{x})\subset U$, and $0<\epsilon<2^{-\frac{\beta+1}{\beta}}(KH)^{-\frac{1}{\beta}}$, then:
\begin{enumerate}[$(1)$]
\item $F^{-1}$ is a well-defined differentiable open map on  $W:=B_{\delta}(F(\un{x}))$, $\delta:=\frac{\epsilon}{2K}$.
\item $\|(dF^{-1})_{\un{y}_1}-(dF^{-1})_{\un{y}_2}\|\leq H^\ast\|\un{y}_1-\un{y}_2\|^{\beta}$
for all $\un{y}_1,\un{y}_2\in W$, with $H^\ast:=K^3 H$.
\end{enumerate} }

\begin{proof} Track the constants in the fixed point theorem  proof of the
inverse function theorem (see e.g. \cite{Smart}).
\end{proof}

\subsection*{Proof of Lemma \ref{LemmaFB_1}.}Let $B:=\{\un{x}\in\R^3:\|\un{x}\|<1\}$.

Let $\mathfs V$ be a finite open cover of $M$ such that for every $V\in\mathfs V$:
\begin{enumerate}[(1)]
\item $V=C_V(B)$ where $C_V:B\to V$ is a $C^2$ diffeomorphism.
\item $C_V$ extends to a bi-Lipschitz $C^2$ map from a neighborhood of $\ov{B}$ onto $\ov{V}$.
\item $(dC_V)_{\un{x}}\frac{\partial}{\partial x},(dC_V)_{\un{x}}\frac{\partial}{\partial y},X_{C_V(\un{x})}$
are linearly independent for $\un{x}\in \ov{B}$.
\end{enumerate}

Since $M$ is compact and $X$ has no zeroes, $\|X_p\|$ is bounded from below.
This, together with the $C^{1+\beta}$ regularity of $X$, implies that $\vec{n}_p:=X_p/\|X_p\|$ is Lipschitz on $M$.
Apply the Gram-Schmidt procedure to
$\vec{n}_{C_V(\un{x})},(dC_V)_{\un{x}}\frac{\partial}{\partial x}$,
$(dC_V)_{\un{x}}\frac{\partial}{\partial y}$ for $\un{x}\in \ov{B}$. The result is a Lipschitz orthonormal
frame $\{\vec{n}_p, \vec{u}_p,\vec{v}_p\}$ for $T_p M$, $p\in \ov{V}$.

For every $p\in \ov{V}$, define the function $F_p(x,y,t):=\vf^t[\exp_p(x \vec{u}_p+y \vec{v}_p)]$.
Then $(dF_p)_{\un{0}}$ is non-singular for every $p\in\ov{V}$. Since $p\mapsto \det (dF_p)_{\un{0}}$
is continuous and $\ov{V}$ is compact, $\det(dF_p)_{\un{0}}$ is bounded away from zero for $p\in\ov{V}$.
Since $(p,x,y,t)\mapsto \det (dF_p)_{(x,y,t)}$ is uniformly continuous on $\ov{V}\x \ov{B}$,
$\exists\delta(V)>0$ s.t. $\det(dF_p)_{(x,y,t)}$ is bounded away from zero on
$\{(p,x,y,t):p\in\ov{V}, x^2+y^2\leq \delta(V)^2, |t|\leq \delta(V)\}$.

Fix
$0<\delta<\min\{\delta(V):V\in\mathfs V\}$ s.t. $\delta<{r_{\mathrm{leb}}}/{2S_0}$ where
$S_0:=1+\max_{p\in M}\|X_p\|$ and $r_{\mathrm{leb}}$ is a Lebesgue number for $\mathfs V$.
For every $p\in M$,
$F_p\bigl(\{(x,y,t):x^2+y^2\leq \delta^2,|t|\leq \delta\}\bigr)\subset B_{r_{\mathrm{leb}}}(p)$,
so $\exists V\in\mathfs V$ s.t. $F_p\bigl(\{(x,y,t):x^2+y^2\leq \delta^2,|t|\leq \delta\}\bigr)\subset V=\dom (C_V^{-1})$.
For this $V$,
$$
G=G_{p,V}:=C_V^{-1}\circ F_p: \{(x,y,t):x^2+y^2\leq \delta^2,|t|\leq \delta\}\to\R^3
$$
is a well-defined map, with Jacobian uniformly bounded away from zero.
A direct calculation shows that $\|dG_{(x,y,t)}\|, \|(dG_{(x,y,t)})^{-1}\|$ and the $\beta$--H\"older
norm of $dG$ are uniformly bounded by constants that do not depend on $p,V$.

By the uniform inverse function theorem, for every $0<\delta'\leq \delta$, the image
$G\bigl(\{(x,y,t):x^2+y^2\leq (\delta')^2,|t|\leq \delta'\}\bigr)$ contains a ball $B^\ast$ of
some fixed radius $\mathfrak d'(\delta')$ centered at $C_V^{-1}(p)$, and $G$ can be
inverted on $B^\ast$. So $F_p^{-1}$ is well-defined and smooth on $C_V(B^\ast)$.
Since $C_V$ is bi-Lipschitz, there is a constant $\mathfrak d(V,\delta')$ s.t.
$C_V(B^\ast)\supset B_{\mathfrak d(V,\delta')}(p)$, so $F_p^{-1}$ is well-defined and
smooth on $B_{\mathfrak d(V,\delta')}(p)$. The $C^{1+\beta}$ norm of the $F_p^{-1}$
there is uniformly bounded by a constant which only depends on $V$. Thus $(q,t)\mapsto \vf^t(q)$
can be inverted with bounded $C^{1+\beta}$ norm on $B_{\mathfrak d(V,\delta')}(p)$.
Let $K(V)$ denote a bound on the Lipschitz constant of the inverse function, and let
$\rho(V):=\delta/2K(V)$, then $(q,t)\mapsto \vf^t(q)$ is a diffeomorphism from
$S_{\rho(V)}(p)\x[-\rho(V),\rho(V)]$ onto ${\rm FB}_{\rho(V)}(p)$. Let
$\mathfrak r_f:=\min\{\rho(V):V\in\mathfs V\}$, then $(q,t)\mapsto\vf^t(q)$ is a diffeomorphism
from $S_{\mathfrak r_f}(q)\x [-\mathfrak r_f, \mathfrak r_f]$ onto ${\rm FB}_{\mathfrak r_f}(p)$.
The lemma follows with this $\mathfrak r_f$, and with
$\mathfrak d:=\min\{\mathfrak d(V,\frac{1}{2}\mathfrak r_f):V\in\mathfs V\}$.\hfill$\Box$

\subsection*{Proof of Lemma \ref{LemmaFB_2}}

We use the notation of the previous proof. Invert the function $F_p(x,y,t):=\vf^t[\exp_p(x\vec{u}_p+y\vec{v}_p)]$
on $B_{\mathfrak d}(p)$:
$$
F_p^{-1}(z)=(x_p(z),y_p(z),t_p(z)) \ \ \ (z\in B_{\mathfrak d}(p)).
$$
By the uniform inverse function theorem, the $C^{1+\beta}$ norm of $G^{-1}$ is bounded by
some constant independent of $p,V$.
Since $F_p^{-1}=G^{-1}\circ C_V^{-1}$, $C_V$ is bi-Lipschitz, and $\mathfs V$ is finite,
$x_p(\cdot),y_p(\cdot),t_p(\cdot)$ have uniformly bounded Lipschitz constants (independent of $p$),
and the differentials of $x_p, y_p, t_p$ are $\beta$--H\"older with uniformly bounded
H\"older constants (independent of $p$).
Clearly ${\mathfrak t}_p(z):=t_p(z)$ and ${\mathfrak q}_p(z):=\exp_p[x_p(z)\vec{u}_p+y_p(z)\vec{v}_p]$
are the unique solutions for
$z=\vf^{{\mathfrak t}_p(z)}[{\mathfrak q}_p(z)]$. Thus ${\mathfrak t}_p, {\mathfrak q}_p$ are
Lipschitz functions with Lipschitz constant bounded by some $\mathfrak L$ independent of $p$,
and $C^{1+\beta}$ norm bounded by some $\mathfrak H$ independent of $p$.
\hfill$\Box$

\subsection*{Proof of Lemma \ref{LemmaStandardSection}}

Cover $M$ by a finite number of flow boxes ${\rm FB}_r(z_i)$ with radius $r$.
The union of $S_r(z_i)$ is a Poincar\'e section, but this section is not necessarily standard,
because $S_r(z_i)$ are not necessarily pairwise disjoint.
To solve this problem we approximate each $S_r(z_i)$ by a finite ``net" of points $z^i_{jk}$,
and shift each $z^i_{jk}$ up or down along the flow to points $p^i_{jk}=\vf^{\theta^i_{jk}}(z^i_{jk})$
in such a way that $S_{R_0}(p^i_{jk})$ are pairwise disjoint for some $R_0<r$ which is still large
enough to ensure that $\bigcup S_{R_0}(z^i_{jk})$ is a Poincar\'e section.

We begin with the choice of some constants. Let:
\begin{enumerate}[$\circ$]
\item $h_0>0$ small, $K_0>1$ large (given to us). Without loss of generality,  $0<h_0<\mathfrak r_f$.
\item $r_{\mathrm{inj}}\in(0,1)$  s.t. $\exp_p:\{\vec{v}\in T_p M:\|\vec{v}\|\leq r_{\mathrm{inj}}\}\to M$
is $\sqrt{2}$--bi-Lipschitz  for all $p\in M$.
\item $S_0:=1+\max\|X_p\|$ and $\mathfrak r, \mathfrak d, \mathfrak L$
are as in Lemmas \ref{LemmaTD}--\ref{LemmaFB_2}.
Recall that $\mathfrak r,\mathfrak d\in (0,1)$ and  $\mathfrak L>1$.
\item $r_0:=\frac{1}{9}\mathfrak r \mathfrak d h_0 r_{\mathrm{inj}}/(K_0+S_0)$. Notice that $r_0<\frac{1}{9}\mathfrak r,\frac{1}{9}\mathfrak d, \frac{1}{9}h_0,\frac{1}{9}r_{\mathrm{inj}}$.
\end{enumerate}
By Lemma \ref{LemmaFB_1} and the compactness of $M$, it is possible to cover
$M$ by finitely many flow boxes ${\rm FB}_{r_0}(z_1),\ldots,{\rm FB}_{r_0}(z_N)$.
With this $N$ in mind, let:
\begin{enumerate}[$\circ$]
\item $\rho_0:=r_0(10 K_0 S_0 N \mathfrak L)^{-20}$. This is smaller than $r_0$.
\item $R_0:=K_0 \rho_0$. This is larger than $\rho_0$, but still much smaller than $r_0$.
\item $\d_0:=\rho_0/(8\mathfrak L^2)$. This is much smaller than $r_0$.
\item $\kappa_0:=\lceil 10^2 K_0 \mathfrak L^4\rceil$, a big integer.
\end{enumerate}

For every $i$, complete $\vec{n}_i:=X_{z_i}/\|X_{z_i}\|$ to an orthonormal basis
$\{\vec{u}_i,\vec{v_i},\vec{n}_i\}$ of $T_{z_i}M$, and let $J_i:\R^2\to M$ be the map
$$
J_i(x,y)=\exp_{z_i}(x\vec{u}_i+y\vec{v}_i),
$$
then $S_{r_0}(z_i)=J_i\bigl(\{(x,y):x^2+y^2\leq r_0^2\}\bigr)$.
The map $J_i$ is $\sqrt{2}$--bi-Lipschitz, because $r_0<r_{\mathrm{inj}}$.
Let $I:=\{(j,k)\in\Z^2:(j\d_0)^2+(k\d_0)^2\leq r_0^2\}$. Given $1\leq i\leq N$ and $(j,k)\in I$, define
$$
z_{jk}^i:=J_i(j\d_0,k\d_0).
$$
Then $\{z_{jk}^i: (j,k)\in I\}$ is a net of points in $S_{r_0}(z_i)$, and for all $(j,k)\neq (\ell,m)$:
\begin{equation}\label{GridEst}
\frac{1}{\sqrt{2}}\leq \frac{\dist_M(z_{jk}^i, z_{\ell m}^i)}{\d_0\sqrt{(j-l)^2+(k-m)^2}}\leq \sqrt{2}.
\end{equation}

We will construct points $p_{jk}^i:=\vf^{\theta_{jk}^i}(z_{jk}^i)$ with $\theta_{jk}^i\in [-r_0,r_0]$
s.t. ${S_{R_0}(p_{jk}^i)}$ are pairwise disjoint. The following claim will help us prove disjointness.

\medskip
\noindent
{\sc Claim.\/} {\em Suppose $p^i_{jk}=\vf^{\theta^i_{jk}}(z^i_{jk})$, $p^i_{\ell m}=\vf^{\theta^i_{\ell m}}(z^i_{\ell m})$,
where $\theta^i_{jk}, \theta^i_{\ell m}\in [-r_0,r_0]$. If
$
{S_{R_0}(p^i_{jk})}\cap {S_{R_0}(\vf^{\tau_1}(z^\gamma_{\alpha \beta}))}\neq \emptyset\textrm{ and }
 {S_{R_0}(p^i_{\ell m})}\cap {S_{R_0}(\vf^{\tau_2}(z^\gamma_{\alpha \beta}))}\neq \emptyset
$
for the same  $z^\gamma_{\alpha \beta}$ and some $\tau_1,\tau_2\in [-r_0,r_0]$,
then $\max\{|j-\ell|, |k-m|\}<\kappa_0$.}

\medskip
\noindent
In particular, $S_{R_0}(p^i_{jk})\cap S_{R_0}(p^i_{\ell m})\neq\emptyset\Rightarrow\max\{|j-\ell|, |k-m|\}<\kappa_0$
(take $z_{\alpha\beta}^{\gamma}=z^i_{\ell m}$, $\tau_1=\tau_2=\theta^i_{\ell m}$).

\medskip
\noindent
{\em Proof.\/}
${S_{R_0}(p^i_{jk})},{S_{R_0}(p^i_{\ell m})},z^\gamma_{\alpha\beta},
\vf^{\tau_1}(z^\gamma_{\alpha\beta}),\vf^{\tau_2}(z^\gamma_{\alpha\beta})$
are all contained in $B_{\mathfrak d}(z_i)$:
\begin{enumerate}[$\circ$]
\item ${S_{R_0}(p^i_{jk})}\subset B_{\mathfrak d}(z_i)$, because if $q\in {S_{R_0}(p^i_{jk})}$ then
$\dist_M(q,z_i)\leq \dist_M(q,p^i_{jk})+\dist_M(p^i_{jk},z^i_{jk})+\dist_M(z^i_{jk},z_i)\, \leq R_0+r_0 S_0+r_0<\mathfrak d$.  Similarly, ${S_{R_0}(p^i_{\ell m})}\subset B_{\mathfrak d}(z_i)$.

\item  $z^\gamma_{\alpha\beta}\in B_{\mathfrak d}(z_i)$:
$\dist_M(z^\gamma_{\alpha \beta},z_i)\leq \dist_M(z^\gamma_{\alpha \beta},\vf^{\tau_1}(z^\gamma_{\alpha \beta}))+\dist_M(\vf^{\tau_1}(z^\gamma_{\alpha \beta}),p^i_{jk})+\dist_M(p^i_{jk},z^i_{jk})+\dist_M(z^i_{jk},z_i)
\leq r_0 S_0+2R_0+r_0 S_0+r_0<\mathfrak d$.
\item $\vf^{\tau_1}(z^\gamma_{\alpha\beta})\in B_{\mathfrak d}(z_i)$:
$\dist_M(\vf^{\tau_1}(z^\gamma_{\alpha\beta}),z_i)\leq \dist_M(\vf^{\tau_1}(z^\gamma_{\alpha\beta}),p^i_{jk})+\dist_M(p^i_{jk},z^i_{jk})+\dist_M(z^i_{jk},z_i)<2R_0+r_0S_0+r_0<\mathfrak d$.
Similarly, $\vf^{\tau_2}(z^\gamma_{\alpha\beta})\in B_{\mathfrak d}(z_i)$.
\end{enumerate}

By Lemma \ref{LemmaFB_1}, the flow box coordinates ${\mathfrak t}_{z_i}(\cdot),{\mathfrak q}_{z_i}(\cdot)$
of $\vf^{\tau_1}(z^\gamma_{\alpha\beta})$, $\vf^{\tau_2}(z^\gamma_{\alpha\beta})$, $z^\gamma_{\alpha\beta}$,
and of every point in ${S_{R_0}(p^i_{jk})}, {S_{R_0}(p^i_{\ell m})}$ are well-defined.

Recall that ${\mathfrak t}_{z_i}, {\mathfrak q}_{z_i}$ have Lipschitz constants less than $\mathfrak L$.
In the set of circumstances we consider
$\dist_M(p^i_{jk},\vf^{\tau_1}(z^\gamma_{\alpha \beta}))\leq 2R_0$
and ${\mathfrak q}_{z_i}(p^i_{jk})=z^i_{jk}$, so
$$
\dist_M(z^i_{jk},{\mathfrak q}_{z_i}(z^\gamma_{\alpha \beta}))=\dist_M({\mathfrak q}_{z_i}(p^i_{jk}),{\mathfrak q}_{z_i}(\vf^{\tau_1}(z^\gamma_{\alpha \beta})))\leq 2 \mathfrak L R_0.
$$
Similarly, $\dist(z^i_{\ell m},{\mathfrak q}_{z_i}(z^\gamma_{\alpha \beta}))\leq 2\mathfrak L R_0$.
It follows that $\dist_M(z^i_{jk},z^i_{\ell m})\leq 4\mathfrak L R_0$. By (\ref{GridEst}),
$\max\{|j-\ell|,|k-m|\}\leq 4\sqrt{2}\mathfrak L R_0/\d_0=
4\sqrt{2}\mathfrak L K_0\rho_0\big/(\rho_0/8\mathfrak L^2)< \kappa_0$.

\medskip
The  claim is proved. We proceed to construct by induction $\theta_{jk}^i\in [-r_0,r_0]$ and
$p^i_{jk}:=\vf^{\theta^i_{jk}}(z^i_{jk})$ such that $\{{S_{R_0}(p^i_{jk})}:1\leq i\leq N, (j,k)\in I\}$ are pairwise disjoint.

\medskip
\noindent
{\sc Basis of induction:\/} {\em $\exists\theta_{jk}^1\in [-r_0,r_0]$ s.t.
$\{{S_{R_0}(p^1_{jk})}\}_{(j,k)\in I}$ are pairwise disjoint.}

\medskip
\noindent
{\em Construction:\/}
Let $\wh{\s}:\{0,1,\ldots,\kappa_0-1\}\x \{0,1,\ldots,\kappa_0-1\}\to \{1,\ldots,\kappa_0^2\}$ be a bijection,
and set $\sigma_{jk}:=\wh{\sigma}\bigl(j\mathrm{\ mod\,}\kappa_0,k\mathrm{\ mod\,}\kappa_0\bigr)$.
This has the effect that
$$
0<\max\{|j-\ell|,|k-m|\}<\kappa_0\Longrightarrow|\sigma_{jk}-\sigma_{\ell m}|\geq 1.
$$
We let $\theta^1_{jk}:=2R_0\mathfrak L\s_{jk}$ and $p^1_{jk}:=\vf^{\theta^1_{jk}}(z^1_{jk})$.
It is easy to check that $0<\theta^1_{jk}<r_0$.
One shows as in the proof of the claim that ${S_{R_0}(p^1_{jk})}\subset B_{\mathfrak d}(z_1)$,
therefore ${\mathfrak t}_{z_1}$ is well-defined on ${S_{R_0}(p^1_{jk})}$.
Since $\Lip({\mathfrak t}_{z_1})\leq \mathfrak L$ and ${\mathfrak t}_{z_1}(p^1_{jk})=\theta^1_{jk}$,
\begin{equation}\label{T-range}
{\mathfrak t}_{z_1}\bigl[{S_{R_0}(p^1_{jk})}\bigr]\subset \bigl(\theta^1_{jk}-\mathfrak LR_0, \theta^1_{jk}+\mathfrak LR_0\bigr).
\end{equation}

Now suppose $(j,k)\neq (\ell,m)$. If $\max\{|j-\ell|,|k-m|\}\geq \kappa_0$,
then ${S_{R_0}(p^1_{jk})}\cap {S_{R_0}(p^1_{\ell m})}=\emptyset$, because of the claim.
If $\max\{|j-\ell|,|k-m|\}<\kappa_0$, then
$|\theta_{jk}^1-\theta_{\ell m}^1|\geq 2R_0\mathfrak L$. By (\ref{T-range}),
${\mathfrak t}_{z_1}\bigl[{S_{R_0}(p^1_{jk})}\bigr]\cap {\mathfrak t}_{z_1}\bigl[{S_{R_0}(p^1_{\ell m})}\bigr]=\emptyset$,
and again ${S_{R_0}(p^1_{jk})}\cap {S_{R_0}(p^1_{\ell m})}=\emptyset$.

\medskip
\noindent
{\sc Induction step:\/} {\em  If  $\exists\theta^i_{jk}\in[-r_0,r_0]$ s.t.
$\{{S_{R_0}(p^{i}_{jk})}:1\leq i\leq n,(j,k)\in I\}$ are pairwise disjoint, then $\exists\theta^i_{jk}\in[-r_0,r_0]$
s.t. $\{{S_{R_0}(p^{i}_{jk})}:1\leq i\leq n+1,(j,k)\in I\}$ are pairwise disjoint.}

 \medskip
 Fix $(j,k)\in I$. We divide $\{p^i_{\ell,m}: 1\leq i\leq n, (\ell,m)\in I\}$ into two groups:
 \begin{enumerate}[$\circ$]
 \item {\em ``Dangerous"} (for $p^{n+1}_{jk}$): $\exists\theta\in[-r_0,r_0]$ s.t.
 ${S_{R_0}(p^{i}_{\ell m})}\cap {S_{R_0}(\vf^\theta(z^{n+1}_{jk}))}\neq \emptyset$;
 \item {\em ``Safe"} (for $p^{n+1}_{jk}$): not dangerous.
 \end{enumerate}
Here we employ the terminology ``safe'' when the induction step follows directly from the
basis of induction, and ``dangerous'' otherwise. Indeed, no matter how we define $\theta^{n+1}_{jk}$,
${S_{R_0}(p^{n+1}_{jk})}\cap {S_{R_0}(p^{i}_{\ell m})}=\emptyset$ for all safe $p^i_{\ell m}$.
But the dangerous $p^i_{\ell m}$ will introduce constraints on the possible values of $\theta^{n+1}_{jk}$.

By the claim, if $p^i_{\ell_1,m_1}, p^i_{\ell_2,m_2}$ are dangerous for $p^{n+1}_{jk}$, then
$|\ell_1-\ell_2|, |m_1-m_2|<\kappa_0$. It follows that there are at most $4\kappa_0^2N$
dangerous points for a given $p^{n+1}_{jk}$.

Let $W_{n+1}(p^i_{\ell m}):={\mathfrak t}_{z_{n+1}}\bigl[{S_{R_0}(p^{i}_{\ell m})}\bigr]$.
Since $\mathrm{Lip}({\mathfrak t}_{z_{n+1}})\leq \mathfrak L$, $W_{n+1}(p^i_{\ell m})$
is a closed interval of length less than $w_0:=2\mathfrak L R_0$.

If $p^i_{\ell m}$ is dangerous for $p^{n+1}_{jk}$, then we call $W_{n+1}(p^i_{\ell m})$ a
``dangerous interval" for $p^{n+1}_{jk}$. Let $W^{n+1}_{jk}$ denote the union of all dangerous
intervals for $p^{n+1}_{jk}$, and define
$$
T^{n+1}(j,k):=\bigcup_{|j'-j|, |k'-k|<\kappa_0}W^{n+1}_{j' k'}.
$$
This is a union of no more than $16\kappa_0^4 N$ intervals of length less than $w_0$ each.

Cut $[-r_0,r_0]$ into four equal ``quarters": $Q_1:=[-r_0,-\frac{r_0}{2}],\ldots,Q_4:=[\frac{r_0}{2},r_0]$.
If we subtract $n<L/w$ intervals of length less than $w$ from an interval of length $L$,
then the remainder must contain at least one interval of length $(L-nw)/(n+1)$.
It follows that for every $s=1,\ldots,4$,
$$
Q_s\setminus T^{n+1}(\kappa_0\lfloor \tfrac{j}{\kappa_0}\rfloor,\kappa_0\lfloor \tfrac{k}{\kappa_0}\rfloor)\supset\textrm{ an interval of length }
\frac{r_0}{32\kappa_0^4 N+2}-w_0\gg10\kappa_0^2 w_0.
$$
Let $\tau^{n+1}(j,k)$ denote the
\begin{enumerate}[$\circ$]
\item center of such an interval in $Q_1$, when $(\lfloor\frac{j}{\kappa_0}\rfloor,\lfloor\frac{k}{\kappa_0}\rfloor)=(0,0)\mod 2$,
\item center of such an interval in $Q_2$, when $(\lfloor\frac{j}{\kappa_0}\rfloor,\lfloor\frac{k}{\kappa_0}\rfloor)=(1,0)\mod 2$,
\item center of such an interval in $Q_3$, when $(\lfloor\frac{j}{\kappa_0}\rfloor,\lfloor\frac{k}{\kappa_0}\rfloor)=(0,1)\mod 2$,
\item center of such an interval in $Q_4$, when $(\lfloor\frac{j}{\kappa_0}\rfloor,\lfloor\frac{k}{\kappa_0}\rfloor)=(1,1)\mod 2$.
\end{enumerate}

\medskip
Define $\theta^{n+1}_{jk}:=
\tau^{n+1}\bigl(\kappa_0\lfloor\frac{j}{\kappa_0}\rfloor,\kappa_0\lfloor\frac{k}{\kappa_0}\rfloor\bigr)+3w_0\sigma_{jk}$.
This belongs to $[-r_0,r_0]$, because $\tau^{n+1}\bigl(\kappa_0\lfloor\frac{j}{\kappa_0}\rfloor, \kappa_0\lfloor\frac{k}{\kappa_0}\rfloor\bigr)$ is the center of an interval in $Q_s$ of radius at least $5\kappa_0^2 w_0>3w_0\sigma_{jk}$,
and $Q_s\subset [-r_0,r_0]$.
Moreover, since $\mathrm{Lip}(\mathfrak t_{z_{n+1}})\leq \mathfrak L$ and $w_0=2\mathfrak L R_0$,
we have $\mathfrak t_{z_{n+1}}[S_{R_0}(p^{n+1}_{jk})]\subset \bigl[\theta^{n+1}_{jk}-w_0,\theta^{n+1}_{jk}+w_0\bigr]$,
which by the definition of $\tau^{n+1}(j,k)$, lies outside
$T^{n+1}(\kappa_0\lfloor\tfrac{j}{\kappa_0}\rfloor,\kappa_0\lfloor \tfrac{k}{\kappa_0}\rfloor)$. Thus
\begin{equation}\label{t-coordinate-inclusion}
\mathfrak t_{z_{n+1}}[S_{R_0}(p^{n+1}_{jk})]\subset
[-r_0,r_0]\setminus T^{n+1}(\kappa_0\lfloor\tfrac{j}{\kappa_0}\rfloor,\kappa_0\lfloor \tfrac{k}{\kappa_0}\rfloor).
\end{equation}
We use this to show that  ${S_{R_0}(p^{n+1}_{jk})}\cap {S_{R_0}(p^{i}_{\ell m})}=\emptyset$
for $(\ell,m)\in I$, $i\leq n$. If $p^i_{\ell m}$ is safe for $p^{n+1}_{jk}$, then there is nothing to prove.
If it is dangerous, ${\mathfrak t}_{z_{n+1}}\bigl[{S_{R_0}(p^{i}_{\ell m})}\bigr]\subset W^{n+1}_{jk}\subset T^{n+1}\bigl(\kappa_0\lfloor\frac{j}{\kappa_0}\rfloor, \kappa_0\lfloor\frac{k}{\kappa_0}\rfloor\bigr)$.
By (\ref{t-coordinate-inclusion}), ${S_{R_0}(p^{n+1}_{jk})}\cap {S_{R_0}(p^{i}_{\ell m})}=\emptyset$.

Next we show that  ${S_{R_0}(p^{n+1}_{jk})}$ is disjoint from every
${S_{R_0}(p^{n+1}_{\ell m})}$ s.t. $(\ell,m)\neq (j,k)$. There are three cases:
 \begin{enumerate}[$\circ$]
\item $\max\{|j-\ell|, |k-m|\}\geq \kappa_0$: use the claim.

\medskip
\item $0<\max\{|j-\ell|, |k-m|\}< \kappa_0$ and $(\lfloor\tfrac{j}{\kappa_0}\rfloor, \lfloor\tfrac{k}{\kappa_0}\rfloor)=(\lfloor\tfrac{\ell}{\kappa_0}\rfloor,\lfloor\tfrac{m}{\kappa_0}\rfloor)$: in this case
$|\theta^{n+1}_{jk}-\theta^{n+1}_{\ell m}|\geq 3w_0$. Since
${\mathfrak t}_{z_{n+1}}\bigl[{S_{R_0}(p^{n+1}_{jk})}\bigr]\subset [\theta^{n+1}_{jk}-w_0,\theta^{n+1}_{jk}+w_0]$
and
${\mathfrak t}_{z_{n+1}}\bigl[{S_{R_0}(p^{n+1}_{\ell m})}\bigr]\subset [\theta^{n+1}_{\ell m}-w_0,\theta^{n+1}_{\ell m}+w_0]$,
${\mathfrak t}_{z_{n+1}}\bigl[{S_{R_0}(p^{n+1}_{jk})}\bigr]\cap{\mathfrak t}_{z_{n+1}}\bigl[{S_{R_0}(p^{n+1}_{\ell m})}\bigr]=\emptyset$. So ${S_{R_0}(p^{n+1}_{jk})}\cap {S_{R_0}(p^{n+1}_{\ell m})}=\emptyset$.

\medskip
\item $0<\max\{|j-\ell|, |k-m|\}< \kappa_0$ and $(\lfloor\tfrac{j}{\kappa_0}\rfloor, \lfloor\tfrac{k}{\kappa_0}\rfloor)\neq (\lfloor\tfrac{\ell}{\kappa_0}\rfloor,\lfloor\tfrac{m}{\kappa_0}\rfloor)$: in this case
$$
\max\bigl\{\bigl|\lfloor\tfrac{j}{\kappa_0}\rfloor-\lfloor\tfrac{\ell}{\kappa_0}\rfloor\bigr|,
\bigl|\lfloor\tfrac{k}{\kappa_0}\rfloor-\lfloor\tfrac{m}{\kappa_0}\rfloor\bigr|\bigr\}=1,
$$
so $\tau^{n+1}\bigl(\kappa_0\lfloor\frac{j}{\kappa_0}\rfloor, \kappa_0\lfloor\frac{k}{\kappa_0}\rfloor\bigr)$,
$\tau^{n+1}\bigl(\kappa_0\lfloor\frac{\ell}{\kappa_0}\rfloor, \kappa_0\lfloor\frac{m}{\kappa_0}\rfloor\bigr)$
fall in different $Q_s$. Necessarily
${\mathfrak t}_{z_{n+1}}\bigl[{S_{R_0}(p^{n+1}_{jk})}\bigr]\cap {\mathfrak t}_{z_{n+1}}\bigl[{S_{R_0}(p^{n+1}_{\ell m})}\bigr]=\emptyset$, so ${S_{R_0}(p^{n+1}_{jk})}\cap {S_{R_0}(p^{n+1}_{\ell m})}=\emptyset$.
\end{enumerate}
This concludes the inductive step, and the construction of $\theta^i_{jk}$.

\medskip
\noindent
{\sc Completion of the proof:\/}
{\em For every $r\in[\rho_0,R_0]$, $\Lambda_r:=\biguplus_{i=1}^N\biguplus_{(j,k)\in I}S_{r}(p^i_{jk})$
is a standard Poincar\'e section with roof function bounded above by $h_0$.}

\medskip
We saw that the union is disjoint for $r=R_0$, therefore it is disjoint for all $r\leq R_0$.
We will show that the union is a Poincar\'e section with roof function bounded by $h_0$
for $r=\rho_0$, and then this statement will follow for all $r\geq \rho_0$.

Given  $p\in M$, we must find  $0<R<h_0$ s.t. $\vf^R(p)\in\Lambda_{\rho_0}$.
Since $M\subset \bigcup_{i=1}^N {\rm FB}_{r_0}(z_i)$, $\exists i$ s.t. $\vf^{4r_0}(p)\in {\rm FB}_{r_0}(z_i)$.
Therefore $\vf^{4r_0}(p)=\vf^t(z)$ for some $z\in S_{r_0}(z_i)$, $|t|<r_0$, whence $\vf^{4r_0-t}(p)\in S_{r_0}(z_i)$.

Write $\vf^{4r_0-t}(p)=J_i(x,y)$ for some $(x,y)$ s.t. $x^2+y^2\leq r_0^2$, and choose $(j,k)\in I$
s.t. $|x-j\d_0|, |y-k\d_0|<\d_0$. Since $J_i$ is $\sqrt{2}$--bi-Lipschitz, $\dist_M(\vf^{4r_0-t}(p),z^i_{jk})<2\d_0$.
It follows that $\dist_M(\vf^{4r_0-t}(p),p^i_{jk})<2\d_0+r_0 S_0<\mathfrak d$. This places
$\vf^{4r_0-t}(p)$ inside ${\rm FB}_{\mathfrak r_f}(p^i_{jk})$. Let $\dist_S$ denote the intrinsic
distance on $S_{\mathfrak r_s}(p^i_{jk})$. We have $\dist_S\leq 2\dist_M$ (see Lemma \ref{LemmaTD}),
therefore, since $p^i_{jk}={\mathfrak q}_{p^i_{jk}}(z^i_{jk})$,
\begin{align*}
&\dist_S({\mathfrak q}_{p^i_{jk}}(\vf^{4r_0-t}(p)),p^i_{jk})=
\dist_S({\mathfrak q}_{p^i_{jk}}(\vf^{4r_0-t}(p)),{\mathfrak q}_{p^i_{jk}}(z^i_{jk}))\leq\\
&\leq 2\dist_M({\mathfrak q}_{p^i_{jk}}(\vf^{4r_0-t}(p)),{\mathfrak q}_{p^i_{jk}}(z^i_{jk}))
\leq 2\mathfrak L \dist_M(\vf^{4r_0-t}(p),z^i_{jk})<4\mathfrak L \d_0<\rho_0.
\end{align*}
Thus $\vf^R(p)\in S_{\rho_0}(p^i_{jk})\subset\Lambda_{\rho_0}$ for
$R:=4r_0-t-{\mathfrak t}_{p^i_{jk}}[\vf^{4r_0-t}(p)]$.

Now $|{\mathfrak t}_{p^i_{jk}}[\vf^{4r_0-t}(p)]|\leq 2r_0$, because $|\theta^i_{jk}|\leq r_0$ and
$|{\mathfrak t}_{p^i_{jk}}[\vf^{4r_0-t}(p)]+\theta^i_{jk}|=
|{\mathfrak t}_{p^i_{jk}}[\vf^{4r_0-t}(p)]-{\mathfrak t}_{p^i_{jk}}[z^i_{jk}]|\leq \mathfrak L \dist_M(\vf^{4r_0-t}(p),z^i_{jk})
<2\mathfrak L\d_0<r_0$.
Also $|t|<r_0$. So $r_0<R<7r_0$. Since $r_0<\frac{1}{9}h_0$, we conclude that $0<R<h_0$.

\subsection*{Proof of Theorem \ref{Thm_Symbolic_Dynamics_For_f}(5)}

The proof is motivated by \cite{Bowen-Regional-Conference}.
Say that $R,R'\in\mathfs R$ are {\em affiliated}, if there are $Z,Z'\in\mathfs Z$ s.t.
$R\subset Z$, $R'\subset Z'$, and $Z\cap Z'\neq\emptyset$. Let $N(R,S):=N(R)N(S)$, where
$$
N(R):=\#\{(R',v')\in\mathfs R\x\mathfs A:\textrm{$R'$ is affiliated to $R$ and $Z(v')\supset R'$}\}.
$$
This is finite, because of the local finiteness of $\mathfs Z$.
Let $x=\pi(\un{R})$ where $R_i=R$ for infinitely many $i<0$ and $R_i=S$ for infinitely many $i>0$.
Let $N:=N(R,S)$, and suppose by way of contradiction that $x$ has $N+1$ different pre-images
$\un{R}^{(0)},\ldots,\un{R}^{(N)}\in\Sigma^{\#}(\widehat{\mathfs G})$, with $\un{R}^{(0)}=\un{R}$.
Write $\un{R}^{(j)}=\{R^{(j)}_k\}_{k\in\Z}$. By Lemma \ref{Lem_R_Z} there are $\un{v}^{(j)}\in \Sigma(\mathfs G)$
s.t. for every $n$,
$$
_{-n}[R^{(j)}_{-n},\ldots,R^{(j)}_n]\subset Z_{-n}(v^{(j)}_{-n},\ldots,v^{(j)}_n)\textrm{ and } R^{(j)}_n\subset Z(v^{(j)}_n).
$$
For every $j$, $\un{v}^{(j)}\in \Sigma^\#({\mathfs G})$, because $\un{R}^{(j)}\in\Sigma^\#(\widehat{\mathfs G})$
and $\mathfs Z$ is locally finite. It follows that $\pi(\un{v}^{(j)})\in Z_{-n}(v^{(j)}_{-n},\ldots,v^{(j)}_n)$
for all $n$.\footnote{At this point the proof given in \cite{Sarig-JAMS} has a mistake.
There it is claimed that $\pi(\un{v}^{(j)})\in Z_{-n}(v_{-n}^{(j)},\ldots,v^{(j)}_n)$ without making
the assumption that $\un{R}^{(j)}\in\Sigma^\#(\wh{\mathfs G})$. }

Since $x=\wh{\pi}(\un{R}^{(j)})\in\ov{_{-n}[R_{-n},\ldots,R_n]}\subset \ov{Z_{-n}(v^{(j)}_{-n},\ldots,v^{(j)}_{n})}$, and since the diameter of $Z_{-n}(v^{(j)}_{-n},\ldots,v^{(j)}_{n})$ tends to zero as $n\to\infty$ by the H\"older continuity of $\pi$,
$\pi(\un{v}^{(j)})=x$.
Thus $Z(v^{(0)}_i),\ldots,Z(v^{(N)}_i)$ all intersect (they  contain $f^i(x)=\pi[\sigma^i(\un{v}^{(j)})]$). This and the inclusion
$R^{(j)}_i\subset Z(v^{(j)}_i)$ give that $R^{(0)}_i,\ldots, R^{(N)}_i$ are affiliated for all $i$.

In particular, if $k,\ell>0$ satisfy $R^{(0)}_{-k}=R$ and $R^{(0)}_{\ell}=S$ (there are infinitely many such $k,\ell$), then there are at most $N=N(R)N(S)$ possibilities for the quadruple $(R^{(j)}_{-k},Z(v^{(j)}_{-k});R^{(j)}_\ell,Z(v^{(j)}_\ell))$, $j=0,\ldots,N$. By the pigeonhole principle, there are $0\leq j_1,j_2\leq N$ s.t. $j_1\neq j_2$ and
$$
(R^{(j_1)}_{-k},v^{(j_1)}_{-k})=(R^{(j_2)}_{-k},v^{(j_2)}_{-k})\textrm{ and }(R^{(j_1)}_\ell,v^{(j_1)}_\ell)=(R^{(j_2)}_\ell,v^{(j_2)}_\ell).
$$
We can also guarantee that
$$
(R^{(j_1)}_{-k},\ldots,R^{(j_1)}_{\ell})\neq (R^{(j_2)}_{-k},\ldots,R^{(j_2)}_{\ell}).
$$
To do this fix in advance some $m$ s.t. $(R^{(j)}_{-m},\ldots,R^{(j)}_{m})\ (j=0,\ldots,N)$ are all different, and work with $k,\ell>m$.

Now let $\un{A}:=\un{R}^{(j_1)}$, $\un{B}:=\un{R}^{(j_2)}$, $\un{a}:=\un{v}^{(j_1)}$, $\un{b}:=\un{v}^{(j_2)}$.  Write
$A_{-k}=B_{-k}=:B$, $A_{\ell}=B_{\ell}=:A$, $a_{-k}=b_{-k}=:b$, and $a_\ell=b_\ell=:a$. Choose
$$
x_A\in {_{-k}[}A_{-k},\ldots,A_\ell]\textrm{ and }x_B\in {_{-k}[}B_{-k},\ldots,B_\ell]
$$
and two points $z_A, z_B$ by the equations
\begin{align*}
f^{-k}(z_A)&:=[f^{-k}(x_B),f^{-k}(x_A)]\in W^u(f^{-k}(x_B),B)\cap W^s(f^{-k}(x_A),B)\\
f^{\ell}(z_B)&:=[f^{\ell}(x_B),f^{\ell}(x_A)]\in W^u(f^{\ell}(x_B),A)\cap W^s(f^{\ell}(x_A),A).
\end{align*}
This makes sense, because $f^{-k}(x_A), f^{-k}(x_B)\in B$ and $f^{\ell}(x_A), f^{\ell}(x_B)\in A$.
One checks using the Markov property of $\mathfs R$ that $z_A\in  {_{-k}[}A_{-k},\ldots,A_\ell]$,
and $z_B\in {_{-k}[}B_{-k},\ldots,B_\ell]$. Since $(A_{-k},\ldots,A_\ell)\neq (B_{-k},\ldots,B_\ell)$
and the elements of $\mathfs R$ are pairwise disjoint, $z_A\neq z_B$.
We will obtain the contradiction we are after by showing that $z_A=z_B$.

Since $f^{\ell}(z_A)\in A_{\ell}=A\subset Z(a)$ and $f^{-k}(z_B)\in B_{-k}=B\subset  Z(b)$,
there are $\un{\alpha},\un{\beta}\in\Sigma^\#(\mathfs G)$ s.t. $z_A=\pi(\un{\alpha}), z_B=\pi(\un{\beta})$,
$\alpha_\ell=a$, $\beta_{-k}=b$. Let $\un{c}=\{c_i\}_{i\in\Z}$ where $c_i=\beta_i$ for $i\leq -k$,
$c_i=a_i$ for $-k<i<\ell$, and $c_i=\alpha_i$ for $i\geq \ell$.
%
%
%
This belongs to $\Sigma^\#(\mathfs G)$, because $\un{\alpha}, \un{\beta}\in \Sigma^\#(\mathfs G)$
and $\beta_{-k}=b=a_{-k}$ and $a_\ell=a=\alpha_\ell$. We will show that $z_A=\pi(\un{c})=z_B$.
Write $c_i=\Psi_{x_i}^{p^u_i,p^s_i}$ $(i\in\Z)$. By the definition of $z_A,z_B$ and the Markov property,
$f^{-k}(z_A),f^{-k}(z_B)$ both belong to $W^u(f^{-k}(x_B),B)$, thus
$$
W^u(f^{-k}(z_A),B)=W^u(f^{-k}(z_B),B)=W^u(\pi(\s^{-k}\un{\beta}),B)\subset V^u[(c_i)_{i\leq -k}].
$$
It follows that $f^{i}(z_A),f^{i}(z_B)\in \Psi_{x_i}([-Q_\epsilon(x_i),Q_\epsilon(x_i)]^2)$ for all $i\leq -k$.

Similarly, $f^\ell(z_A),f^\ell(z_B)$ both belong to $W^s(f^\ell(x_A),A)$, whence
$$
W^s(f^{\ell}(z_A),A)=W^s(f^{\ell}(z_B),A)=W^s(\pi(\s^{\ell}\un{\alpha}),A)\subset V^s[(c_i)_{i\geq \ell}].
$$
It follows that $f^{i}(z_A),f^{i}(z_B)\in \Psi_{x_i}([-Q_\epsilon(x_i),Q_\epsilon(x_i)]^2)$ for all $i\geq \ell$.
For $-k< i< \ell$, $f^i(z_A), f^i(z_B)\in A_i\cup B_i\subset Z(a_i)\cup Z(b_i)$.
The sets $Z(a_i), Z(b_i)$ intersect, because as we saw above:
\begin{enumerate}[$\circ$]
\item $x=\pi(\un{a})\in Z_{-k}(a_{-k},\ldots,a_\ell)$, whence $f^i(x)\in Z(a_i)$.
\item $x=\pi(\un{b})\in Z_{-k}(b_{-k},\ldots,b_\ell)$, whence $f^i(x)\in Z(b_i)$.
\end{enumerate}
By the overlapping charts property of $\mathfs Z$ (see \S 5) and since $a_i=c_i$ for $-k< i< \ell$,
$$
Z(a_i)\cup Z(b_i)\subset \Psi_{x_i}([-Q_{\epsilon}(x_i),Q_{\epsilon}(x_i)]^2)\ \text{ for }i=-k+1,\ldots,\ell-1.
$$
In summary, $f^i(z_A), f^i(z_B)\in\Psi_{x_i}([-Q_{\epsilon}(x_i),Q_{\epsilon}(x_i)]^2)$ for all $i\in\Z$.
As shown in the proof of the shadowing lemma (Thm. \ref{Thm_Shadowing}),
$\un{c}$ shadows both $z_A$ and $z_B$, whence $z_A=z_B$.\hfill$\Box$

\medskip
\noindent
{\bf Remark.} We take this opportunity to correct a mistake in \cite{Sarig-JAMS}.
Theorem 12.8 in \cite{Sarig-JAMS} (the analogue of the statement we just proved) is
stated wrongly as a bound for the number of all pre-images of $x\in\wh{\pi}[\Sigma^\#(\wh{\mathfs G})]$.
But what is actually proved there (and all that is needed for the remainder of the paper)
is just a bound on the number of pre-images which belong to $\Sigma^\#(\wh{\mathfs G})$
(denoted there by $\Sigma_\chi^\#$). Thus the statements of Theorems 1.3 and 1.4 in \cite{Sarig-JAMS}
should be read as bounds on the number of pre-images in $\Sigma^\#_\chi$ (denoted here by
$\Sigma^\#(\wh{\mathfs G})$), and not as bounds on the number of pre-images in $\Sigma_\chi$
(denoted here by $\Sigma(\wh{\mathfs G})$). The other results or proofs in \cite{Sarig-JAMS}
are not affected by these changes, since
$\Sigma_\chi\setminus\Sigma^\#_\chi$ does not contain any periodic orbits, and because
$\Sigma_\chi\setminus\Sigma^\#_\chi$ has zero measure for every shift invariant probability
measure (Poincar\'e recurrence theorem).

\subsection*{Proof of Lemma \ref{Lemma-BW}}

Let $\psi:\Sigma_1\to\Sigma_1$ be the {\em constant} suspension flow, then:
\begin{enumerate}[$\circ$]
\item
For every horizontal segment $[z,w]_h$,
$|\tau|<1\Longrightarrow \left|\frac{\ell([\psi^\tau(z),\psi^\tau(w)]_h)}{\ell([z,w]_h)}-1\right|\leq 2e^2|\tau|$.
This uses the trivial bound $ d(\un{x},\un{y})/d(\sigma^k(\un{x}),\sigma^k(\un{y}))\in [e^{-1},e]$
for $|k|\leq 1$ and the metric $d(\un{x},\un{y}):=\exp[-\min\{|n|:x_n\neq y_n\}]$.
\item For every vertical segment $[z,w]_v$, $\ell([\psi^{\tau}(z),\psi^{\tau}(w)]_v)=\ell([z,w]_v)$ for all $\tau$.
\item Thus for all $z,w\in \Sigma_1$,
$|\tau|<1\Longrightarrow (1+2e^2|\tau|)^{-1}\leq \frac{d_1(\psi^\tau(z),\psi^\tau(w))}{d_1(z,w)}\leq (1+2e^2|\tau|)$.
\end{enumerate}

\noindent
{\sc Claim:\/} {\em $d_r$ is a metric on $\Sigma_r$.}

\medskip
\noindent
{\em Proof.\/} It is enough to show that $d_1$ is a metric. Symmetry and the triangle inequality are obvious;
we show that $d_1(z,w)=0\Rightarrow z=w$.
Let $z=(\un{x},t)$, $w=(\un{y},s)$, $\tau:=\frac{1}{2}-t$. If $d_1(z,w)=0$, then $d_1(\psi^\tau(z),\psi^\tau(w))=0$.
Let $\gamma=(z_0,z_1,\ldots,z_n)$ be a basic path from $\psi^\tau(z)$ to $\psi^\tau(w)$ with length
less than $\epsilon$, with $\epsilon<\frac{1}{3}$ fixed but arbitrarily small.
Write $z_i=(\un{x}_i,t_i)$, then $\psi^\tau(z)=(\un{x}_0,t_0)$ and $\psi^\tau(w)=(\un{x}_n,t_n)$.

Since the lengths of the vertical segments of $\wt{\gamma}$ add up to less than $\epsilon$ and $t_0=\frac{1}{2}$,
$\wt{\gamma}$ does not leave $\Sigma\x [\frac{1}{2}-\epsilon,\frac{1}{2}+\epsilon]$.
It follows that $|t_n-t_0|<\epsilon$. Since $\epsilon$ was arbitrary, $t_n=t_0$, and $\psi^\tau(z)$, $\psi^\tau(w)$
have the same second coordinate.

Since $\wt{\gamma}$ does not leave $\Sigma\x [\frac{1}{2}-\epsilon,\frac{1}{2}+\epsilon]$,
it does not cross $\Sigma\x\{0\}$. Writing a list of the horizontal segments
$[(\un{x}_{i_k},t_{i_k}),(\un{x}_{i_k+1},t_{i_k+1})]_h$, we find that $\un{x}_{i_k+1}=\un{x}_{i_{k+1}}$.
By the triangle inequality
$\epsilon>d_1(\psi^\tau(\un{x},t),\psi^\tau(\un{y},t))\geq
e^{-1}\sum d(\un{x}_{i_k},\un{x}_{i_{k+1}})\geq e^{-1}d(\un{x}_{0},\un{x}_{n})$.
Since $\epsilon$ is arbitrary, $\un{x}_0=\un{x}_n$, and $\psi^\tau(z)$, $\psi^\tau(w)$ have the same first coordinate.
Thus $\psi^\tau(z)=\psi^\tau(w)$, whence $z=w$.

\medskip
\noindent
{\sc Part (1):} {\em $d_r((\un{x},t),(\un{y},s))\leq \const[d(\un{x},\un{y})^\alpha+|t-s|]$,
where $\alpha$ denotes the H\"older exponent of $r$.}

\medskip
\noindent
{\em Proof.\/} $d_r((\un{x},t),(\un{y},s))\equiv d_1((\un{x},\frac{t}{r(\un{x})}),(\un{y},\frac{s}{r(\un{y})}))$.
The basic path $(\un{x},\frac{t}{r(\un{x})})$, $(\un{x},\frac{s}{r(\un{y})})$, $(\un{y},\frac{s}{r(\un{y})})$
shows that
$d_1((\un{x},\frac{t}{r(\un{x})}),(\un{y},\frac{s}{r(\un{y})}))
\leq\bigl|\frac{t}{r(\un{x})}-\frac{s}{r(\un{y})}\bigr|+ed(\un{x},\un{y})
\leq\frac{1}{\inf(r)}\bigl[|t-s|+\Hol_\alpha(r)d(\un{x},\un{y})^\alpha\bigr]+e d(\un{x},\un{y})
\leq\const[d(\un{x},\un{y})^\alpha+|t-s|]$.

\medskip
\noindent
{\sc Part (2):} {\em Let $\alpha$ denote the H\"older exponent of $r$. There is a constant $C_2$ which only depends on $r$ s.t. for all $z=(\un{x},t)$, $w=(\un{y},s)$ in $\Sigma_r$:
\begin{enumerate}[{\rm (a)}]
\item If $\bigl|\frac{t}{r(\un{x})}-\frac{s}{r(\un{y})}\bigr|\leq \frac{1}{2}$, then $d(\un{x},\un{y})\leq C_2 d_r(z,w)$ and $|s-t|\leq C_2 d_r(z,w)^\alpha$.
\item If $\frac{t}{r(\un{x})}-\frac{s}{r(\un{y})}>\frac{1}{2}$, then $d(\sigma(\un{x}),\un{y})\leq C_2 d_r(z,w)$ and $|t-r(x)|,s\leq C_2 d_r(z,w)$.
\end{enumerate}}

\noindent
{\em Proof.\/}
These estimates are trivial when $d_r(z,w)$ is bounded away from zero, so it is enough to prove
part (2) for $z,w$ s.t. $d_r(z,w)<\epsilon_0$, with $\epsilon_0$ a positive constant that will be chosen later.

Suppose $\bigl|\frac{t}{r(\un{x})}-\frac{s}{r(\un{y})}\bigr|<\frac{1}{2}$
and let $\tau:=\frac{1}{2}-\frac{t}{r(\un{x})}$ (a number in $(-\frac{1}{2},\frac{1}{2}]$), then
\begin{align*}
&\,d_r(z,w)=d_1(\vartheta_r(z),\vartheta_r(w))\geq (1+2e^2|\tau|)^{-1}d_1(\psi^\tau[\vartheta_r(z)],\psi^\tau[\vartheta_r(w)])\\
&\geq(1+2e^2)^{-1}d_1((\un{x},\tfrac{1}{2}),(\un{y},\tfrac{1}{2}+\delta)),\textrm{ where }\delta:=\tfrac{s}{r(\un{y})}-\tfrac{t}{r(\un{x})}.
\end{align*}
Notice that $(\un{y},\frac{1}{2}+\delta)\in\Sigma_1$, because $|\delta|<\frac{1}{2}$.

Suppose $\epsilon_0(1+2e^2)<\frac{1}{4}$, then  $d_1((\un{x},\tfrac{1}{2}),(\un{y},\tfrac{1}{2}+\delta))<\frac{1}{4}$. The basic paths whose lengths approximate $d_1((\un{x},\tfrac{1}{2}),(\un{y},\tfrac{1}{2}+\delta))$ are not long enough to leave $\Sigma\x[\frac{1}{4},\frac{3}{4}]$,   and  they cannot cross $\Sigma\x\{0\}$. For such paths the lengths of the vertical  segments add up to at least $\delta$, and the lengths of the horizontal segments add up to at least $e^{-1}d(\un{x},\un{y})$. Since $d_1((\un{x},\tfrac{1}{2}),(\un{y},\tfrac{1}{2}+\delta))\leq (1+2e^2)d_r(z,w)$,
$$
d(\un{x},\un{y})\leq e(1+2e^2)d_r(z,w)\textrm{ and } \left|\delta\right|\leq (1+2e^2)d_r(z,w).
$$
In particular, $d(\un{x},\un{y})\leq \const d_r(z,w)$, and
$|s-t|=\bigl|r(\un{y})\tfrac{s}{r(\un{y})}-r(\un{x})\tfrac{t}{r(\un{x})}\bigr|\leq \sup(r)|\delta|+|r(\un{y})-r(\un{x})|
\leq (1+2e^2)\sup(r)d_r(z,w)+\Hol_\alpha(r)d(\un{x},\un{y})^\alpha\leq \const d_r(z,w)^\alpha,$
where the last inequality uses our estimate for $d(\un{x},\un{y})$ and the finite diameter of   $d_r$.
This proves part (a) when $\bigl|\frac{t}{r(\un{x})}-\frac{s}{r(\un{y})}\bigr|<\frac{1}{2}$.
If $\frac{t}{r(\un{x})}-\frac{s}{r(\un{y})}=\frac{1}{2}$, repeat the previous argument with $\tau:=0.49-\frac{t}{r(\un{x})}$.

For part (b), suppose $\frac{t}{r(\un{x})}-\frac{s}{r(\un{y})}>\frac{1}{2}$, and let
$\tau:=\frac{r(\un{x})-t}{r(\un{x})}+\frac{1}{2}$. Now $\psi^\tau[\vartheta_r(z)]=(\sigma(\un{x}),\frac{1}{2})$
and $\psi^\tau[\vartheta_r(w)]=(\un{y},\frac{1}{2}+\delta')$, where
$\delta':=1-\bigl(\frac{t}{r(\un{x})}-\frac{s}{r(\un{y})}\bigr)$. As before,
$$
d(\sigma(\un{x}),\un{y})\leq e(1+2e^2)d_r(z,w)\textrm{ and } \left|\delta'\right|\leq (1+2e^2)d_r(z,w).
$$
Using  $s\leq r(\un{y})\delta'$, $|r(\un{x})-t|\leq r(\un{x})\delta'$, we see that $s,|t-r(\un{x})|<(1+2e^2)\sup(r) d_r(z,w)$.

\medskip
\noindent
{\sc Part (3):\/} {\em There are constants $C_3>0, 0<\kappa<1$ which only depend on $r$ s.t.
for all $z,w\in\Sigma_r$ and $|\tau|<1$, $d_r(\sigma_r^{\tau}(z),\sigma_r^{\tau}(w))\leq C_3 d_r(z,w)^\kappa$.}

\medskip
\noindent
{\em Proof.\/}
We will only discuss the case $\tau>0$. The case $\tau<0$ can be handled similarly, or deduced from
the following symmetry: Let $\wh{\Sigma}:=\{\wh{\un{x}}:\un{x}\in\Sigma\}$ where $\wh{x}_i:=x_{-i}$,
and let $\wh{r}(\un{x}):=r(\wh{\sigma\un{x}})$ (a function on $\wh{\Sigma}$). Then
$\Theta(\un{x},t)=(\wh{\sigma\un{x}},r({\un{x}})-t)$ is a bi-Lipschitz map from $\Sigma_r$ to
$\wh{\Sigma}_{\wh{r}}$, and $\Theta\circ\sigma_r^{-\tau}=\sigma_{\wh{r}}^{\tau}\circ\Theta$.
This symmetry reflects the representation of the flow $\sigma_r^{-t}$ with respect to the Poincar\'e
section $\Sigma\x\{0\}$.

We will construct  a constant $C_3'$ s.t. for all $z,w\in\Sigma_r$, if $0<\tau<\frac{1}{2}\inf(r)$, then
$d_r(\sigma_r^{\tau}(z),\sigma_r^{\tau}(w))\leq C_3' d_r(z,w)^\alpha$. Part (3) follows with $\kappa:=\alpha^{N}$,
$C_3:=(C_3')^{\frac{1}{1-\alpha}}$, $N:=\lceil 1/\min\{1,\frac{1}{2}\inf(r)\}\rceil$.
We will also limit ourselves to the case when $C_2 d_r(z,w)<\frac{1}{2}\inf(r)$;
part (3) is trivial when $d_r(z,w)$ is bounded away from zero.

Let $z:=(\un{x},t), w:=(\un{y},s)$. Since $\tau>0$,
$\sigma_r^{\tau}(z)=(\sigma^{m}(\un{x}),\epsilon r(\sigma^{m}(\un{x}))$ and
$\sigma_r^{\tau}(w)=(\sigma^{n}(\un{y}),\eta r(\sigma^{n}(\un{y}))$ where $0\leq \epsilon,\eta<1$
and $m,n\geq 0$. Notice that $m,n\in\{0,1\}$, (because $0<\tau<\frac{1}{2}\inf(r)$, so
$\sigma_r^t(z),\sigma_r^t(w)$ cannot cross $\Sigma\x\{0\}$ twice).

\medskip
\noindent {\bf Case 1:} $\bigl|\frac{t}{r(\un{x})}-\frac{s}{r(\un{y})}\bigr|\leq\frac{1}{2}$ and $m=n$. Then:
\begin{align*}
&\, d_r(\sigma_r^{\tau}(z),\sigma_r^{\tau}(w))=d_1((\sigma^m(\un{x}),\epsilon),(\sigma^m(\un{y}),\eta))\\
& \leq d_1((\sigma^m(\un{x}),\epsilon),(\sigma^m(\un{y}),\epsilon))
+d_1((\sigma^m(\un{y}),\epsilon),(\sigma^m(\un{y}),\eta))\\
& \leq e d(\sigma^m(\un{x}),\sigma^m(\un{y}))+|\epsilon-\eta|\\
&\leq e^2 d(\un{x},\un{y})+|\epsilon-\eta|
\leq e^2 C_2 d_r(z,w)+|\epsilon-\eta|,\textrm{ by part (2)(a).}
\end{align*}
Since $m=n$,
$|\epsilon-\eta|=\bigl|\frac{t+\tau-r_{m}(\un{x})}{r(\sigma^{m}(\un{x}))}-
\frac{s+\tau-r_{m}(\un{y})}{r(\sigma^{m}(\un{y}))}\bigr|\leq
\frac{1}{\inf(r)^2}[I_1+I_2+I_3]$, where:
\begin{enumerate}[$\circ$]
\item $I_1=|tr(\sigma^m(\un{y}))-sr(\sigma^m(\un{x}))|\leq t|r(\sigma^m(\un{x}))-r(\sigma^m(\un{y}))|+
|t-s| r(\sigma^m(\un{x}))\leq\sup(r)[e^\alpha C_2^\alpha\Hol_\alpha(r)+C_2]d_r(z,w)^\alpha$ by part (2)(a).
\item $I_2=\tau|r(\sigma^m(\un{x}))-r(\sigma^m(\un{y}))|\leq e^\alpha C_2^\alpha\inf(r)\Hol_\alpha(r) d_r(z,w)^\alpha$,
because $m\leq 1$.
\item $I_3=|r_m(\un{x})r(\sigma^m(\un{y}))-r_m(\un{y})r(\sigma^m(\un{x}))|\leq |r_m(\un{x})-r_m(\un{y})|r(\sigma^m(\un{y}))+r_m(\un{y})\cdot|r(\sigma^m(\un{x}))-r(\sigma^m(\un{y}))|\leq
\const\Hol_\alpha(r)d_r(z,w)^\alpha$, again because $m\leq 1$.
\end{enumerate}
Thus $|\epsilon-\eta|\le\const d_r(z,w)^\alpha$, where the constant
only depends on $r$. It follows that $d_r(\sigma_r^{\tau}(z),\sigma_r^{\tau}(w))\leq \const d_r(z,w)^\alpha$
where the constant only depends on $r$.

\medskip
\noindent {\bf Case 2:} $\bigl|\frac{t}{r(\un{x})}-\frac{s}{r(\un{y})}\bigr|\leq \frac{1}{2}$ and $m\not=n$.
We can assume that $n=m+1$, thus:
\begin{align*}
&\, d_r(\sigma_r^{\tau}(z),\sigma_r^{\tau}(w))=d_1((\sigma^{m}(\un{x}),\epsilon),(\sigma^{m+1}(\un{y}),\eta))\\
& \leq d_1((\sigma^{m}(\un{x}),\epsilon),(\sigma^m(\un{y}),\epsilon))
+d_1((\sigma^m(\un{y}),\epsilon),(\sigma^{m+1}(\un{y}),\eta))\\
& \leq e d(\sigma^m(\un{x}),\sigma^m(\un{y}))+1-\epsilon+\eta
\leq e^2 C_2 d_r(z,w)+1-\epsilon+\eta,\textrm{ by part (2)(a).}
\end{align*}
In our scenario, $t+\tau-r_{m+1}(\un{x})$ is negative, and $s+\tau-r_{m+1}(\un{y})$ is non-negative.
The distance between these two numbers is bounded by $|t-s|+|r_{m+1}(\un{x})-r_{m+1}(\un{y})|$,
whence by $\const d_r(z,w)^\alpha$. So $|t+\tau-r_{m+1}(\un{x})|,|s+\tau-r_{m+1}(\un{y})|\leq \const d_r(z,w)^\alpha$.
Since $1-\epsilon=\frac{|t+\tau-r_{m+1}(\un{x})|}{r(\sigma^m\un{x})}$,
$\eta=\frac{s+\tau-r_{m+1}(\un{y})}{r(\sigma^{m+1}\un{y})}$, and the denominators are at
least $\inf(r)$, there is a constant which only depends on $r$ s.t. $1-\epsilon,\eta<\const d_r(z,w)^\alpha$.
It follows that $d_r(\sigma_r^{\tau}(z),\sigma_r^{\tau}(w))\leq \const d_r(z,w)^\alpha$.

\medskip
\noindent
{\bf Case 3:} $\frac{t}{r(\un{x})}-\frac{s}{r(\un{y})}>\frac{1}{2}$ and $m=n$. We have:
\begin{align*}
&\,d_r(\sigma_r^\tau(z),\sigma_r^\tau(w))=d_1((\sigma^m(\un{x}),\epsilon),(\sigma^m(\un{y}),\eta))\\
&\leq d_1((\sigma^m(\un{x}),\epsilon),(\sigma^{m+1}(\un{x}),\eta))+d_1((\sigma^{m+1}(\un{x}),\eta),(\sigma^m(\un{y}),\eta))\\
&\leq 1-\epsilon+\eta+e^2 C_2 d_r(z,w),\textrm{ by part (2)(b), and since $m\leq 1$.}
\end{align*}
Because $t+\tau-r_{m+1}(\un{x})<0\leq s+\tau-r_m(\un{y})$, it follows by part (2)(b) that
\begin{align*}
&\, |t+\tau-r_{m+1}(\un{x})|,|s+\tau-r_m(\un{y})|\leq |t-s-r_{m+1}(\un{x})-r_m(\un{y})|\\
&\leq |t-r(\un{x})|+s+|r_m(\sigma(\un{x}))-r_m(\un{y})|\leq 2C_2 d_r(z,w)+\const d_r(z,w)^\alpha.
\end{align*}
As in case 2, this means that $d_r(\sigma_r^{\tau}(z),\sigma_r^{\tau}(w))<\const d_r(z,w)^\alpha$.

\medskip
\noindent
{\bf Case 4:} $\frac{t}{r(\un{x})}-\frac{s}{r(\un{y})}>\frac{1}{2}$ and $m\neq n$.
Recall that $C_2 d_r(z,w),\tau<\frac{1}{2}\inf(r)$. By part (2)(b), $s\leq \frac{1}{2}\inf(r)$,
thus $s+\tau<\inf(r)$. Necessarily $n=0$, $m=1$, $m=n+1$, so:
\begin{align*}
&\, d_r(\sigma_r^\tau(z),\sigma_r^\tau(w))=d_1((\sigma^{n+1}(\un{x}),\epsilon),(\sigma^{n}(\un{y}),\eta))\\
&\leq d_1((\sigma^{n+1}(\un{x}),\epsilon),(\sigma^{n+1}(\un{x}),\eta))+d_1((\sigma^{n+1}(\un{x}),\eta),(\sigma^{n}(\un{y}),\eta))\\
&\leq |\epsilon-\eta|+ed(\sigma(\un{x}),\un{y})\ \ (\because n=0)\\
&\leq |\epsilon-\eta|+eC_2 d_r(z,w),\ \ \textrm{by part (2)(b)}.
\end{align*}
We have
$|\epsilon-\eta|=\bigl|\frac{t+\tau-r(\un{x})}{r(\sigma\un{x})}-\frac{s+\tau}{r(\un{y})}\bigr|\leq \frac{1}{\inf(r)^2}[I_1+I_2]$,
where by part (2)(b):
\begin{enumerate}[$\circ$]
\item $I_1:=|[t-r(\un{x})]r(\un{y})-s r(\sigma\un{x})|\leq 2\sup(r)C_2 d_r(z,w)$.
\item $I_2:=\tau|r(\sigma\un{x})-r(\un{y})|\leq \frac{1}{2}\inf(r)\Hol(r)C_2^\alpha d_r(z,w)^\alpha$.
\end{enumerate}
It follows that $d_r(\sigma_r^\tau(z),\sigma_r^\tau(w))\leq \const d_r(z,w)^\alpha$
where the constant only depends on $r$. This completes the proof of part (3).\hfill$\Box$

\section{Acknowledgements}
The authors  thank Anatole Katok, Fran\c{c}ois Ledrappier and Federico Rodriguez-Hertz for  useful and inspiring  discussions. Special thanks  go to  Dmitry Dolgopyat for showing us a proof that weak mixing $\Rightarrow$ mixing in Theorem \ref{theorem dichotomy for suspension},   Giovanni Forni for suggesting the  application to Reeb flows, and  Edriss Titi for giving us an elementary proof that $C^{1+\epsilon}$ vector fields generate flows with  $C^{1+\epsilon}$ time $t$ maps.

\bibliographystyle{alpha}
\bibliography{Flow9}{}

\end{document}